\documentclass[11pt]{article}

\usepackage{color}

\usepackage{amsmath,amsthm}
\usepackage{amsfonts}
\usepackage{latexsym}
\usepackage{indentfirst}
\usepackage{psfrag,epsf}
\usepackage{pstool}

\usepackage{epsfig, subfigure}
\usepackage{graphicx}
\usepackage{epstopdf}
\usepackage{enumerate}
\usepackage[subnum]{cases}
\usepackage{geometry}
\usepackage{fullpage}
\usepackage{tikz}

\usetikzlibrary{shapes,arrows,calc}
\usepackage{multicol}  

\usepackage{bm}

\newtheorem{theorem}{Theorem} \rm
\newtheorem{lemma}[theorem]{Lemma}

\newtheorem{conjecture}[theorem]{Conjecture}

\newtheorem{claim}[theorem]{Claim}

\newtheorem{definition}[theorem]{Definition}
\newtheorem{problem}[theorem]{Problem}

\newtheorem{question}[theorem]{Question}
\theoremstyle{plain}

\title{Squares of subcubic planar graphs without cycles of length 4--8 are 6-choosable}
\author
{
Seog-Jin Kim$^{\rm a}$\thanks{E-mail: {\tt skim12@konkuk.ac.kr
}},
Rong Luo$^{\rm b}$\thanks{E-mail: {\tt rluo@math.wvu.edu
}},
\\
{\footnotesize$^{\rm a}$Department of Mathematics Education, Konkuk university, Korea. }\\
{\footnotesize$^{\rm b}$ Deparment of Mathematics, West Virginia University, USA}\\
}
\date{}
\begin{document}

\maketitle

\begin{abstract}
The {\em square} of a graph $G$, denoted $G^2$, has the same vertex set as $G$ and an edge between any two vertices at distance at most $2$ in $G$. Wegner (1977) conjectured that for a planar graph $G$, $\chi(G^2) \leq 7$ if $\Delta(G) = 3$, $\chi(G^2) \leq \Delta(G)+5$ if $4 \leq \Delta(G) \leq 7$, and $\chi(G^2) \leq \lfloor 3\Delta(G)/2 \rfloor$ if $\Delta(G) \geq 8$, and Thomassen (2018) confirmed the conjecture for $\Delta(G) = 3$. Dvo\v{r}\'{a}k et al. (2008) and Feder et al. (2021) further conjectured that $\chi(G^2) \leq 6$ for cubic bipartite planar graphs. A natural question is whether this bound also holds for the list-chromatic number, i.e., whether $\chi_{\ell}(G^2) \leq 6$ for such graphs. More generally, it is of interest to determine sufficient conditions ensuring $\chi_{\ell}(G^2) \leq 6$ for subcubic planar graphs. In this paper, we prove that $\chi_{\ell}(G^2) \leq 6$ for subcubic planar graphs containing no $k$-cycles for $4 \leq k \leq 8$, improving a result of Cranston and Kim (2008).


\end{abstract}

\section{Introduction}

The {\em square} of a graph $G$, denoted $G^2$, is obtained by adding an edge between any two vertices of $G$ that are at distance at most $2$, while keeping the vertex set unchanged. A graph is {\em subcubic} if its maximum degree $\Delta(G)$ is at most $3$, and a cycle of length $k$ is called a \emph{$k$-cycle}. The {\em girth} of $G$, denoted $g(G)$, is the length of its shortest cycle.

Graph coloring and, in particular, coloring squares of graphs has received considerable attention. Let $\chi(G)$ denote the {\em chromatic number} of $G$. A {\em list assignment} for $G$ assigns to each vertex a list of colors, and $G$ is {\em $L$-colorable} if it admits a proper coloring $f$ with $f(v) \in L(v)$ for all $v$. If $G$ is $L$-colorable whenever each list has size at least $k$, then $G$ is {\em $k$-choosable}, and the {\em list chromatic number} $\chi_{\ell}(G)$ is the smallest such $k$. Let $x$ be a vertex in $G$. Denote by $N_G(x)$ the set of vertices adjacent to $x$.

One of the most well-known open problems in this area is Wegner's conjecture:

\begin{conjecture}[\cite{Wegner}] \label{conj-Wegner}
For a planar graph $G$,
\[
\chi(G^2)\leq 
\begin{cases} 
7, & \text{if } \Delta(G) = 3,\\
\Delta(G)+5, & \text{if } 4 \leq \Delta(G) \leq 7,\\
\lfloor 3\Delta(G)/2 \rfloor, & \text{if } \Delta(G) \geq 8.
\end{cases}
\]
\end{conjecture}

The case $\Delta(G)=3$ was confirmed by Thomassen \cite{Thomassen} and independently by Hartke, Jahanbekam, and Thomas \cite{Hartke} using extensive computer-assisted case analysis. The conjecture for $\Delta(G) \geq 4$ remains open, with recent progress on subcubic planar graphs documented in \cite{Cranston22, JKK, KL23}.

Another closely related conjecture concerns cubic bipartite planar graphs:

\begin{conjecture}[\cite{DST08, FHS}] \label{conj-FHS}
If $G$ is a cubic bipartite planar graph, then $\chi(G^2) \leq 6$.
\end{conjecture}

This bound is tight, as exemplified by the hexagonal prism. Feder et al. \cite{FHS} verified Conjecture \ref{conj-FHS} in special cases, showing that $\chi(G^2) \leq 6$ if the faces of a bipartite cubic plane graph can be three-colored red, blue, and green such that red faces have even size and blue and green faces have sizes divisible by four. In particular, this implies $\chi(G^2) \leq 6$ if all face sizes are multiples of four.

It is natural to extend these questions to list colorings:

\begin{question}
Does $\chi_{\ell}(G^2) \leq 6$ hold for cubic bipartite planar graphs?
\end{question}

Currently, it is unknown even whether the square of a cubic bipartite planar graph is 7-choosable. More generally, identifying sufficient conditions ensuring $\chi_{\ell}(G^2) \leq 6$ for subcubic planar graphs is an interesting direction.

This line of research echoes the approach taken for Steinberg's conjecture, which stated that planar graphs without 4- or 5-cycles are 3-colorable. Though disproved by Cohen-Addad  et al. \cite{stein}, it inspired the question suggested by Erd\H{o}s (1991) on the minimum $k$ such that planar graphs without cycles of lengths $\{4,\dots,k\}$ are 3-colorable. Borodin et al. \cite{Borodin}, improving earlier results \cite{abbott, sanders}, showed that $k \leq 7$, i.e., planar graphs with no cycles of lengths $4$ through $7$ are $3$-colorable.

Following this perspective, one can ask:

\begin{problem} \label{problem-main}
Determine the smallest $k$ such that $\chi_{\ell}(G^2) \leq 6$ whenever $G$ is a subcubic planar graph with no cycles of lengths $\{4,\dots,k\}$.
\end{problem}

Cranston and Kim \cite{CK} showed that $\chi_{\ell}(G^2) \leq 6$ for subcubic planar graphs of girth at least 9. In this paper, we improve upon their result and resolve Problem \ref{problem-main} with the following theorem:

\begin{theorem} \label{main-thm}
If $G$ is a subcubic planar graph containing no $k$-cycles for $4 \leq k \leq 8$, then $\chi_{\ell}(G^2) \leq 6$.
\end{theorem}
\section{Elements of configurations}
In this section, we prove several key configurations which are of independent interest.
First, the following lemma is folklore, but we include here for completeness.


\begin{lemma} \label{lem-key-K4-edge}
Let $P_4=v_1v_2v_3v_4$ be a path and $L$ be a list assignment of $P_4$. 
If $|L(v_1)| = |L(v_3)| = |L(v_4)| = 2$ and $|L(v_2)| = 3$, then $P_4^2$ is $L$-colorable. 
\end{lemma} 
\begin{proof}
If $L(v_1) \cap L(v_4) \neq \emptyset$, then  first color $v_1$ and $v_4$ with  a color $c \in L(v_1) \cap L(v_4)$, and then color $v_3$ and $v_2$ in this order to obtain an $L$-coloring of $P_4^2$.

If $L(v_1) \cap L(v_4) = \emptyset$, then  first color $v_2$ with a color $c \in L(v_2) \setminus L(v_3)$.  Since $L(v_1) \cap L(v_4) = \emptyset$, $c \notin L(v_1)$ or $c \notin L(v_4)$.  
Without loss of generality, we may assume that $c \notin L(v_1)$.  Then, greedily color $v_4, v_3, v_1$ in this order to obtain an $L$-coloring of $P_4^2$.  

In either case, we obtain an $L$-coloring of $P_4^2$ and this proves the lemma.
\end{proof}


\begin{lemma} \label{lem-key-J1} 
Let $J_1$ be the graph  in Figure \ref{fig-J12}  and $L$ be a list assignment of $J_1$. If  $|L(v_1)| = 2, |L(v_2)| = 3,  |L(v_3)| = 4, |L(v_4)| = 3$,  and $|L(v_5)| = 3$, then $J_1^2$ is $L$-colorable.
\end{lemma}
\begin{proof}
We consider two cases.

\medskip
\noindent {\bf Case 1}: $L(v_1) \cap L(v_4) \neq \emptyset$

Color $v_1$ and $v_4$ with a color  $\alpha \in L(v_1) \cap L(v_4)$, and then greedily color $v_5, v_2, v_3$ in this order to obtain an $L$-coloring of $J_1^2$.

\medskip
\noindent {\bf Case 2}: $L(v_1) \cap L(v_4) = \emptyset$

In this case, we have that $|L(v_1) \cup L(v_4)| = 5$.
If there is a color $\alpha \in L(v_1) \setminus L(v_3)$, color $v_1$ with $\alpha$ first,  and then color $v_2, v_5, v_4, v_3$ in this order to obtain an $L$-coloring of $J_1^2$.

Now assume that $L(v_1) \subset L(v_3)$.  Then $L(v_4)\setminus L(v_3) \not  = \emptyset$ since $|L(v_1) \cup L(v_4)| = 5$.  Let $\beta  \in L(v_4)\setminus L(v_3)$. Then $\beta \not \in L(v_1)$ and either $\beta \in L(v_5)$, or $\beta \not \in L(v_5)$.

If $\beta \not \in L(v_5)$, color $v_4$ with $\beta$ first,  and then color  $v_1, v_2,v_5,v_3$ in this order  to obtain an $L$-coloring of $J_1^2$.

If $\beta \in L(v_5)$, color $v_5$ with $\beta$ first,  and then color  $v_1, v_2,v_4,v_3$ in this order  to obtain an $L$-coloring of $J_1^2$.

In each case, we obtain an $L$-coloring of $J_1^2$. This completes the proof of the lemma.
\end{proof} 


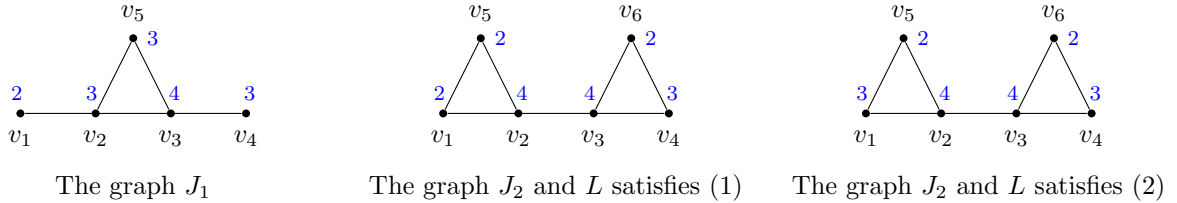
\begin{figure*}[htbp]
\begin{multicols}{3}\begin{center}
\begin{tikzpicture}
[u/.style={fill=black, minimum size =3pt,ellipse,inner sep=1pt},node distance=1.5cm,scale=1]
\node[u] (v1) at (1, 0){};
\node[u] (v2) at (2, 0){};
\node[u] (v3) at (3, 0){};
\node[u] (v4) at (4, 0){};
\node[u] (v5) at (2.5, 1){};

  \draw (v1) -- (v4);  
  \draw (v2) -- (v5);   
  \draw (v3) -- (v5);    

   \node[below=0.1cm, font=\small] at (v1) {$v_1$};  
   \node[below=0.1cm,font=\small] at (v2) {$v_2$};  
   \node[below=0.1cm, font=\small] at (v3) {$v_3$};     
   \node[below=0.1cm,font=\small] at (v4) {$v_4$};     
    \node[above=0.1cm,font=\small] at (v5) {$v_5$};

 \node[left=0.05cm, above=0.05cm,  font=\scriptsize] at (v1) {\textcolor{blue}{2}};
\node[left=0.05cm, above=0.05cm,font=\scriptsize] at (v2) {\textcolor{blue}{3}};  
\node[right=0.05cm, above=0.05cm,font=\scriptsize] at (v3) {\textcolor{blue}{4}};  
\node[right=0.05cm, above=0.05cm, font=\scriptsize] at (v4) {\textcolor{blue}{3}};  
\node[right=0.05cm, font=\scriptsize] at (v5) {\textcolor{blue}{3}}; 
 \end{tikzpicture}
         \vfill {\small The graph  $J_1$}  
\end{center} 
\begin{center}
\begin{tikzpicture}
[u/.style={fill=black, minimum size =3pt,ellipse,inner sep=1pt},node distance=1.5cm,scale=1]
\node[u] (v1) at (1, 0){};
\node[u] (v2) at (2, 0){};
\node[u] (v3) at (3, 0){};
\node[u] (v4) at (4, 0){};
\node[u] (v5) at (1.5, 1){};
\node[u] (v6) at (3.5, 1){};

  \draw (v1) -- (v4);  
  \draw (v1) -- (v5);
  \draw (v2) -- (v5);   
  \draw (v3) -- (v6);
  \draw (v4) -- (v6);

   \node[below=0.1cm, font=\small] at (v1) {$v_1$};  
   \node[below=0.1cm,font=\small] at (v2) {$v_2$};  
   \node[below=0.1cm, font=\small] at (v3) {$v_3$};     
   \node[below=0.1cm,font=\small] at (v4) {$v_4$};     
    \node[above=0.1cm,font=\small] at (v5) {$v_5$};  
     \node[above=0.1cm, font=\small] at (v6) {$v_6$};

 \node[left=0.05cm, above=0.05cm,  font=\scriptsize] at (v1) {\textcolor{blue}{2}};
\node[right=0.05cm, above=0.05cm,font=\scriptsize] at (v2) {\textcolor{blue}{4}};  
\node[left=0.05cm, above=0.05cm,font=\scriptsize] at (v3) {\textcolor{blue}{4}};  
\node[right=0.05cm, above=0.05cm, font=\scriptsize] at (v4) {\textcolor{blue}{3}};  
\node[right=0.05cm, font=\scriptsize] at (v5) {\textcolor{blue}{2}}; 
\node[right=0.05cm, font=\scriptsize] at (v6) {\textcolor{blue}{2}}; 
 \end{tikzpicture}
        \vfill {\small The graph  $J_2$ and $L$ satisfies (1)}  
\end{center} 
\par
\begin{center}
\begin{tikzpicture}
[u/.style={fill=black, minimum size =3pt,ellipse,inner sep=1pt},node distance=1.5cm,scale=1]
\node[u] (v1) at (1, 0){};
\node[u] (v2) at (2, 0){};
\node[u] (v3) at (3, 0){};
\node[u] (v4) at (4, 0){};
\node[u] (v5) at (1.5, 1){};
\node[u] (v6) at (3.5, 1){};

  \draw (v1) -- (v4);  
  \draw (v1) -- (v5);
  \draw (v2) -- (v5);   
  \draw (v3) -- (v6);
  \draw (v4) -- (v6);

   \node[below=0.1cm, font=\small] at (v1) {$v_1$};  
   \node[below=0.1cm,font=\small] at (v2) {$v_2$};  
   \node[below=0.1cm, font=\small] at (v3) {$v_3$};     
   \node[below=0.1cm,font=\small] at (v4) {$v_4$};     
    \node[above=0.1cm,font=\small] at (v5) {$v_5$};  
     \node[above=0.1cm, font=\small] at (v6) {$v_6$};

 \node[left=0.05cm, above=0.05cm,  font=\scriptsize] at (v1) {\textcolor{blue}{3}};
\node[right=0.05cm, above=0.05cm,font=\scriptsize] at (v2) {\textcolor{blue}{4}};  
\node[left=0.05cm, above=0.05cm,font=\scriptsize] at (v3) {\textcolor{blue}{4}};  
\node[right=0.05cm, above=0.05cm, font=\scriptsize] at (v4) {\textcolor{blue}{3}};  
\node[right=0.05cm, font=\scriptsize] at (v5) {\textcolor{blue}{2}}; 
\node[right=0.05cm, font=\scriptsize] at (v6) {\textcolor{blue}{2}}; 
 \end{tikzpicture}
        \vfill {\small The graph  $J_2$ and $L$ satisfies (2)} 
\end{center} 
\end{multicols} 
\caption{The graphs $J_1$ and $J_2$ with list sizes shown at vertices; the number at each vertex denotes the number of available colors. } 
\label{fig-J12}
\end{figure*}

\begin{lemma} \label{lem-J2}
Let $J_2$ be the graph  in Figure \ref{fig-J12}  and $L$ be a list assignment of $J_2$.  Then $J_2^2$ is  $L$-colorable if $L$ satisfies one of the following conditions:
\begin{enumerate}[(1)]
\item $|L(v_1)| = 2, |L(v_2)| = 4,  |L(v_3)| = 4, |L(v_4)| = 3, |L(v_{5})| = 2, |L(v_{6})| = 2$ and $L(v_1) \neq L(v_{5})$;

\item $|L(v_1)| = 3, |L(v_2)| = 4,  |L(v_3)| = 4, |L(v_4)| = 3, |L(v_{5})| = 2, |L(v_{6})| = 2$.
\end{enumerate}
\end{lemma}
\begin{proof}
Note that if $L$ satisfies (2), then, since $|L(v_1)| = 3$, we may remove a color 
$c \in L(v_1)$ such that $L(v_1) \setminus \{c\} \neq L(v_5)$. 
By (1), this implies that $J_2^2$ is $L$-colorable. 
Therefore, it remains to show that if $L$ satisfies (1), then $J_2^2$ is $L$-colorable. 

We first show that there exist two distinct colors 
$c_1 \in L(v_1)$ and $c_5 \in L(v_5)$ such that, 
after coloring $v_1$ and $v_5$ with $c_1$ and $c_5$, respectively, 
the resulting list assignment $L'$ satisfies:
\begin{enumerate}[(i)]
\item $|L'(v_2)| \ge 2$, $|L'(v_3)| \ge 2$, $|L'(v_4)| \ge 3$, and $|L'(v_6)| \ge 2$;
\item $|L'(v_2) \cup L'(v_3) \cup L'(v_4) \cup L'(v_6)| \ge 4$; and
\item $L'(v_3) \ne L'(v_6)$.
\end{enumerate}

Note that, regardless of the choices of $c_1$ and $c_5$, Condition (i) is always satisfied.

Since $|L(v_3)| = 4 > |L(v_4)| = 3$, let $\alpha \in L(v_3) \setminus L(v_4)$. 

\medskip
\noindent
\textbf{Case 1:} $\alpha \notin L(v_6)$.  
We can choose $c_1$ and $c_5$ such that $\alpha \notin \{c_1, c_5\}$.  
Such a choice exists since $L(v_1) \ne L(v_5)$ and both lists have size $2$. 
Then $L'(v_3) \ne L'(v_6) = L(v_6)$, so Condition (iii) is satisfied.  
Moreover, since $\alpha \in L'(v_3)$ but $\alpha \notin L'(v_4) = L(v_4)$ and $|L'(v_4)| = 3$, we have
\[
|L'(v_2) \cup L'(v_3) \cup L'(v_4) \cup L'(v_6)| 
\ge |L'(v_3) \cup L'(v_4)| \ge 4,
\]
so (ii) is satisfied. Hence, $c_1$ and $c_5$ are the desired colors.

\medskip
\noindent
\textbf{Case 2:} $\alpha \in L(v_6) = L'(v_6)$.  
Since $\alpha \notin L'(v_4)$ and $|L'(v_4)| = 3$, Condition (ii) is always satisfied.  
It remains to choose $c_1$ and $c_5$ so that $L'(v_3) \ne L'(v_6)$.

If $\alpha \in L(v_1) \cup L(v_5)$, choose $c_1$ and $c_5$ such that $\alpha \in \{c_1, c_5\}$.  
Then $\alpha \notin L'(v_3)$ and $\alpha \in L'(v_6)$, so $L'(v_3) \ne L'(v_6)$.

If $\alpha \notin L(v_1) \cup L(v_5)$, note that $L(v_1) \ne L(v_5)$ and $|L(v_1)| = |L(v_5)| = 2$, giving at least three possible pairs $(c_1, c_5)$.  
Among them, there exists a pair such that 
$L'(v_3) = L(v_3) \setminus \{c_1, c_5\} \ne L(v_6) = L'(v_6)$.

\medskip
Let $c_1 \in L(v_1)$ and $c_5 \in L(v_5)$ be two distinct colors satisfying (i)--(iii).  
Color $v_1$ and $v_5$ with $c_1$ and $c_5$, respectively.  
We now extend this coloring to $v_2, v_3, v_4$, and $v_6$ to obtain an $L$-coloring of $J_2^2$.

Construct an auxiliary bipartite graph $W$ with bipartition $(X,Y)$, where
\begin{itemize}
\item $X = \{v_2, v_3, v_4, v_6\}$ and $Y = \bigcup_{v_i \in X} L'(v_i)$;
\item For each $v_i \in X$ and each $\beta \in Y$, $v_i\beta \in E(W)$ if and only if $\beta \in L'(v_i)$.
\end{itemize}

By (ii), $|N_W(v_2) \cup N_W(v_3) \cup N_W(v_4) \cup N_W(v_6)| \ge 4$, and by (iii), $|N_W(v_3) \cup N_W(v_6)| \ge 3$.  
Together with (i), it follows that $|N_W(S)| \ge |S|$ for every $S \subseteq \{v_2, v_3, v_4, v_6\}$.  
Hence, by Hall's theorem, $W$ has a matching $M$ saturating all vertices of $X$.  
Let $v_2c_2$, $v_3c_3$, $v_4c_4$, and $v_6c_6$ be the edges of such a matching.  
Color $v_2, v_3, v_4, v_6$ with $c_2, c_3, c_4, c_6$, respectively, to obtain an $L$-coloring of $J_2^2$.  

This completes the proof of the lemma.
\end{proof}
\section{Proof of Theorem~ \ref{main-thm}}

In this section, we prove Theorem~\ref{main-thm} by contradiction.
Let G, together with a list assignment L of size 6, be a minimal counterexample to Theorem~\ref{main-thm}.
We also assume that G is a plane graph. Clearly, G is connected.
We begin with the following claim, whose proof is straightforward.

\begin{claim}
\label{CL:2-vertex}
\begin{enumerate}[(1)]
\item The minimum degree of $G$ satisfies $\delta(G) \ge 2$.

\item No $2$-vertex of $G$ is contained in a triangle.

\item No two $2$-vertices of $G$ are adjacent.
\end{enumerate}
\end{claim}

\begin{proof}
(1) Suppose to the contrary that $G$ contains a $1$-vertex $x$, and let $y$ be the neighbor of $x$. By the minimality of $G$, the square graph of  $G - x$ has an $L$-coloring $\phi$. Since there are at most three vertices of $G$ at distance at most $2$ from $x$, the coloring $\phi$ can always be extended to $x$, giving an $L$-coloring of $G^2$, a contradiction.

(2) Suppose to the contrary that $G$ has a $2$-vertex $x$ contained in a triangle $xyzx$.   By the minimality of $G$, the square graph  of $G-x$ has an $L$-coloring $\phi$. Since at most four colors are forbidden for $x$ and $|L(x)| = 6$, the coloring  $\phi$  can be extended to $x$, giving an $L$-coloring of $G^2$, a contradiction.

(3) Suppose to the contrary that $G$ contains two adjacent $2$-vertices $x$ and $y$. Let $x_1$ be the other neighbor of $x$, and let $y_1$ be the other neighbor of $y$. By (2), we have $x_1 \ne y_1$.   

By the minimality of $G$, the square graph  of $G - \{x, y\}$ has an $L$-coloring $\phi$.

For each $w \in \{x, y\}$, define
\[
L'(w) = L(w) \setminus \{\phi(z) : zw \in E(G^2) \text{ and } z \notin \{x, y\}\}.
\]
Since at most four colors are forbidden for each of $x$ and $y$, we have $|L'(x)| \ge 6 - 4 = 2$ and $|L'(y)| \ge 2$.  
Thus $\phi$ extends to both $x$ and $y$, a contradiction. This completes the proof of the claim.
\end{proof}
\begin{figure*}[htbp]
\begin{multicols}{3}
\begin{center}
\begin{tikzpicture}
[u/.style={fill=black, minimum size =3pt,ellipse,inner sep=1pt},
u-circle/.style={circle, draw, fill=none, minimum size =3pt,ellipse,inner sep=1pt},node distance=1.5cm,scale=0.8]

\node[u-circle] (v1) at (1, 0){};
\node[u] (v2) at (2, 0){};
\node[u] (v3) at (3, 0){};
\node[u-circle] (v4) at (4, 0){};
\node[u] (v5) at (2.5, 1){};

  \draw (0.5, 0) -- (v1);   
  \draw (v1) -- (v2);   
  \draw (v2) -- (v3);
  \draw (v3) -- (v4);   
  \draw (v4) -- (4.5, 0);        

  \draw (v2) -- (v5);   
  \draw (v3) -- (v5);

   \node[below=0.1cm, font=\small] at (v1) {$v_1$};  
   \node[below=0.1cm,font=\small] at (v2) {$v_2$};  
   \node[below=0.1cm, font=\small] at (v3) {$v_3$};     
   \node[below=0.1cm,font=\small] at (v4) {$v_4$};     
    \node[above=0.1cm,font=\small] at (v5) {$v_5$};  

 \end{tikzpicture}
        \vfill {\small The subgraph $T_1$}  
\end{center} 
\par
\begin{center}
\begin{tikzpicture}
[u/.style={fill=black, minimum size =3pt,ellipse,inner sep=1pt},
u-circle/.style={circle, draw, fill=none, minimum size =3pt,ellipse,inner sep=1pt},node distance=1.5cm,scale=0.8]

\node[u] (v1) at (1, 0){};
\node[u] (v2) at (2, 0){};
\node[u] (v3) at (3, 0){};
\node[u] (v4) at (4, 0){};
\node[u-circle] (v5) at (5, 0){};
\node[u] (v6) at (1.5, 1){};
\node[u] (v7) at (3.5, 1){};
 
  \draw (0.5, 0) -- (v1);   
  \draw (v1) -- (v2);   
  \draw (v2) -- (v3);
  \draw (v3) -- (v4);   
  \draw (v4) -- (v5);  
  \draw (v5) -- (5.5, 0);   
   
  \draw (v1) -- (v6);
  \draw (v2) -- (v6);    
  \draw (v3) -- (v7);
  \draw (v4) -- (v7);

   \node[below=0.1cm, font=\small] at (v1) {$v_1$};  
   \node[below=0.1cm,font=\small] at (v2) {$v_2$};  
   \node[below=0.1cm, font=\small] at (v3) {$v_3$};     
   \node[below=0.1cm,font=\small] at (v4) {$v_4$};     
      
    \node[below=0.1cm, font=\small] at (v5) {$v_5$};    
     \node[above=0.1cm, font=\small] at (v6) {$v_6$};         
     \node[above=0.1cm, font=\small] at (v7) {$v_7$};                
    
 \end{tikzpicture}
        \vfill {\small The subgraph  $T_2$} 
\end{center}
\par
\begin{center}
\begin{tikzpicture}
[u/.style={fill=black, minimum size =3pt,ellipse,inner sep=1pt},
u-circle/.style={circle, draw, fill=none, minimum size =3pt,ellipse,inner sep=1pt},node distance=1.5cm,scale=0.8]

\node[u] (v1) at (1, 0){};
\node[u] (v2) at (2, 0){};
\node[u-circle] (v3) at (3, 0){};
\node[u] (v4) at (4, 0){};
\node[u-circle] (v5) at (5, 0){};
\node[u] (v6) at (1.5, 1){};

  \draw (v1) -- (v2);   
  \draw (v2) -- (v3);
  \draw (v3) -- (v4);   
  \draw (v4) -- (v5);  
  \draw (v5) -- (5.5, 0);   
  
  \draw (v1) -- (v6);
  \draw (v2) -- (v6);    
  \draw (v4) -- (4, 0.5);

   \node[below=0.1cm, font=\small] at (v1) {$v_1$};  
   \node[below=0.1cm,font=\small] at (v2) {$v_2$};  
   \node[below=0.1cm, font=\small] at (v3) {$v_3$};     
   \node[below=0.1cm,font=\small] at (v4) {$v_4$};     
      
    \node[below=0.1cm, font=\small] at (v5) {$v_5$};    
     \node[above=0.1cm, font=\small] at (v6) {$v_6$};                  
 \end{tikzpicture}
        \vfill {\small The subgraph  $T_3$}  
\end{center} 
\end{multicols} 
\caption{Reducible configurations $T_1$, $T_2$, and $T_3$. A black vertex denotes a 3-vertex and a white vertex denotes a 2-vertex.} 
\label{T123-subgraph}
\end{figure*}
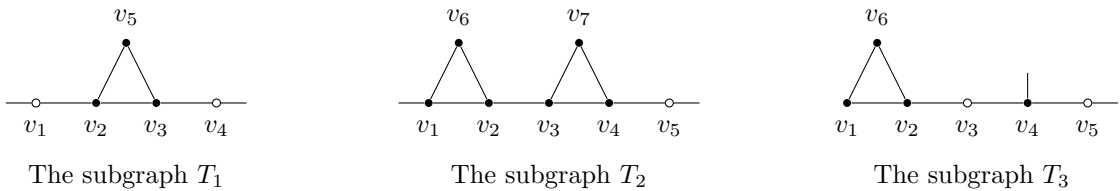

One of the main steps is to estimate the number of $2$-vertices and $3$-faces adjacent to any face of length at least $9$.  
We begin by proving the following claim.

\begin{claim}\label{cl:reducible}
For each $i \in \{1,2,3\}$, the graph $G$ does not contain the subgraph $T_i$ illustrated in Figure~\ref{T123-subgraph}.
\end{claim}

\begin{proof}
We prove that none of $T_1$, $T_2$, or $T_3$ appears in $G$, each by contradiction.  
Recall that by Claim~\ref{CL:2-vertex}, every vertex in a triangle is a $3$-vertex.

\medskip\noindent
(a)~Suppose that $G$ contains $T_1$ as a subgraph.  
By the minimality of $G$, the square graph of  $G - V(T_1)$ has an $L$-coloring $\phi$.  
For each $v_i \in V(T_1)$, define
\[
L'(v_i) = L(v_i) \setminus \{\phi(x) : v_i x \in E(G^2) \text{ and } x \notin V(T_1)\}.
\]
Then, as shown in Figure~\ref{T123-color}(i),
\[
|L'(v_i)| \ge
\begin{cases}
3, & i \in \{1,4,5\},\\[2pt]
4, & i \in \{2,3\}.
\end{cases}
\]
By Lemma~\ref{lem-key-J1}, the coloring $\phi$ extends to an $L$-coloring of $G^2$, 
contradicting the choice of $G$.

\medskip\noindent
(b)~Suppose that $G$ contains $T_2$ as a subgraph.  
By the minimality of $G$, let $\phi$ be an $L$-coloring of  the square graph of $G - \{v_2,v_3,v_4,v_5,v_7\}$.  
For each $i \in \{2,3,4,5,7\}$, define
\[
L'(v_i) = L(v_i) \setminus \{\phi(x) : v_i x \in E(G^2) \text{ and } x \notin \{v_2,v_3,v_4,v_5,v_7\}\}.
\]
Then, as shown in Figure~\ref{T123-color}(ii),
\[
|L'(v_2)| \ge 2,\quad
|L'(v_3)| \ge 3,\quad
|L'(v_4)| \ge 4,\quad
|L'(v_5)| \ge 3,\quad
|L'(v_7)| \ge 3.
\]
By Lemma~\ref{lem-key-J1}, the coloring $\phi$ extends to an $L$-coloring of $G^2$, 
again a contradiction.

\medskip\noindent
(c)~Suppose that $G$ contains $T_3$ as a subgraph.  
Let $\phi$ be an $L$-coloring  of the square graph of  $G - \{v_2,v_3,v_4,v_5\}$, and for each $i \in \{2,3,4,5\}$, define
\[
L'(v_i) = L(v_i) \setminus \{\phi(x) : v_i x \in E(G^2) \text{ and } x \notin \{v_2,v_3,v_4,v_5\}\}.
\]
Then, as shown in Figure~\ref{T123-color}(iii),
\[
|L'(v_2)| \ge 2,\qquad
|L'(v_3)| \ge 3,\qquad
|L'(v_4)| \ge 2,\qquad
|L'(v_5)| \ge 2.
\]
By Lemma~\ref{lem-key-K4-edge}, the coloring $\phi$ extends to an $L$-coloring of $G^2$, 
a contradiction.

\medskip
In all three cases we obtain a contradiction.  
Thus $G$ contains none of $T_1$, $T_2$, or $T_3$ as a subgraph.
\end{proof}

\begin{figure*}[htbp]
\begin{multicols}{3}
\begin{center}
\begin{tikzpicture}
[u/.style={fill=black, minimum size =3pt,ellipse,inner sep=1pt},
u-circle/.style={circle, draw, fill=none, minimum size =3pt,ellipse,inner sep=1pt},node distance=1.5cm,scale=0.8]

\node[u-circle] (v1) at (1, 0){};
\node[u] (v2) at (2, 0){};
\node[u] (v3) at (3, 0){};
\node[u-circle] (v4) at (4, 0){};
\node[u] (v5) at (2.5, 1){};

  \draw (0.5, 0) -- (v1);   
  \draw (v1) -- (v2);   
  \draw (v2) -- (v3);
  \draw (v3) -- (v4);   
  \draw (v4) -- (4.5, 0);        
  \draw (v2) -- (v5);   
  \draw (v3) -- (v5);    
    
   \node[below=0.1cm, font=\small] at (v1) {$v_1$};  
   \node[below=0.1cm,font=\small] at (v2) {$v_2$};  
   \node[below=0.1cm, font=\small] at (v3) {$v_3$};     
   \node[below=0.1cm,font=\small] at (v4) {$v_4$};     
    \node[above=0.1cm,font=\small] at (v5) {$v_5$};          
    
 \node[above=0.05cm,  font=\scriptsize] at (v1) {\textcolor{blue}{3}};
\node[left=0.05cm, above=0.05cm,font=\scriptsize] at (v2) {\textcolor{blue}{4}};  
\node[right=0.05cm, above=0.05cm,font=\scriptsize] at (v3) {\textcolor{blue}{4}};  
\node[above=0.05cm, font=\scriptsize] at (v4) {\textcolor{blue}{3}};  
\node[right=0.05cm, font=\scriptsize] at (v5) {\textcolor{blue}{3}}; 
 \end{tikzpicture}
        \vfill {(i) List sizes of $T_1$}  
\end{center} 
\par
\begin{center}
\begin{tikzpicture}
[u/.style={fill=black, minimum size =3pt,ellipse,inner sep=1pt},
u-circle/.style={circle, draw, fill=none, minimum size =3pt,ellipse,inner sep=1pt},node distance=1.5cm,scale=0.8]
\node[u] (v1) at (1, 0){};
\node[u] (v2) at (2, 0){};
\node[u] (v3) at (3, 0){};
\node[u] (v4) at (4, 0){};
\node[u-circle] (v5) at (5, 0){};
\node[u] (v6) at (1.5, 1){};
\node[u] (v7) at (3.5, 1){};

  \draw (0.5, 0) -- (v1);   
  \draw (v1) -- (v2);   
  \draw (v2) -- (v3);
  \draw (v3) -- (v4);   
  \draw (v4) -- (v5);  
  \draw (v5) -- (5.5, 0);   
   
  \draw (v1) -- (v6);
  \draw (v2) -- (v6);    
  \draw (v3) -- (v7);
  \draw (v4) -- (v7);

   \node[below=0.1cm, font=\small] at (v1) {$v_1$};  
   \node[below=0.1cm,font=\small] at (v2) {$v_2$};  
   \node[below=0.1cm, font=\small] at (v3) {$v_3$};     
   \node[below=0.1cm,font=\small] at (v4) {$v_4$};     
      
    \node[below=0.1cm, font=\small] at (v5) {$v_5$};    
     \node[above=0.1cm, font=\small] at (v6) {$v_6$};         
     \node[above=0.1cm, font=\small] at (v7) {$v_7$};                
    
 \node[left=0.05cm, above=0.05cm,  font=\scriptsize] at (v1) {};
\node[right=0.05cm, above=0.05cm,font=\scriptsize] at (v2) {\textcolor{blue}{2}};  
\node[left=0.05cm, above=0.05cm,font=\scriptsize] at (v3) {\textcolor{blue}{3}};  
\node[right=0.05cm, above=0.05cm, font=\scriptsize] at (v4) {\textcolor{blue}{4}}; 
\node[above=0.05cm, font=\scriptsize] at (v5) {\textcolor{blue}{3}};  
\node[right=0.1cm, font=\scriptsize] at (v6) {}; 
\node[right=0.1cm, font=\scriptsize] at (v7) {\textcolor{blue}{3}}; 
 \end{tikzpicture}
        \vfill {(ii) List sizes of $T_2$} 
\end{center}
\par
\begin{center}
\begin{tikzpicture}
[u/.style={fill=black, minimum size =3pt,ellipse,inner sep=1pt},
u-circle/.style={circle, draw, fill=none, minimum size =3pt,ellipse,inner sep=1pt},node distance=1.5cm,scale=0.8]
\node[u] (v1) at (1, 0){};
\node[u] (v2) at (2, 0){};
\node[u-circle] (v3) at (3, 0){};
\node[u] (v4) at (4, 0){};
\node[u-circle] (v5) at (5, 0){};
\node[u] (v6) at (1.5, 1){};

  \draw (v1) -- (v2);   
  \draw (v2) -- (v3);
  \draw (v3) -- (v4);   
  \draw (v4) -- (v5);  
  \draw (v5) -- (5.5, 0);   
  
  \draw (v1) -- (v6);
  \draw (v2) -- (v6);    
  \draw (v4) -- (4, 0.5);

   \node[below=0.1cm, font=\small] at (v1) {$v_1$};  
   \node[below=0.1cm,font=\small] at (v2) {$v_2$};  
   \node[below=0.1cm, font=\small] at (v3) {$v_3$};     
   \node[below=0.1cm,font=\small] at (v4) {$v_4$};     
      
    \node[below=0.1cm, font=\small] at (v5) {$v_5$};    
     \node[above=0.1cm, font=\small] at (v6) {$v_6$};                  
    
 \node[left=0.05cm, above=0.05cm,  font=\scriptsize] at (v1) {};
\node[right=0.05cm, above=0.05cm,font=\scriptsize] at (v2) {\textcolor{blue}{2}};  
\node[above=0.05cm,font=\scriptsize] at (v3) {\textcolor{blue}{3}};  
\node[right=0.1cm, above=0.05cm, font=\scriptsize] at (v4) {\textcolor{blue}{2}}; 
\node[above=0.05cm, font=\scriptsize] at (v5) {\textcolor{blue}{2}};  
\node[right=0.1cm, font=\scriptsize] at (v6) {}; 
 \end{tikzpicture}
        \vfill {(iii) List sizes of $T_3$}  
\end{center} 
\end{multicols} 
\caption{
The number at  each vertex denotes the number of available colors.} 
\label{T123-color}
\end{figure*}
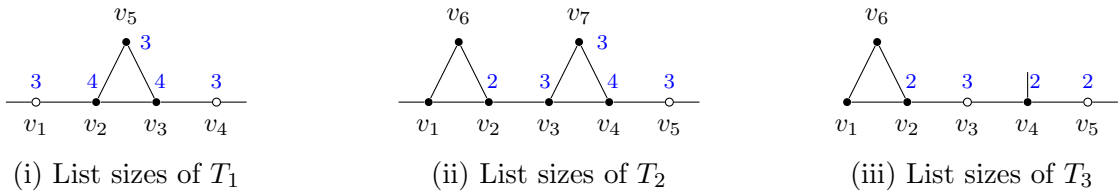

Claim~\ref{cl:reducible} yields several structural constraints on a long cycle, which will allow us to estimate the total number of $2$-vertices and triangles adjacent to it.  
To describe these properties precisely, we introduce the following notation.

For a cycle $C$ in $G$, let $d(C)$ denote its length.  
A \emph{$k$-cycle} is a cycle of length $k$, and a \emph{$k^+$-cycle} is a cycle of length at least $k$.

\begin{definition}
Let $C$ be a $9^+$-cycle of $G$.

\smallskip
(a) Let $t(C)$ denote the sum of the number of $2$-vertices on $C$  and the number of  triangles sharing an edge with $C$.

\smallskip
(b) Let $A(C) = \{z_1,\dots, z_k\}$ be the set of $3$-vertices adjacent to $C$, listed in cyclic order around $C$, that are not adjacent to any triangle.  
For each $i$, let $M_i$ be the segment of $C$ from $z_i$ to $z_{i+1}$, with indices taken modulo $k$, and let $t(M_i)$ denote the number of $2$-vertices and triangles sharing an edge with $M_i$.
\end{definition}

As an example, consider the cycle in $H_3$ in Figure~\ref{H234}.  
We may label $v_6 = z_1$, $v_8 = z_2$, and $v_{10} = z_3$.  
Then we obtain:

\begin{itemize}
\item $t(H_3) = 5$;
\item $A(C) = \{z_1, z_2, z_3\}$, so $k = 3$;
\item $M_1 = v_6v_7v_8$, \; $M_2 = v_8v_9v_{10}$, \; and  
$M_3 = v_{10}v_1v_2v_3v_4v_5v_6$;
\item $t(M_1) = t(M_2) = 1$ and $t(M_3) = 3$.
\end{itemize}

With this observation, the next two claims follow directly from 
Claim~\ref{cl:reducible} and the definitions.

\begin{claim} \label{segment-S}
Let $C$ be a $9^+$-cycle and suppose that 
$A(C)=\{z_1,z_2,\dots,z_k\}\ne\emptyset$.  
Then each segment $M_i$ is isomorphic to one of the configurations 
$S_1$, $S_2$, $S_3$, $S_4$, or $S_5$ shown in  
Figure~\ref{S45}.
\end{claim}

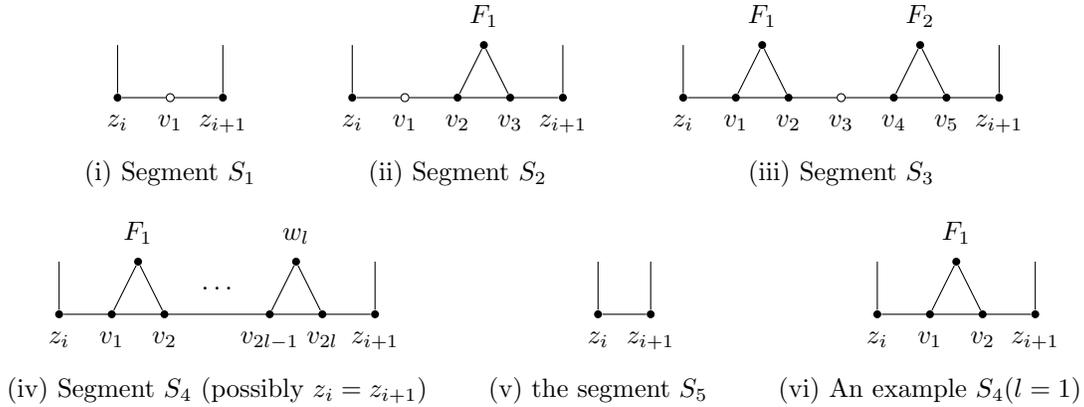
\begin{figure*}[htbp]
\begin{center}
\subfigure{
\begin{tikzpicture}
[u/.style={fill=black, minimum size =3pt,ellipse,inner sep=1pt},
u-circle/.style={circle, draw, fill=none, minimum size =3pt,ellipse,inner sep=1pt},node distance=1.5cm,scale=0.7]
\node[u] (v0) at (0, 0){};
\node[u-circle] (v1) at (1, 0){};
\node[u] (v2) at (2, 0){};
  \draw (v0) -- (v1);  
  \draw (v1) -- (v2);    
  \draw (v0) -- (0, 1);
  \draw (v2) -- (2, 1);    
   \node[below=0.1cm, font=\small] at (v0) {$z_i$};  
    \node[below=0.1cm, font=\small] at (v1) {$v_1$};   
   \node[below=0.1cm,font=\small] at (v2) {$z_{i+1}$};  
   \node[left=0.7cm,below=0.7cm] at (v2) {\small (i) Segment $S_1$};              
 \end{tikzpicture}}
  \hspace{0.5cm}
\subfigure{
\begin{tikzpicture}
[u/.style={fill=black, minimum size =3pt,ellipse,inner sep=1pt},
u-circle/.style={circle, draw, fill=none, minimum size =3pt,ellipse,inner sep=1pt},node distance=1.5cm,scale=0.7]
\node[u] (v0) at (0, 0){};
\node[u-circle] (v1) at (1, 0){};
\node[u] (v2) at (2, 0){};
\node[u] (v3) at (3, 0){};
\node[u] (v4) at (4, 0){};
\node[u] (w1) at (2.5, 1){};

  \draw (v0) -- (v1);  
  \draw (v1) -- (v4);    
  \draw (v2) -- (w1);
  \draw (v3) -- (w1);                
  \draw (v0) -- (0, 1);     
   \draw (v4) -- (4, 1);          
   
   \node[below=0.1cm, font=\small] at (v1) {$v_1$};  
   \node[below=0.1cm,font=\small] at (v2) {$v_2$};       
   \node[below=0.1cm,font=\small] at (v3) {$v_3$}; 
      \node[below=0.1cm,font=\small] at (v4) {$z_{i+1}$}; 
      \node[below=0.1cm, font=\small] at (v0) {$z_i$};        
      
    \node[above=0.1cm, font=\small] at (w1) {$F_1$};    
   \node[left=0.7cm,below=0.7cm] at (v3) {\small (ii) Segment $S_2$};        
 \end{tikzpicture}}
 \hspace{0.5cm}
\subfigure{
\begin{tikzpicture}
[u/.style={fill=black, minimum size =3pt,ellipse,inner sep=1pt},
u-circle/.style={circle, draw, fill=none, minimum size =3pt,ellipse,inner sep=1pt},node distance=1.5cm,scale=0.7]
\node[u] (v0) at (0, 0){};
\node[u] (v1) at (1, 0){};
\node[u] (v2) at (2, 0){};
\node[u-circle] (v3) at (3, 0){};
\node[u] (v4) at (4, 0){};
\node[u] (v5) at (5, 0){};
\node[u] (v6) at (6, 0){};
\node[u] (w1) at (1.5, 1){};
\node[u] (w2) at (4.5, 1){};

  \draw (v0) -- (v3);  
  \draw (v3) -- (v6);   
  \draw (v1) -- (w1);
  \draw (v2) -- (w1);    
  \draw (v4) -- (w2);
  \draw (v5) -- (w2);            
  \draw (v0) -- (0, 1);     
   \draw (v6) -- (6, 1);     
    
   \node[below=0.1cm, font=\small] at (v1) {$v_1$};  
   \node[below=0.1cm,font=\small] at (v2) {$v_2$};  
   \node[below=0.1cm, font=\small] at (v3) {$v_3$};     
   \node[below=0.1cm,font=\small] at (v4) {$v_4$};     
   \node[below=0.1cm, font=\small] at (v5) {$v_5$};     
   \node[below=0.1cm,font=\small] at (v6) {$z_{i+1}$}; 
      \node[below=0.1cm, font=\small] at (v0) {$z_i$};        
      
    \node[above=0.1cm, font=\small] at (w1) {$F_1$};    
     \node[above=0.1cm, font=\small] at (w2) {$F_2$};  
   \node[left=0.7cm,below=0.7cm] at (v4) {\small (iii) Segment $S_3$};               
 \end{tikzpicture}}
\subfigure{\begin{tikzpicture}
[u/.style={fill=black, minimum size =3pt,ellipse,inner sep=1pt},
p/.style={fill=none, minimum size =3pt,ellipse,inner sep=1pt},
node distance=1.5cm,scale=0.7];
\node[u] (v0) at (0, 0){};
\node[u] (v1) at (1, 0){};
\node[u] (v2) at (2, 0){};
\node[u] (v4) at (4, 0){};
\node[u] (v5) at (5, 0){};
\node[u] (v6) at (6, 0){};
\node[u] (w1) at (1.5, 1){};
\node[u] (w2) at (4.5, 1){};
\node[p] (A1) at (3, 0.5){};

  \draw (v0) -- (v6);  
  \draw (v1) -- (w1);
  \draw (v2) -- (w1);    
  \draw (v4) -- (w2);
  \draw (v5) -- (w2);            
  \draw (v0) -- (0, 1);     
   \draw (v6) -- (6, 1);          
   
   \node[below=0.1cm, font=\small] at (v1) {$v_1$};  
   \node[below=0.1cm,font=\small] at (v2) {$v_2$};  
   \node[below=0.1cm,font=\small] at (v4) {$v_{2l-1}$};     
   \node[below=0.1cm, font=\small] at (v5) {$v_{2l}$};     
   \node[below=0.1cm,font=\small] at (v6) {$z_{i+1}$}; 
      \node[below=0.1cm, font=\small] at (v0) {$z_i$};        
      
    \node[above=0.1cm, font=\small] at (w1) {$F_1$};    
     \node[above=0.1cm, font=\small] at (w2) {$w_l$};   
      \node[font=\small] at (A1) {$\ldots$};  
   \node[left=0.7cm,below=0.7cm] at (v4) {\small (iv) Segment $S_4$ (possibly $z_i = z_{i+1}$)};                 
 \end{tikzpicture}}
 \hspace{0.3cm}
\subfigure{\begin{tikzpicture}
[u/.style={fill=black, minimum size =3pt,ellipse,inner sep=1pt},node distance=1.5cm,scale=0.7]

\node[u] (v1) at (1, 0){};
\node[u] (v2) at (2, 0){};

  \draw (v1) -- (v2);  
  \draw (v1) -- (1, 1);
  \draw (v2) -- (2, 1);

   \node[below=0.1cm, font=\small] at (v1) {$z_i$};  
   \node[below=0.1cm,font=\small] at (v2) {$z_{i+1}$}; 
   \node[left=0.7cm,below=0.7cm] at (v2) {\small  (v) the segment $S_5$};       
            
 \end{tikzpicture}}
  \hspace{0.3cm}
\subfigure{ \begin{tikzpicture}
[u/.style={fill=black, minimum size =3pt,ellipse,inner sep=1pt},node distance=1.5cm,scale=0.7]
\node[u] (v0) at (0, 0){};
\node[u] (v1) at (1, 0){};
\node[u] (v2) at (2, 0){};
\node[u] (v3) at (3, 0){};
\node[u] (w1) at (1.5, 1){};

  \draw (v0) -- (v3);  
  \draw (v1) -- (w1);
  \draw (v2) -- (w1);                
  \draw (v0) -- (0, 1);     
   \draw (v3) -- (3, 1);          
   
   \node[below=0.1cm, font=\small] at (v1) {$v_1$};  
   \node[below=0.1cm,font=\small] at (v2) {$v_2$};       
   \node[below=0.1cm,font=\small] at (v3) {$z_{i+1}$}; 
      \node[below=0.1cm, font=\small] at (v0) {$z_i$};        
      
    \node[above=0.1cm, font=\small] at (w1) {$F_1$};    
    \node[left=0.7cm,below=0.7cm] at (v2) {\small (vi) An example $S_4$($l=1$)};          
 \end{tikzpicture} }
\end{center} 
\caption{\small The segments $S_1$--$S_5$. A black vertex denotes a 3-vertex and a white vertex denotes a 2-vertex.} 
\label{S45}
\end{figure*}


Note that 
\[
t(S_1)=1,\qquad 
t(S_2)=2,\qquad 
t(S_3)=3,\qquad 
t(S_4)=l \text{ for some } l\ge1,\qquad 
t(S_5)=0.
\]

\begin{claim}
\label{CL:S}
\begin{enumerate}[(1)]
\item If $A(C)=\emptyset$, then $C$ is an $S_4$-segment with 
$t(C)=\frac{d(C)}{2}$. Hence $d(C)$ is even.

\item If $A(C)=\{z_1\}$, then $C$ is an $S_4$-segment and 
$t(C)=\frac{d(C)-1}{2}$. Hence $d(C)$ is odd.

\item For each $i\in\{1,\dots,5\}$,
\[
t(S_i)\; = \;\left\lfloor \frac{|S_i|-1}{2} \right\rfloor,
\]
where $|S_i|$ denotes the number of vertices in $S_i$.
\end{enumerate}
\end{claim}

With the previous structural claims, we are now ready to estimate $t(C)$ for a $9^+$-cycle $C$.

\begin{claim}
\label{lemma-C11}
If $C$ is a $9^+$-cycle in $G$, then
\[
t(C)\;\le\; d(C)-\left\lceil \frac{d(C)}{2}\right\rceil.
\]
In particular, $t(C)\le d(C)-6$ whenever $d(C)\ge 11$.
\end{claim}

\begin{proof}
If $A(C)=\emptyset$, then $d(C)$ is even and
\[
t(C)=\frac{d(C)}{2}= d(C)-\left\lceil \frac{d(C)}{2} \right\rceil.
\]

Now assume $A(C)=\{z_1,\dots,z_k\}$ with $k\ge 1$.  
For each segment $M_i$, we have
\[
t(M_i)\;=\;\left\lfloor \frac{|M_i|-1}{2} \right\rfloor,
\]
where $|M_i|$ is the number of vertices on $M_i$ since $M_i \in \{S_1, S_2, S_3, S_4, S_5\}$ by Claim \ref{segment-S}.  Hence
\[
t(C)=\sum_{i=1}^k t(M_i)
\;=\;
\sum_{i=1}^k \left\lfloor \frac{|M_i|-1}{2} \right\rfloor
\le
 \left\lfloor \sum_{i=1}^k \frac{|M_i|-1}{2} \right\rfloor
=
\left\lfloor \frac{d(C)}{2}\right\rfloor
=
d(C)-\left\lceil \frac{d(C)}{2}\right\rceil,
\]
since $\sum_{i=1}^k (|M_i|-1) = d(C)$.

If $d(C)\ge 11$, then $\left\lceil \frac{d(C)}{2}\right\rceil\ge 6$, and therefore
\[
t(C)
\le
d(C)-\left\lceil \frac{d(C)}{2} \right\rceil
\le d(C)-6.
\]
\end{proof}

To complete the discharging argument, we need to  extend  $t(C) \leq d(C) - 6$ for $11^+$-cycles $C$ in 
Claim~\ref{lemma-C11} to cycles of lengths $9$ and $10$.

\begin{claim}
\label{lemma-C10}
If $C$ is a cycle with $d(C)\in\{9,10\}$, then $t(C)\le d(C)-6$.
\end{claim}

The proof of Claim~\ref{lemma-C10} is technical and is given in the next section.

\medskip
We now proceed with the discharging method.
For each vertex $x$, define $\omega(x)=2d(x)-6$,  
and for each face $f$, define $\omega(f)=d(f)-6$.  
Thus every element $x\in V(G)\cup F(G)$ has {\it initial charge} $\omega(x)$.
By Euler’s formula $|V(G)|-|E(G)|+|F(G)|=2$, we have
\begin{equation}\label{Eulereq}
\sum_{x\in V(G)\cup F(G)} \omega(x)
=
\sum_{v\in V(G)} (2d(v)-6)
\;+\;
\sum_{f\in F(G)} (d(f)-6)
=
-12.
\end{equation}

We now redistribute charge according to the following rule and let $\omega'(x)$ be the {\it  final charge} of $x$.

\medskip
\noindent\textbf{Discharging Rule}

\medskip \noindent
(R)  
Each $9^+$-face sends a charge of $1$ to each incident $2$-vertex  
and to each adjacent $3$-face.

\medskip

We will show that $\omega'(x)\ge 0$ for every $x\in V(G)\cup F(G)$.

\smallskip
If $x$ is a $2$-vertex, then it receives exactly $1$ from each of its two incident $9^+$-faces by (R), so $\omega'(x)=0$.  
If $x$ is a $3$-vertex, then it receives no charge and  
$\omega'(x)=\omega(x)=0$.

\smallskip
Let $f$ be a face of $G$.  
If $d(f)\ge 9$, then $t(f)\le d(f)-6$ by 
Claim~\ref{lemma-C11} and Claim~\ref{lemma-C10}.  
Thus
\[
\omega'(f)
=
d(f)-6 - t(f)
\ge 0.
\]
If $d(f)=3$, then $f$ receives $1$ from each of its three incident $9^+$-faces, so
\[
\omega'(f)= d(f)-6+3 = 0.
\]

\medskip
Therefore,
\[
0\;\le\;\sum_{x\in V(G)\cup F(G)} \omega'(x)
=
\sum_{x\in V(G)\cup F(G)} \omega(x)
=
-12,
\]
a contradiction.  
This completes the proof of Theorem \ref{main-thm}. \hfill $\Box$

\section{Proof of Claim~\ref{lemma-C10}}

We now complete the proof of Claim~\ref{lemma-C10} by contradiction.  
Assume that 
\[
t(C) \ge d(C)-5.
\]
Since $t(C) \le d(C)-\left\lceil d(C)/2 \right\rceil$ by Claim~\ref{lemma-C11} and  
$9 \le d(C) \le 10$, it follows that 
\[
t(C) = d(C)-5.
\]
By the definitions of the parameters, we first establish the following claim.

\begin{claim}
\label{cl:t(f)}
Let $C$ be a cycle with $9 \le d(C) \le 10$ and let $k = |A(C)|$.  
Let $s_i$ denote the number of $S_i$-segments on $C$, and let  
$p = s_4$.  
Let $l_1,\dots,l_p$ be the number of triangles  in the $p$ $S_4$-segments sharing an edge with $C$, respectively.  
Then the following hold.

\medskip \noindent
\textup{(1)} If $k=0$, then $d(C)=10$ and $G$ contains $H_1$ in Figure~\ref{H234} as a subgraph.

\medskip \noindent
\textup{(2)} If $k=1$, then $d(C)=9$ and $G$ contains $F_{12}$ in Figure~\ref{W12-subgraph} as a subgraph.

\medskip \noindent
\textup{(3)} If $k \ge 2$, then:

\begin{itemize}
\item[(3-1)]  
$t(C) = s_1 + 2s_2 + 3s_3 + (l_1+\cdots+l_p)$;

\item[(3-2)]  
$k = s_1+s_2+s_3+s_4+s_5 
   = s_1+s_2+s_3 + (5 - t(C))$;

\item[(3-3)]  
$5 = k + s_2 + 2s_3 + (l_1+\cdots+l_p)$.
\end{itemize}
\end{claim}

\begin{proof}
Parts (1) and (2) follow directly from Claim~\ref{CL:S}.
  
Assume now that $k \ge 2$.  
By the definitions we have:

\begin{itemize}
\item[(a)] 
$t(C) = s_1 + 2s_2 + 3s_3 + (l_1+\cdots+l_p)$;

\item[(b)] 
$k = s_1 + s_2 + s_3 + s_4 + s_5$;

\item[(c)] 
$d(C) = k + s_1 + 3s_2 + 5s_3 + 2(l_1+\cdots+l_p)$.
\end{itemize}
This immediately yields (3-1) and the first identity in (3-2).  

On the other hand, 
since $d(C) - t(C) = 5$, Equations (a) and (c) imply
\[
5 = k + s_2 + 2s_3 + (l_1+\cdots+l_p),
\]
which gives (3-3).  Now, using (3-1) and (3-3), we obtain
\[
t(C) - 5
  = s_1 + s_2 + s_3 + (l_1+\cdots+l_p) - k,
\]
and thus
\[
k = s_1 + s_2 + s_3 + (l_1+\cdots+l_p) - (t(C)-5),
\]
proving the second identity in (3-2).
\end{proof}

We now apply Claim~\ref{cl:t(f)} to obtain reducible configurations on $10$-cycles and $9$-cycles.

\subsection{Reducible configurations on 10-cycles}

Assume that  $C$ is a $10$-cycle with $t(C)=d(C)-5=5$.

\begin{claim}
$G$ contains $H_i$   in Figure~\ref{H234} as a subgraph for some $i\in\{1,2,3,4\}$.
\end{claim}

\begin{proof}
If $k=|A(C)| \le 1$, then by Claim~\ref{cl:t(f)} we must have $k=0$, and therefore $G$ contains $H_1$ as a subgraph as shown in Figure~\ref{H234}.

Assume now that $k \ge 2$.   
From Claim~\ref{cl:t(f)}(3-2), we have
\[
s_4 = s_5 = 0,
\qquad\text{since}\qquad
s_4 + s_5 = 5 - t(C) = 0,
\]
and from (3-1) and (3-3), we have
\[
5 = s_1 + 2s_2 + 3s_3, 
\qquad\text{and}\qquad
k = s_1 + s_2 + s_3.
\]

\medskip\noindent
\textbf{Case $k=2$.}  
Then $s_2 + 2s_3 = 3$.  
Since $s_2 \le k = 2$, we obtain $s_3=1$ and $s_2=1$, hence $s_1=0$.  
Thus $G$ contains $H_2$ in Figure~\ref{H234}.

\medskip\noindent
\textbf{Case $k=3$.}  
Then $s_2 + 2s_3 = 2$.

If $s_3=1$, then $s_2=0$ and $s_1=2$, yielding the configuration $H_3$ in Figure~\ref{H234}. 

If $s_3=0$, then $s_2=2$ and $s_1=1$, which forces $G$ to contain $T_3$, a contradiction to Claim~\ref{cl:reducible}.

\medskip\noindent
\textbf{Case $k=4$.}  
Then $s_3=0$, $s_2=1$, and $s_1=3$, again forcing $T_3$, a contradiction.

\medskip\noindent
\textbf{Case $k=5$.}  
Then $s_1=5$ and $s_2=s_3=0$, which yields the configuration $H_4$ in Figure~\ref{H234}.

\medskip

Thus $G$ must contain one of $H_1, H_2, H_3,$ or $H_4$   in Figure~\ref{H234} as a subgraph.
\end{proof}

We will show next that none of the configurations $H_1, H_2, H_3,$ or $H_4$ can occur in $G$, completing the contradiction and proving that  
$t(C) \le d(C)-6$ whenever $d(C)=10$.
\begin{figure}[htbp]
\begin{center}
\subfigure{\begin{tikzpicture}
[u/.style={fill=black, minimum size =3pt,ellipse,inner sep=1pt},node distance=1.5cm,scale=1.2]
\node[u] (v1) at (108:1){};
\node[u] (v2) at (72:1){};
\node[u] (v3) at (36:1){};
\node[u] (v4) at (0:1){};
\node[u] (v5) at (324:1){};
\node[u] (v6) at (288:1){};
\node[u] (v7) at (252:1){};
\node[u] (v8) at (216:1){};
\node[u] (v9) at (180:1){};
\node[u] (v10) at (144:1){};
\node[u] (v11) at (90:1.5){};
\node[u] (v12) at (18:1.5){};
\node[u] (v13) at (306:1.5){};
\node[u] (v14) at (234:1.5){};
\node[u] (v15) at (162:1.5){};
\node (C10) at (0, 0){$C_{10}$};

\draw   (0,0) circle[radius=1cm];
  \draw (v1) -- (v11);
  \draw (v2) -- (v11);   
  \draw (v3) -- (v12);
  \draw (v4) -- (v12);  
  \draw (v5) -- (v13);
  \draw (v6) -- (v13);
  \draw (v7) -- (v14);
  \draw (v8) -- (v14); 
  \draw (v9) -- (v15);
  \draw (v10) -- (v15);

   \node[left=0.05cm, above=-0.01cm, font=\scriptsize] at (v1) {$v_1$};  
   \node[ right=0.05cm, above=-0.01cm,  font=\scriptsize] at (v2) {$v_2$};  
   \node[above=0.05cm, font=\scriptsize] at (v3) {$v_3$};     
   \node[right=0.01cm, font=\scriptsize] at (v4) {$v_4$};     
    \node[below=0.2cm,right=0.01cm, font=\scriptsize] at (v5) {$v_5$};    
   \node[below=0.1cm,font=\scriptsize] at (v6) {$v_6$};      
   \node[right=0.05cm, below=0.05cm,  font=\scriptsize] at (v7) {$v_7$};  
   \node[left=-0.1cm, font=\scriptsize] at (v8) {$v_8$};    
    \node[left=-0.05cm,font=\scriptsize] at (v9) {$v_9$};  
   \node[left=0.1cm, above=0.05cm,  font=\scriptsize] at (v10) {$v_{10}$};        
    \node[above=0.02cm, font=\scriptsize] at (v11) {$v_{11}$}; 
    \node[above=0.01cm,  font=\scriptsize] at (v12) {$v_{12}$};        
    \node[below=0.02cm, font=\scriptsize] at (v13) {$v_{13}$};          
     \node[below=0.02cm,  font=\scriptsize] at (v14) {$v_{14}$};        
    \node[above=0.02cm, font=\scriptsize] at (v15) {$v_{15}$};               
    
 \node[below=0.05cm,  font=\scriptsize] at (v1) {\textcolor{blue}{5}};
\node[below=0.05cm,  font=\scriptsize] at (v2) {\textcolor{blue}{5}};  
\node[left=0.05cm,font=\scriptsize] at (v3) {\textcolor{blue}{5}};  
\node[above=0.05cm, left=0.05cm,font=\scriptsize] at (v4) {\textcolor{blue}{5}};  
\node[left=0.05cm,font=\scriptsize] at (v5) {\textcolor{blue}{5}};  
\node[above=0.05cm, font=\scriptsize] at (v6) {\textcolor{blue}{5}};  
\node[above=0.05cm, font=\scriptsize] at (v7) {\textcolor{blue}{5}};  
\node[right=0.05cm, font=\scriptsize] at (v8) {\textcolor{blue}{5}};  
\node[right=0.05cm, font=\scriptsize] at (v9) {\textcolor{blue}{5}};  
\node[right=0.05cm,  font=\scriptsize] at (v10) {\textcolor{blue}{5}}; 
\node[right=0.1cm, font=\scriptsize] at (v11) {\textcolor{blue}{3}}; 
\node[right=0.1cm, font=\scriptsize] at (v12) {\textcolor{blue}{3}}; 
\node[right=0.1cm, font=\scriptsize] at (v13) {\textcolor{blue}{3}}; 
\node[left=0.1cm, font=\scriptsize] at (v14) {\textcolor{blue}{3}}; 
\node[below=0.1cm, font=\scriptsize] at (v15) {\textcolor{blue}{3}}; 
\node[left=0.4cm,below=1cm] at (v6) {\small $H_1$}; 
 \end{tikzpicture}}
\subfigure{
\begin{tikzpicture}
[u/.style={fill=black, minimum size =3pt,ellipse,inner sep=1pt},node distance=1.5cm,scale=1.2]
\node[u] (v1) at (108:1){};
\node[u] (v2) at (72:1){};
\node[u] (v3) at (36:1){};
\node[u] (v4) at (0:1){};
\node[u] (v5) at (324:1){};
\node[u] (v6) at (288:1){};
\node[u] (v7) at (252:1){};
\node[u] (v8) at (216:1){};
\node[u] (v9) at (180:1){};
\node[u] (v10) at (144:1){};
\node[u] (v11) at (90:1.5){};
\node[u] (v12) at (342:1.5){};
\node[u] (v13) at (234:1.5){};
\node (C10) at (0, 0){$C_{10}$};

\draw   (0,0) circle[radius=1cm];
  \draw (v1) -- (v11);
  \draw (v2) -- (v11);   
  \draw (v4) -- (v12);
  \draw (v5) -- (v12);  
  \draw (v7) -- (v13);
  \draw (v8) -- (v13);

\node[u] (w6) at (288:1.4){};
\node[u] (w10) at (144:1.4){};
  \draw (v6) -- (w6);
  \draw (v10) -- (w10);

   \node[left=0.05cm, above=-0.01cm, font=\scriptsize] at (v1) {$v_1$};  
   \node[ right=0.05cm, above=-0.01cm,  font=\scriptsize] at (v2) {$v_2$};  
   \node[above=0.03cm,  right=0.01cm,font=\scriptsize] at (v3) {$v_3$};     
   \node[right=-0.03cm,font=\scriptsize] at (v4) {$v_4$};     
    \node[below=0.2cm,right=-0.05cm, font=\scriptsize] at (v5) {$v_5$};    
   \node[right=0.3cm, below=-0.01cm, font=\scriptsize] at (v6) {$v_6$};      
   \node[right=0.05cm, below=-0.05cm,  font=\scriptsize] at (v7) {$v_7$};  
   \node[left=-0.05cm, font=\scriptsize] at (v8) {$v_8$};    
    \node[left=-0.05cm,font=\scriptsize] at (v9) {$v_9$};  
   \node[left=0.1cm, above=0.05cm,  font=\scriptsize] at (v10) {$v_{10}$};        
    \node[above=0.05cm,  above=-0.01cm, font=\scriptsize] at (v11) {$v_{11}$}; 
    \node[below=0.05cm,  font=\scriptsize] at (v12) {$v_{12}$};        
    \node[below=0.02cm, font=\scriptsize] at (v13) {$v_{13}$};          
    
 \node[below=0.05cm,  font=\scriptsize] at (v1) {\textcolor{blue}{4}};
\node[below=0.05cm,  font=\scriptsize] at (v2) {\textcolor{blue}{5}};  
\node[left=0.05cm,font=\scriptsize] at (v3) {\textcolor{blue}{6}};  
\node[above=0.03cm, left=0.1cm,font=\scriptsize] at (v4) {\textcolor{blue}{5}};  
\node[left=0.05cm,font=\scriptsize] at (v5) {\textcolor{blue}{4}};  
\node[above=0.05cm, font=\scriptsize] at (v6) {\textcolor{blue}{3}};  
\node[above=0.05cm, font=\scriptsize] at (v7) {\textcolor{blue}{4}};  
\node[right=0.05cm, font=\scriptsize] at (v8) {\textcolor{blue}{5}};  
\node[right=0.05cm, font=\scriptsize] at (v9) {\textcolor{blue}{5}};  
\node[right=0.05cm,  font=\scriptsize] at (v10) {\textcolor{blue}{3}}; 
\node[right=0.1cm, font=\scriptsize] at (v11) {\textcolor{blue}{3}}; 
\node[above=0.1cm, font=\scriptsize] at (v12) {\textcolor{blue}{3}}; 
\node[left=0.1cm, font=\scriptsize] at (v13) {\textcolor{blue}{3}}; 
 \node[left=0.4cm,below=1cm]  at (v6) {\small $H_2$};  
 \end{tikzpicture}}
\subfigure{\begin{tikzpicture}
[u/.style={fill=black, minimum size =3pt,ellipse,inner sep=1pt},node distance=1.5cm,scale=1.2]
\node[u] (v1) at (108:1){};
\node[u] (v2) at (72:1){};
\node[u] (v3) at (36:1){};
\node[u] (v4) at (0:1){};
\node[u] (v5) at (324:1){};
\node[u] (v6) at (288:1){};
\node[u] (v7) at (252:1){};
\node[u] (v8) at (216:1){};
\node[u] (v9) at (180:1){};
\node[u] (v10) at (144:1){};
\node[u] (v11) at (90:1.5){};
\node[u] (v12) at (342:1.5){};

\node (C10) at (0, 0){$C_{10}$};

\draw   (0,0) circle[radius=1cm];
  \draw (v1) -- (v11);
  \draw (v2) -- (v11);   
  \draw (v4) -- (v12);
  \draw (v5) -- (v12);     
  
\node[u] (w6) at (288:1.4){};
\node[u] (w8) at (216:1.4){};
\node[u] (w10) at (144:1.4){};

  \draw (v6) -- (w6);
  \draw (v8) -- (w8);  
  \draw (v10) -- (w10);
    
       \node[left=0.05cm, above=-0.01cm, font=\scriptsize] at (v1) {$v_1$};  
   \node[ right=0.05cm, above=-0.01cm,  font=\scriptsize] at (v2) {$v_2$};  
   \node[above=0.05cm, right=0.01cm, font=\scriptsize] at (v3) {$v_3$};     
   \node[right=-0.05cm,font=\scriptsize] at (v4) {$v_4$};     
    \node[below=0.2cm,right=-0.08cm, font=\scriptsize] at (v5) {$v_5$};    
   \node[right=0.3cm, below=-0.01cm, font=\scriptsize] at (v6) {$v_6$};      
   \node[right=0.05cm, below=0.01cm,  font=\scriptsize] at (v7) {$v_7$};  
   \node[left=-0.05cm, font=\scriptsize] at (v8) {$v_8$};    
    \node[left=-0.02cm,font=\scriptsize] at (v9) {$v_9$};  
   \node[left=0.1cm, above=0.05cm,  font=\scriptsize] at (v10) {$v_{10}$};        
    \node[above=0.02cm, font=\scriptsize] at (v11) {$v_{11}$}; 
    \node[below=0.03cm,font=\scriptsize] at (v12) {$v_{12}$};        
    
 \node[below=0.05cm,  font=\scriptsize] at (v1) {\textcolor{blue}{4}};
\node[below=0.05cm,  font=\scriptsize] at (v2) {\textcolor{blue}{5}};  
\node[left=0.05cm,font=\scriptsize] at (v3) {\textcolor{blue}{6}};  
\node[above=0.05cm, left=0.05cm,font=\scriptsize] at (v4) {\textcolor{blue}{5}};  
\node[left=0.05cm,font=\scriptsize] at (v5) {\textcolor{blue}{4}};  
\node[above=0.05cm, font=\scriptsize] at (v6) {\textcolor{blue}{3}};  
\node[above=0.05cm, font=\scriptsize] at (v7) {\textcolor{blue}{4}};  
\node[right=0.05cm, font=\scriptsize] at (v8) {\textcolor{blue}{3}};  
\node[right=0.05cm, font=\scriptsize] at (v9) {\textcolor{blue}{4}};  
\node[right=0.05cm,  font=\scriptsize] at (v10) {\textcolor{blue}{3}}; 
\node[right=0.1cm, font=\scriptsize] at (v11) {\textcolor{blue}{3}}; 
\node[above=0.1cm, font=\scriptsize] at (v12) {\textcolor{blue}{3}}; 

 \node[left=0.4cm,below=1cm] at (v6) { \small  $H_3$};  
 \end{tikzpicture}}
\subfigure{\begin{tikzpicture}
[u/.style={fill=black, minimum size =3pt,ellipse,inner sep=1pt},node distance=1.5cm,scale=1.2]
\node[u] (v1) at (108:1){};
\node[u] (v2) at (72:1){};
\node[u] (v3) at (36:1){};
\node[u] (v4) at (0:1){};
\node[u] (v5) at (324:1){};
\node[u] (v6) at (288:1){};
\node[u] (v7) at (252:1){};
\node[u] (v8) at (216:1){};
\node[u] (v9) at (180:1){};
\node[u] (v10) at (144:1){};
\node (C10) at (0, 0){$C_{10}$};

\draw   (0,0) circle[radius=1cm];

 \node[u] (w2) at (72:1.4){};
\node[u] (w4) at (0:1.4){}; 
\node[u] (w6) at (288:1.4){};
\node[u] (w8) at (216:1.4){};
\node[u] (w10) at (144:1.4){};
  \draw (v2) -- (w2);
  \draw (v4) -- (w4);   
  \draw (v6) -- (w6);
  \draw (v8) -- (w8);  
  \draw (v10) -- (w10);
    
    \node[left=0.05cm, above=-0.01cm, font=\scriptsize] at (v1) {$v_1$};  
   \node[left=0.1cm, above=0.01cm,  font=\scriptsize] at (v2) {$v_2$};  
   \node[above=0.05cm, right=0.01cm, font=\scriptsize] at (v3) {$v_3$};     
   \node[right=0.2cm, above=-0.01cm, font=\scriptsize] at (v4) {$v_4$};     
    \node[below=0.1cm,right=-0.01cm, font=\scriptsize] at (v5) {$v_5$};    
   \node[right=0.25cm, below=-0.03cm, font=\scriptsize] at (v6) {$v_6$};      
   \node[right=0.05cm, below=0.01cm,  font=\scriptsize] at (v7) {$v_7$};  
   \node[left=-0.01cm, font=\scriptsize] at (v8) {$v_8$};    
    \node[left=-0.01cm,font=\scriptsize] at (v9) {$v_9$};  
   \node[left=0.1cm, above=0.05cm,  font=\scriptsize] at (v10) {$v_{10}$};

 \node[below=0.05cm,  font=\scriptsize] at (v1) {\textcolor{blue}{4}};
\node[below=0.05cm,  font=\scriptsize] at (v2) {\textcolor{blue}{3}};  
\node[left=0.05cm,font=\scriptsize] at (v3) {\textcolor{blue}{4}};  
\node[above=-0.05cm, left=0.05cm,font=\scriptsize] at (v4) {\textcolor{blue}{3}};  
\node[left=0.05cm,font=\scriptsize] at (v5) {\textcolor{blue}{4}};  
\node[above=0.05cm, font=\scriptsize] at (v6) {\textcolor{blue}{3}};  
\node[above=0.05cm, font=\scriptsize] at (v7) {\textcolor{blue}{4}};  
\node[right=0.02cm, font=\scriptsize] at (v8) {\textcolor{blue}{3}};  
\node[right=0.02cm, font=\scriptsize] at (v9) {\textcolor{blue}{4}};  
\node[right=0.05cm,  font=\scriptsize] at (v10) {\textcolor{blue}{3}}; 

\node[left=0.4cm,below=1cm] at (v6) {\small $H_4$}; 
 \end{tikzpicture} } 
\end{center} 
\caption{Reducible configurations on $10$-cycle.
The number at each vertex  denotes the number of available colors after coloring $(G-V(H_i))^2$.} 
\label{H234}
\end{figure}
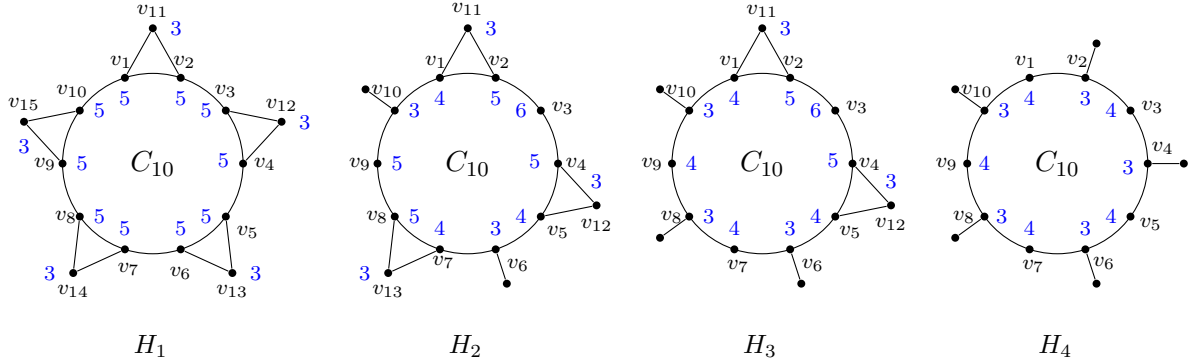


\begin{claim}\label{lem-H234}
For each $i \in \{1,2,3,4\}$, the graph $G$ does not contain $H_i$ as a subgraph.
\end{claim}
\noindent {\bf \it Proof.}
Suppose to the contrary that $G$ contains $H_i$ as a subgraph for some $i \in \{1,2,3,4\}$.  
For each $i \in \{1,2,3,4\}$, let $G_1 = G - V(H_i)$.  
Since no vertex of $H_i$ has two neighbors outside $H_i$, for any two vertices $u,v \in V(G_1)$, if $d_G(u,v)\le 2$ then $d_{G_1}(u,v)\le 2$.  
Moreover, because $G$ contains no $4$- to $8$-cycles, any two $3$-vertices of $H_i$ that do not lie on the $10$-cycle have distance at least $3$ in $G$, so their colors never conflict.

By the minimality of $G$, the graph $G_1^2$ admits an $L$-coloring $\phi$.  
We will show that $\phi$ extends to $H_i$, producing an $L$-coloring of $G^2$, contradicting that $G$ is a counterexample.

For each $v \in V(H_i)$, define
\[
L_{H_i}(v)
   = L(v)\setminus\{\phi(x) : vx \in E(G^2)\ \text{and}\ x \notin V(H_i)\}.
\]

\begin{enumerate}[(a)]
\item \textbf{Case $H_i = H_1$.}

From the structure of $H_1$  and as indicated in Figure~\ref{H234}, we have:
\[
|L_{H_1}(v_i)| \ge 
\begin{cases}
3, & i \in \{11,\dots,15\},\\[2mm]
5, & i \in \{1,\dots,10\}.
\end{cases}
\]

We first extend $\phi$ to $v_9,v_{10}$, then to $v_7,v_8$, and finally to $v_5,v_6$, choosing colors so that each of $v_{13},v_{14},v_{15}$ always retains at least one available color regardless of later choices.

Since $|L_{H_1}(v_9)|, |L_{H_1}(v_{10})|\ge 5$ and $|L_{H_1}(v_{15})|\ge 3$, we may choose distinct colors
\[
c_9 \in L_{H_1}(v_9), \qquad c_{10} \in L_{H_1}(v_{10})
\]
so that $|L_{H_1}(v_{15}) \setminus \{c_9,c_{10}\}| \ge 3$.  
Color $v_9$ and $v_{10}$ with $c_9$ and $c_{10}$, respectively.  
The resulting available list sizes are shown in  Figure~\ref{C10-five-Steps} (1).

Applying the same argument, we color $v_7,v_8$ and then $v_5,v_6$ so that the remaining available list sizes match those in  Figure~\ref{C10-five-Steps} (2).

Next, observe that the subgraph induced by $\{v_1,v_2,v_3,v_4,v_{11},v_{12}\}$ is $J_2$ (Figure~\ref{fig-J12}), and their list sizes satisfy Lemma~\ref{lem-J2}(2).  
By Lemma~\ref{lem-J2}(2), we can color these six vertices.

Finally, color $v_{13}$, then $v_{14}$, and then $v_{15}$.  
This yields an $L$-coloring of $G^2$, a contradiction.

\begin{figure*}[htbp]
\begin{multicols}{2}
\begin{center}
\begin{tikzpicture}
[u/.style={fill=black, minimum size =3pt,ellipse,inner sep=1pt},node distance=1.5cm,scale=1.5]
\node[u] (v1) at (108:1){};
\node[u] (v2) at (72:1){};
\node[u] (v3) at (36:1){};
\node[u] (v4) at (0:1){};
\node[u] (v5) at (324:1){};
\node[u] (v6) at (288:1){};
\node[u] (v7) at (252:1){};
\node[u] (v8) at (216:1){};
\node[u] (v9) at (180:1){};
\node[u] (v10) at (144:1){};
\node[u] (v11) at (90:1.5){};
\node[u] (v12) at (18:1.5){};
\node[u] (v13) at (306:1.5){};
\node[u] (v14) at (234:1.5){};
\node[u] (v15) at (162:1.5){};
\node (C10) at (0, 0){$C_{10}$};

\draw   (0,0) circle[radius=1cm];
  \draw (v1) -- (v11);
  \draw (v2) -- (v11);   
  \draw (v3) -- (v12);
  \draw (v4) -- (v12);  
  \draw (v5) -- (v13);
  \draw (v6) -- (v13);
  \draw (v7) -- (v14);
  \draw (v8) -- (v14); 
  \draw (v9) -- (v15);
  \draw (v10) -- (v15);

   \node[left=0.05cm, above=0.05cm, font=\small] at (v1) {$v_1$};  
   \node[ right=0.05cm, above=0.05cm,  font=\small] at (v2) {$v_2$};  
   \node[above=0.05cm, font=\small] at (v3) {$v_3$};     
   \node[right=0.05cm,font=\small] at (v4) {$v_4$};     
    \node[below=0.2cm,right=0.01cm, font=\small] at (v5) {$v_5$};    
   \node[below=0.1cm,font=\small] at (v6) {$v_6$};      
   \node[right=0.05cm, below=0.05cm,  font=\small] at (v7) {$v_7$};  
   \node[left=0.1cm, font=\small] at (v8) {$v_8$};    
    \node[left=0.05cm,font=\small] at (v9) {$v_9$};  
   \node[left=0.1cm, above=0.05cm,  font=\small] at (v10) {$v_{10}$};        
    \node[above=0.05cm, font=\small] at (v11) {$v_{11}$}; 
    \node[above=0.05cm,  font=\small] at (v12) {$v_{12}$};        
    \node[below=0.05cm, font=\small] at (v13) {$v_{13}$};          
     \node[below=0.05cm,  font=\small] at (v14) {$v_{14}$};        
    \node[above=0.05cm, font=\small] at (v15) {$v_{15}$};               
    
 \node[below=0.05cm,  font=\scriptsize] at (v1) {\textcolor{blue}{3}};
\node[below=0.05cm,  font=\scriptsize] at (v2) {\textcolor{blue}{4}};  
\node[left=0.05cm,font=\scriptsize] at (v3) {\textcolor{blue}{5}};  
\node[above=0.05cm, left=0.05cm,font=\scriptsize] at (v4) {\textcolor{blue}{5}};  
\node[left=0.05cm,font=\scriptsize] at (v5) {\textcolor{blue}{5}};  
\node[above=0.05cm, font=\scriptsize] at (v6) {\textcolor{blue}{5}};  
\node[above=0.05cm, font=\scriptsize] at (v7) {\textcolor{blue}{4}};  
\node[right=0.05cm, font=\scriptsize] at (v8) {\textcolor{blue}{3}};  
\node[right=0.05cm, font=\scriptsize] at (v9) {};  
\node[right=0.05cm,  font=\scriptsize] at (v10) {}; 
\node[right=0.1cm, font=\scriptsize] at (v11) {\textcolor{blue}{2}}; 
\node[right=0.1cm, font=\scriptsize] at (v12) {\textcolor{blue}{3}}; 
\node[right=0.1cm, font=\scriptsize] at (v13) {\textcolor{blue}{3}}; 
\node[left=0.1cm, font=\scriptsize] at (v14) {\textcolor{blue}{2}}; 
\node[below=0.1cm, font=\scriptsize] at (v15) {\textcolor{blue}{3}}; 
 \end{tikzpicture}
        \vfill {\small (1)  The list sizes after coloring $v_9, v_{10}$\ \ \ \ }  
\end{center} 
\par
\begin{center}
\begin{tikzpicture}
[u/.style={fill=black, minimum size =3pt,ellipse,inner sep=1pt},node distance=1.5cm,scale=1.5]
\node[u] (v1) at (108:1){};
\node[u] (v2) at (72:1){};
\node[u] (v3) at (36:1){};
\node[u] (v4) at (0:1){};
\node[u] (v5) at (324:1){};
\node[u] (v6) at (288:1){};
\node[u] (v7) at (252:1){};
\node[u] (v8) at (216:1){};
\node[u] (v9) at (180:1){};
\node[u] (v10) at (144:1){};
\node[u] (v11) at (90:1.5){};
\node[u] (v12) at (18:1.5){};
\node[u] (v13) at (306:1.5){};
\node[u] (v14) at (234:1.5){};
\node[u] (v15) at (162:1.5){};
\node (C10) at (0, 0){$C_{10}$};

\draw   (0,0) circle[radius=1cm];
  \draw (v1) -- (v11);
  \draw (v2) -- (v11);   
  \draw (v3) -- (v12);
  \draw (v4) -- (v12);  
  \draw (v5) -- (v13);
  \draw (v6) -- (v13);
  \draw (v7) -- (v14);
  \draw (v8) -- (v14); 
  \draw (v9) -- (v15);
  \draw (v10) -- (v15);

   \node[left=0.05cm, above=0.05cm, font=\small] at (v1) {$v_1$};  
   \node[ right=0.05cm, above=0.05cm,  font=\small] at (v2) {$v_2$};  
   \node[above=0.05cm, font=\small] at (v3) {$v_3$};     
   \node[right=0.05cm,font=\small] at (v4) {$v_4$};     
    \node[below=0.2cm,right=0.01cm, font=\small] at (v5) {$v_5$};    
   \node[below=0.1cm,font=\small] at (v6) {$v_6$};      
   \node[right=0.05cm, below=0.05cm,  font=\small] at (v7) {$v_7$};  
   \node[left=0.1cm, font=\small] at (v8) {$v_8$};    
    \node[left=0.05cm,font=\small] at (v9) {$v_9$};  
   \node[left=0.1cm, above=0.05cm,  font=\small] at (v10) {$v_{10}$};        
    \node[above=0.05cm, font=\small] at (v11) {$v_{11}$}; 
    \node[above=0.05cm,  font=\small] at (v12) {$v_{12}$};        
    \node[below=0.05cm, font=\small] at (v13) {$v_{13}$};          
     \node[below=0.05cm,  font=\small] at (v14) {$v_{14}$};        
    \node[above=0.05cm, font=\small] at (v15) {$v_{15}$};               
    
 \node[below=0.05cm,  font=\scriptsize] at (v1) {\textcolor{blue}{3}};
\node[below=0.05cm,  font=\scriptsize] at (v2) {\textcolor{blue}{4}};  
\node[left=0.05cm,font=\scriptsize] at (v3) {\textcolor{blue}{4}};  
\node[above=0.05cm, left=0.05cm,font=\scriptsize] at (v4) {\textcolor{blue}{3}};  
\node[left=0.05cm,font=\scriptsize] at (v5) {};    
\node[above=0.05cm, font=\scriptsize] at (v6) {};   
\node[above=0.05cm, font=\scriptsize] at (v7) {};   
\node[right=0.05cm, font=\scriptsize] at (v8) {};  
\node[right=0.05cm, font=\scriptsize] at (v9) {};  
\node[right=0.05cm,  font=\scriptsize] at (v10) {}; 
\node[right=0.1cm, font=\scriptsize] at (v11) {\textcolor{blue}{2}}; 
\node[right=0.1cm, font=\scriptsize] at (v12) {\textcolor{blue}{2}}; 
\node[right=0.1cm, font=\scriptsize] at (v13) {\textcolor{blue}{2}}; 
\node[left=0.1cm, font=\scriptsize] at (v14) {\textcolor{blue}{1}}; 
\node[below=0.1cm, font=\scriptsize] at (v15) {\textcolor{blue}{2}}; 
 \end{tikzpicture}
        \vfill {\small (2)  The list sizes after coloring $v_9, v_{10}, v_7,v_8, v_5, v_6$} 
\end{center}
\end{multicols} 
\caption{
The graph $H_1$: the list sizes after coloring some vertices $H_1$} 
\label{C10-five-Steps}
\end{figure*}

\item \textbf{Case $H_i = H_2$.}

As indicated in Figure~\ref{H234}, we have:
\[
|L_{H_2}(v_i)| \ge 
\begin{cases}
3, & i = 6, 10, 11, 12, 13, \\
4, & i = 1, 5, 7, \\
5, & i = 2, 4, 8, 9, \\
6, & i = 3.
\end{cases}
\]

Using an argument similar to the previous case, there exist distinct colors  
$c_7 \in L_{H_2}(v_7)$ and $c_8 \in L_{H_2}(v_8)$ such that  
\(|L_{H_2}(v_{13}) \setminus \{c_7, c_8\}| \ge 3\).  
Color \(v_7\) and \(v_8\) with \(c_7\) and \(c_8\), respectively.

Since \(|L_{H_2}(v_1)| \ge 4\) and \(|L_{H_2}(v_{11})| \ge 3\), there is a color  
\(\alpha \in L_{H_2}(v_1)\) such that  
\(|L_{H_2}(v_{11}) \setminus \{\alpha\}| \ge 3\).

Next, color \(v_9\) with a color  
\(c_9 \in L_{H_2}(v_9) \setminus \{c_7, c_8, \alpha\}\),  
and color \(v_{10}\) with a color  
\(c_{10} \in L_{H_2}(v_{10}) \setminus \{c_8, c_9\}\).  
According to the list sizes indicated in Figure~\ref{H2-color} (a), we can then color 
\(v_6, v_5, v_{12}, v_{13}\) in this order.

Let \(L'_{H_2}(v_i)\) denote the list of available colors at \(v_i\) for  
\(i = 1,2,3,4,11\).  
As shown in Figure~\ref{H2-color} (b),
\[
|L'_{H_2}(v_1)| \ge 2,\quad 
|L'_{H_2}(v_{11})| \ge 2,\quad
|L'_{H_2}(v_2)| \ge 4,\quad
|L'_{H_2}(v_3)| \ge 4,\quad
|L'_{H_2}(v_4)| \ge 2.
\]
Moreover, by the choice of \(c_9\) and \(c_{10}\),  
we can conclude that $L'_{H_2}(v_1) \neq L'_{H_2}(v_{11})$ if $|L'_{H_2}(v_1)| = 2$.
In fact, if $c_{10} \neq \alpha$, then 
$\alpha \in L_{H_2}'(v_1) \setminus L_{H_2}'(v_{11})$. So, we have that $L'_{H_2}(v_1) \neq L'_{H_2}(v_{11})$.
 If $c_{10} = \alpha$, then 
$|L_{H_2}'(v_{11})| = 3$, so
 we have that $L'_{H_2}(v_1) \neq L'_{H_2}(v_{11})$ if $|L'_{H_2}(v_1)| = 2$.

Thus, there exists a color \(\beta \in L'_{H_2}(v_{11})\) such that  
\(|L'_{H_2}(v_1)\setminus\{\beta\}| \ge 2\).  
Color \(v_{11}\) with \(\beta\).
Let \(L''_{H_2}(v_i)\) be the list of available colors at  
\(v_i\) for \(i = 1,2,3,4\).  
As indicated in Figure~\ref{H2-color} (c), we have
\[
|L''_{H_2}(v_1)| \ge 2,\quad
|L''_{H_2}(v_2)| \ge 3,\quad
|L''_{H_2}(v_3)| \ge 3,\quad
|L''_{H_2}(v_4)| \ge 2.
\]

By Lemma~\ref{lem-key-K4-edge}, we may further extend \(\phi\) to  
\(v_1, v_2, v_3, v_4\), yielding an \(L\)-coloring of \(G^2\),  
a contradiction.
\begin{figure*}[htbp]
\begin{multicols}{3}
\begin{center}
\begin{tikzpicture}
[u/.style={fill=black, minimum size =3pt,ellipse,inner sep=1pt},node distance=1.5cm,scale=1.5]
\node[u] (v1) at (108:1){};
\node[u] (v2) at (72:1){};
\node[u] (v3) at (36:1){};
\node[u] (v4) at (0:1){};
\node[u] (v5) at (324:1){};
\node[u] (v6) at (288:1){};
\node[u] (v7) at (252:1){};
\node[u] (v8) at (216:1){};
\node[u] (v9) at (180:1){};
\node[u] (v10) at (144:1){};
\node[u] (v11) at (90:1.5){};
\node[u] (v12) at (342:1.5){};
\node[u] (v13) at (234:1.5){};
\node (C10) at (0, 0){$C_{10}$};

\draw   (0,0) circle[radius=1cm];
  \draw (v1) -- (v11);
  \draw (v2) -- (v11);   
  \draw (v4) -- (v12);
  \draw (v5) -- (v12);  
  \draw (v7) -- (v13);
  \draw (v8) -- (v13);

\node[u] (w6) at (288:1.4){};
\node[u] (w10) at (144:1.4){};
  \draw (v6) -- (w6);
  \draw (v10) -- (w10);
    
   \node[left=0.05cm, above=0.05cm, font=\small] at (v1) {$v_1$};  
   \node[ right=0.05cm, above=0.05cm,  font=\small] at (v2) {$v_2$};  
   \node[above=0.05cm, font=\small] at (v3) {$v_3$};     
   \node[right=0.05cm,font=\small] at (v4) {$v_4$};     
    \node[below=0.2cm,right=0.01cm, font=\small] at (v5) {$v_5$};    
   \node[right=0.3cm, below=0.01cm, font=\small] at (v6) {$v_6$};      
   \node[right=0.05cm, below=0.05cm,  font=\small] at (v7) {$v_7$};  
   \node[left=0.1cm, font=\small] at (v8) {$v_8$};    
    \node[left=0.05cm,font=\small] at (v9) {$v_9$};  
   \node[left=0.1cm, above=0.05cm,  font=\small] at (v10) {$v_{10}$};        
    \node[above=0.05cm, font=\small] at (v11) {$v_{11}$}; 
    \node[below=0.05cm,  font=\small] at (v12) {$v_{12}$};        
    \node[below=0.05cm, font=\small] at (v13) {$v_{13}$};          
    
 \node[below=0.05cm,  font=\scriptsize] at (v1) {\textcolor{blue}{4}};
\node[below=0.05cm,  font=\scriptsize] at (v2) {\textcolor{blue}{5}};  
\node[left=0.05cm,font=\scriptsize] at (v3) {\textcolor{blue}{6}};  
\node[above=0.05cm, left=0.05cm,font=\scriptsize] at (v4) {\textcolor{blue}{5}};  
\node[left=0.05cm,font=\scriptsize] at (v5) {\textcolor{blue}{3}};  
\node[above=0.05cm, font=\scriptsize] at (v6) {\textcolor{blue}{1}};  
\node[right=0.05cm, font=\scriptsize] at (v9) {\textcolor{blue}{3}};  
\node[right=0.05cm,  font=\scriptsize] at (v10) {\textcolor{blue}{2}}; 
\node[right=0.1cm, font=\scriptsize] at (v11) {\textcolor{blue}{3}}; 
\node[above=0.1cm, font=\scriptsize] at (v12) {\textcolor{blue}{3}}; 
\node[left=0.1cm, font=\scriptsize] at (v13) {\textcolor{blue}{3}}; 
 \end{tikzpicture}
        \vfill {\small (a) The list sizes after coloring $v_7,v_8$\  \ \ \  \ \ } 
        \end{center}
\par
\begin{center}
\begin{tikzpicture}
[u/.style={fill=black, minimum size =3pt,ellipse,inner sep=1pt},node distance=1.5cm,scale=1.5]
\node[u] (v1) at (108:1){};
\node[u] (v2) at (72:1){};
\node[u] (v3) at (36:1){};
\node[u] (v4) at (0:1){};
\node[u] (v5) at (324:1){};
\node[u] (v6) at (288:1){};
\node[u] (v7) at (252:1){};
\node[u] (v8) at (216:1){};
\node[u] (v9) at (180:1){};
\node[u] (v10) at (144:1){};
\node[u] (v11) at (90:1.5){};
\node[u] (v12) at (342:1.5){};
\node[u] (v13) at (234:1.5){};
\node (C10) at (0, 0){$C_{10}$};

\draw   (0,0) circle[radius=1cm];
  \draw (v1) -- (v11);
  \draw (v2) -- (v11);   
  \draw (v4) -- (v12);
  \draw (v5) -- (v12);  
  \draw (v7) -- (v13);
  \draw (v8) -- (v13);

\node[u] (w6) at (288:1.4){};
\node[u] (w10) at (144:1.4){};
  \draw (v6) -- (w6);
  \draw (v10) -- (w10);
    
   \node[left=0.05cm, above=0.05cm, font=\small] at (v1) {$v_1$};  
   \node[ right=0.05cm, above=0.05cm,  font=\small] at (v2) {$v_2$};  
   \node[above=0.05cm, font=\small] at (v3) {$v_3$};     
   \node[right=0.05cm,font=\small] at (v4) {$v_4$};     
    \node[below=0.2cm,right=0.01cm, font=\small] at (v5) {$v_5$};    
   \node[right=0.3cm, below=0.01cm, font=\small] at (v6) {$v_6$};      
   \node[right=0.05cm, below=0.05cm,  font=\small] at (v7) {$v_7$};  
   \node[left=0.1cm, font=\small] at (v8) {$v_8$};    
    \node[left=0.05cm,font=\small] at (v9) {$v_9$};  
   \node[left=0.1cm, above=0.05cm,  font=\small] at (v10) {$v_{10}$};        
    \node[above=0.05cm, font=\small] at (v11) {$v_{11}$}; 
    \node[below=0.05cm,  font=\small] at (v12) {$v_{12}$};        
    \node[below=0.05cm, font=\small] at (v13) {$v_{13}$};          
    
 \node[below=0.05cm,  font=\scriptsize] at (v1) {\textcolor{blue}{2}};
\node[below=0.05cm,  font=\scriptsize] at (v2) {\textcolor{blue}{4}};  
\node[left=0.05cm,font=\scriptsize] at (v3) {\textcolor{blue}{4}};  
\node[above=0.05cm, left=0.05cm,font=\scriptsize] at (v4) {\textcolor{blue}{2}};  
\node[right=0.1cm, font=\scriptsize] at (v11) {\textcolor{blue}{2}}; 

 \end{tikzpicture}
 \vfill {\small (b)  The list sizes  after coloring $v_i$ for $i =8,10,13,5,6,12$} 
 \end{center}
       \par
\begin{center}
\begin{tikzpicture}
[u/.style={fill=black, minimum size =3pt,ellipse,inner sep=1pt},node distance=1.5cm,scale=1.5]
\node[u] (v1) at (108:1){};
\node[u] (v2) at (72:1){};
\node[u] (v3) at (36:1){};
\node[u] (v4) at (0:1){};
\node[u] (v5) at (324:1){};
\node[u] (v6) at (288:1){};
\node[u] (v7) at (252:1){};
\node[u] (v8) at (216:1){};
\node[u] (v9) at (180:1){};
\node[u] (v10) at (144:1){};
\node[u] (v11) at (90:1.5){};
\node[u] (v12) at (342:1.5){};
\node[u] (v13) at (234:1.5){};
\node (C10) at (0, 0){$C_{10}$};

\draw   (0,0) circle[radius=1cm];
  \draw (v1) -- (v11);
  \draw (v2) -- (v11);   
  \draw (v4) -- (v12);
  \draw (v5) -- (v12);  
  \draw (v7) -- (v13);
  \draw (v8) -- (v13);

\node[u] (w6) at (288:1.4){};
\node[u] (w10) at (144:1.4){};
  \draw (v6) -- (w6);
  \draw (v10) -- (w10);
    
   \node[left=0.05cm, above=0.05cm, font=\small] at (v1) {$v_1$};  
   \node[ right=0.05cm, above=0.05cm,  font=\small] at (v2) {$v_2$};  
   \node[above=0.05cm, font=\small] at (v3) {$v_3$};     
   \node[right=0.05cm,font=\small] at (v4) {$v_4$};     
    \node[below=0.2cm,right=0.01cm, font=\small] at (v5) {$v_5$};    
   \node[right=0.3cm, below=0.01cm, font=\small] at (v6) {$v_6$};      
   \node[right=0.05cm, below=0.05cm,  font=\small] at (v7) {$v_7$};  
   \node[left=0.1cm, font=\small] at (v8) {$v_8$};    
    \node[left=0.05cm,font=\small] at (v9) {$v_9$};  
   \node[left=0.1cm, above=0.05cm,  font=\small] at (v10) {$v_{10}$};        
    \node[above=0.05cm, font=\small] at (v11) {$v_{11}$}; 
    \node[below=0.05cm,  font=\small] at (v12) {$v_{12}$};        
    \node[below=0.05cm, font=\small] at (v13) {$v_{13}$};          
    
 \node[below=0.05cm,  font=\scriptsize] at (v1) {\textcolor{blue}{2}};
\node[below=0.05cm,  font=\scriptsize] at (v2) {\textcolor{blue}{2}};  
\node[left=0.05cm,font=\scriptsize] at (v3) {\textcolor{blue}{3}};  
\node[above=0.05cm, left=0.05cm,font=\scriptsize] at (v4) {\textcolor{blue}{2}};  
 \end{tikzpicture}
 \vfill {\small (c) The list sizes after coloring $v_{11}$\ \ \ \ \ \ }
 \end{center}
       \end{multicols}
\caption{
The graph $H_2$: The list sizes after coloring some vertices $H_2$} 
\label{H2-color}
\end{figure*}
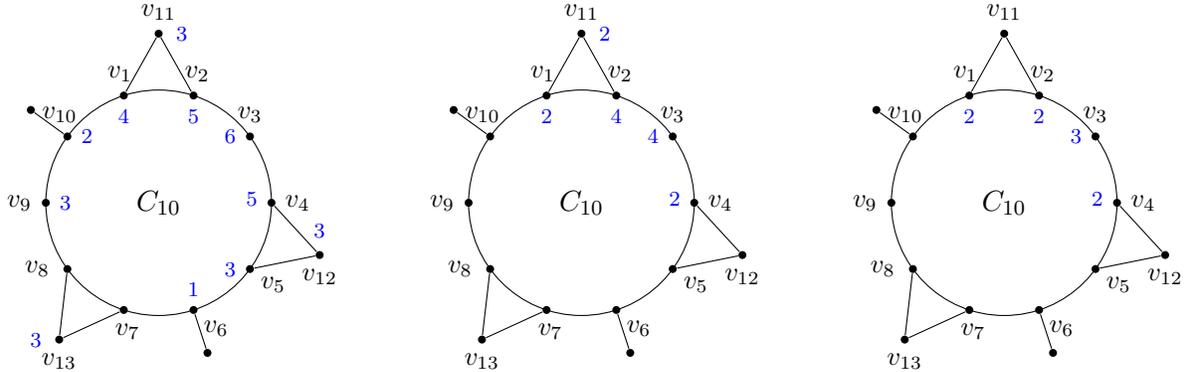

\item \textbf{Case  $H_i = H_3$.}

As indicated in Figure~\ref{H234}, we have:
\[
|L_{H_3}(v_i)| \ge 
\begin{cases}
3, & i = 6, 8, 10, 11, 12, \\
4, & i = 1, 5, 7, 9, \\
5, & i = 2, 4, \\
6, & i = 3.
\end{cases}
\]

Using an argument similar to the case \(H_i = H_2\), there exists a color 
\(\alpha \in L_{H_3}(v_1)\) such that 
\(|L_{H_3}(v_{11}) \setminus \{\alpha\}| \ge 3\).
First color \(v_9\) with a color  
\(c_9 \in L_{H_3}(v_9) \setminus \{\alpha\}\),  
and then color \(v_{10}\) with  
\(c_{10} \in L_{H_3}(v_{10}) \setminus \{\alpha, c_9\}\).  
As indicated by the list sizes in Figure~\ref{H3-color}(a), we can then color  
\(v_8, v_7, v_6, v_5, v_{12}\) in this order.

Let \(L'_{H_3}(v_i)\) denote the list of available colors at  
\(v_i\) for \(i = 1,2,3,4,11\) after coloring  $v_i$ for $i \in \{5,\dots, 10\}\cup \{12\}$.
As shown in Figure~\ref{H3-color}(b),
\[
|L'_{H_3}(v_1)| \ge 2,\qquad
|L'_{H_3}(v_{11})| \ge 2,\qquad
|L'_{H_3}(v_2)| \ge 4,\qquad
|L'_{H_3}(v_3)| \ge 4,\qquad
|L'_{H_3}(v_4)| \ge 2.
\]

Moreover, by the choice of \(c_9\) and \(c_{10}\),
\[
|L'_{H_3}(v_{11}) \setminus \{\alpha\}| 
= |L_{H_3}(v_{11}) \setminus \{\alpha\}| \ge 2,
\qquad\text{and}\qquad
\alpha \in L'_{H_3}(v_1)
   = L_{H_3}(v_1)\setminus \{c_9, c_{10}\}.
\]

Applying the same argument as in the case \(H_i = H_2\),  
we can extend the coloring to all vertices of \(H_3\), yielding an  
\(L\)-coloring of \(G^2\), a contradiction.

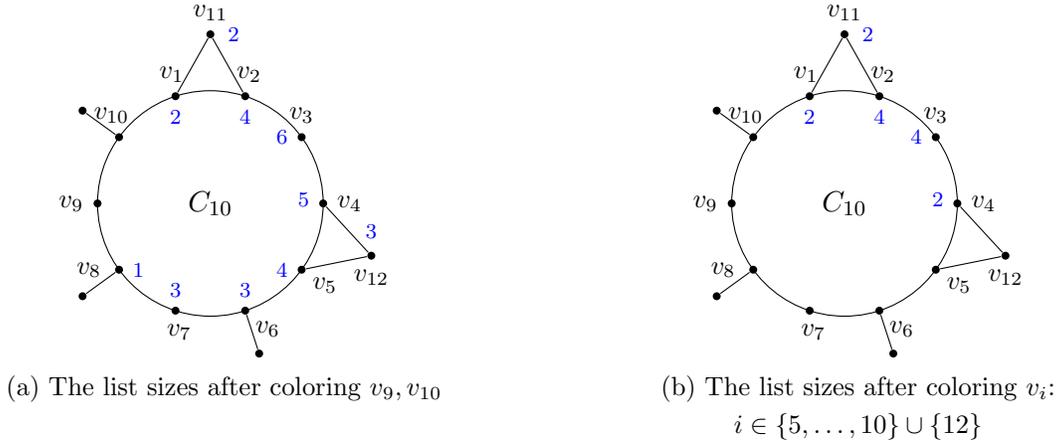
\begin{figure*}[htbp]
\begin{multicols}{2}
\begin{center}
\begin{tikzpicture}
[u/.style={fill=black, minimum size =3pt,ellipse,inner sep=1pt},node distance=1.5cm,scale=1.5]
\node[u] (v1) at (108:1){};
\node[u] (v2) at (72:1){};
\node[u] (v3) at (36:1){};
\node[u] (v4) at (0:1){};
\node[u] (v5) at (324:1){};
\node[u] (v6) at (288:1){};
\node[u] (v7) at (252:1){};
\node[u] (v8) at (216:1){};
\node[u] (v9) at (180:1){};
\node[u] (v10) at (144:1){};
\node[u] (v11) at (90:1.5){};
\node[u] (v12) at (342:1.5){};
\node (C10) at (0, 0){$C_{10}$};

\draw   (0,0) circle[radius=1cm];
  \draw (v1) -- (v11);
  \draw (v2) -- (v11);   
  \draw (v4) -- (v12);
  \draw (v5) -- (v12);

\node[u] (w6) at (288:1.4){};
\node[u] (w8) at (216:1.4){};
\node[u] (w10) at (144:1.4){};
  \draw (v6) -- (w6);
  \draw (v8) -- (w8);  
  \draw (v10) -- (w10);
    
   \node[left=0.05cm, above=0.05cm, font=\small] at (v1) {$v_1$};  
   \node[ right=0.05cm, above=0.05cm,  font=\small] at (v2) {$v_2$};  
   \node[above=0.05cm, font=\small] at (v3) {$v_3$};     
   \node[right=0.05cm,font=\small] at (v4) {$v_4$};     
    \node[below=0.2cm,right=0.01cm, font=\small] at (v5) {$v_5$};    
   \node[right=0.3cm, below=0.01cm, font=\small] at (v6) {$v_6$};      
   \node[right=0.05cm, below=0.05cm,  font=\small] at (v7) {$v_7$};  
   \node[left=0.1cm, font=\small] at (v8) {$v_8$};    
    \node[left=0.05cm,font=\small] at (v9) {$v_9$};  
   \node[left=0.1cm, above=0.05cm,  font=\small] at (v10) {$v_{10}$};        
    \node[above=0.05cm, font=\small] at (v11) {$v_{11}$}; 
    \node[below=0.05cm,  font=\small] at (v12) {$v_{12}$};                      
    
 \node[below=0.05cm,  font=\scriptsize] at (v1) {\textcolor{blue}{2}};
\node[below=0.05cm,  font=\scriptsize] at (v2) {\textcolor{blue}{4}};  
\node[left=0.05cm,font=\scriptsize] at (v3) {\textcolor{blue}{6}};  
\node[above=0.05cm, left=0.05cm,font=\scriptsize] at (v4) {\textcolor{blue}{5}};  
\node[left=0.05cm,font=\scriptsize] at (v5) {\textcolor{blue}{4}};  
\node[above=0.05cm, font=\scriptsize] at (v6) {\textcolor{blue}{3}};  
\node[above=0.05cm, font=\scriptsize] at (v7) {\textcolor{blue}{3}};  
\node[right=0.05cm, font=\scriptsize] at (v8) {\textcolor{blue}{1}};  
\node[right=0.1cm, font=\scriptsize] at (v11) {\textcolor{blue}{2}}; 
\node[above=0.1cm, font=\scriptsize] at (v12) {\textcolor{blue}{3}}; 
 \end{tikzpicture}
       \vfill {\small (a) The list sizes after coloring $v_9,v_{10}$\ \ \ \ \ \ \ \ \ \ \ \ \ \ \ \ \ \ \ \ \ \ \ \ \ \ \ }
 \end{center}
\par
\begin{center}
\begin{tikzpicture}
[u/.style={fill=black, minimum size =3pt,ellipse,inner sep=1pt},node distance=1.5cm,scale=1.5]
\node[u] (v1) at (108:1){};
\node[u] (v2) at (72:1){};
\node[u] (v3) at (36:1){};
\node[u] (v4) at (0:1){};
\node[u] (v5) at (324:1){};
\node[u] (v6) at (288:1){};
\node[u] (v7) at (252:1){};
\node[u] (v8) at (216:1){};
\node[u] (v9) at (180:1){};
\node[u] (v10) at (144:1){};
\node[u] (v11) at (90:1.5){};
\node[u] (v12) at (342:1.5){};
\node (C10) at (0, 0){$C_{10}$};

\draw   (0,0) circle[radius=1cm];
  \draw (v1) -- (v11);
  \draw (v2) -- (v11);   
  \draw (v4) -- (v12);
  \draw (v5) -- (v12);

\node[u] (w6) at (288:1.4){};
\node[u] (w8) at (216:1.4){};
\node[u] (w10) at (144:1.4){};
  \draw (v6) -- (w6);
  \draw (v8) -- (w8);  
  \draw (v10) -- (w10);
    
   \node[left=0.05cm, above=0.05cm, font=\small] at (v1) {$v_1$};  
   \node[ right=0.05cm, above=0.05cm,  font=\small] at (v2) {$v_2$};  
   \node[above=0.05cm, font=\small] at (v3) {$v_3$};     
   \node[right=0.05cm,font=\small] at (v4) {$v_4$};     
    \node[below=0.2cm,right=0.01cm, font=\small] at (v5) {$v_5$};    
   \node[right=0.3cm, below=0.01cm, font=\small] at (v6) {$v_6$};      
   \node[right=0.05cm, below=0.05cm,  font=\small] at (v7) {$v_7$};  
   \node[left=0.1cm, font=\small] at (v8) {$v_8$};    
    \node[left=0.05cm,font=\small] at (v9) {$v_9$};  
   \node[left=0.1cm, above=0.05cm,  font=\small] at (v10) {$v_{10}$};        
    \node[above=0.05cm, font=\small] at (v11) {$v_{11}$}; 
    \node[below=0.05cm,  font=\small] at (v12) {$v_{12}$};                      
    
 \node[below=0.05cm,  font=\scriptsize] at (v1) {\textcolor{blue}{2}};
\node[below=0.05cm,  font=\scriptsize] at (v2) {\textcolor{blue}{4}};  
\node[left=0.05cm,font=\scriptsize] at (v3) {\textcolor{blue}{4}};  
\node[above=0.05cm, left=0.05cm,font=\scriptsize] at (v4) {\textcolor{blue}{2}};  
\node[left=0.05cm,font=\scriptsize] at (v5) {}; 
\node[above=0.05cm, font=\scriptsize] at (v6) {};  
\node[above=0.05cm, font=\scriptsize] at (v7) {};   
\node[right=0.05cm, font=\scriptsize] at (v8) {};
\node[right=0.05cm, font=\scriptsize] at (v9) {};  
\node[right=0.05cm,  font=\scriptsize] at (v10) {};
\node[right=0.1cm, font=\scriptsize] at (v11) {\textcolor{blue}{2}}; 
\node[above=0.1cm, font=\scriptsize] at (v12) {};
 \end{tikzpicture}
       \vfill {\small (b) The list sizes after coloring $v_{i}$: $i \in \{5,\dots, 10\}\cup \{12\}$}
 \end{center}
\end{multicols} 
\caption{The graph $H_3$: The list sizes after coloring some vertices $H_3$} 
\label{H3-color}
\end{figure*}

\item \textbf{Case  $H_i = H_4$.}

As indicated in Figure~\ref{H234}, we have:
\[
|L_{H_4}(v_i)| \ge 
\begin{cases}
3, & i = 2, 4, 6, 8, 10, \\
4, & i = 1, 3, 5, 7, 9.
\end{cases}
\]

We first show that there exist colors 
\(c_1 \in L_{H_4}(v_1)\) and \(c_4 \in L_{H_4}(v_4)\) such that  
\(|L_{H_4}(v_3)\setminus \{c_1, c_4\}| \ge 3\).

If \(L_{H_4}(v_1) \cap L_{H_4}(v_4) \neq \emptyset\),  
choose \(c_1 = c_4 \in L_{H_4}(v_1) \cap L_{H_4}(v_4)\).  
Then
\[
|L_{H_4}(v_3)\setminus \{c_1, c_4\}| 
\ge |L_{H_4}(v_3)| - 1 \ge 3.
\]

If $L_{H_4}(v_1) \cap L_{H_4}(v_4) = \emptyset$, then $|L_{H_4}(v_1) \cup L_{H_4}(v_4)| > |L_{H_4}(v_3)|$.  So, 
$L_{H_4}(v_1)\setminus L_{H_4}(v_3) \neq \emptyset$ or $L_{H_4}(v_4)\setminus L_{H_4}(v_3) \neq \emptyset$.  So, we can color $v_1$ and $v_4$ by $c_1, c_4$, respectively, so that $|L_{H_4}(v_3) \cap \{c_1, c_4\}| \leq 1$. Then, $|L_{H_4}(v_3) \setminus \{c_1, c_4\}| \geq 3$.

Afterward, we can color  
\[
v_2, v_{10}, v_9, v_8, v_6, v_7, v_5, v_3
\]
in this order to obtain an \(L\)-coloring of \(G^2\), a contradiction.  \qed
    \end{enumerate}

\subsection{Reducible Configurations on $9$-Cycles}

Assume that $C$ is a $9$-cycle with $t(C)=d(C)-5=4$.

\begin{claim}\label{claim-Ws}
If $t(C) = 4$, then $G$ contains $F_i$ as a subgraph for some $i \in \{1,2,\dots,12\}$.
\end{claim}

\begin{proof}
If $k = |A(C)| \le 1$, then by Claim~\ref{cl:t(f)} we must have $k=1$, and $G$ contains $F_{12}$ in Figure~\ref{W12-subgraph} as a subgraph.  
Thus we assume $k \ge 2$.  
By Claim~\ref{cl:t(f)}(2), we have:
\begin{itemize}
\item $s_4 + s_5 = 1$, implying $(s_4,s_5) \in \{(1,0),(0,1)\}$, 

\item $s_1 + s_2 + s_3 = k-1$,

\item $4 = s_1 + 2s_2 + 3s_3 + l_1,\quad 0 \le l_1 \le 1.$
\end{itemize}

\medskip
\noindent\textbf{Case $k=2$}.

Here $s_1 + s_2 + s_3 = 1$.  
From \(4 = s_1 + 2s_2 + 3s_3 + l_1\), we must have \(l_1 \ge 1\), and hence $(s_4,s_5)=(1,0)$.  
Thus
\[
(s_1,s_2,s_3,s_4,s_5) \in \{(1,0,0,1,0),\ (0,1,0,1,0),\ (0,0,1,1,0)\}.
\]

If $s_1=1$, then $l_1=3$ and $G$ contains  the configuration $F_{11}$. 
 
If $s_2=1$, then $l_1=2$ and $G$ contains the configuration $F_{10}$.  

If $s_3=1$, then $l_1=1$ and $G$ contains the configuration $F_9$.

\medskip
\noindent\textbf{Case $k=3$}.

Then $s_1 + s_2 + s_3 = 2$ and $4 = s_1 + 2s_2 + 3s_3 + l_1$.

\medskip\noindent
\emph{Subcase} $(s_4,s_5) = (0,1)$.  
Then $l_1=0$, so $s_1 + 2s_2 + 3s_3 =4$, giving $s_2 + 2s_3 =2$.  
Thus $(s_2,s_3)=(0,1)$ or $(2,0)$.

If $(s_2,s_3)=(0,1)$, then $(s_1,s_2,s_3,s_4,s_5) = (1,0,1,0,1)$ and $G$ contains  the configuration $F_8$.
  
If $(s_2,s_3)=(2,0)$, then $(s_1,s_2,s_3,s_4,s_5) = (0,2,0,0,1)$ and $G$ contains the configuration $F_6$ or  $F_7$.

\medskip\noindent
\emph{Subcase} $(s_4,s_5) = (1,0)$.  
Then $l_1 \ge 1$.  
Since $s_1 + s_2 + s_3 = 2$, we have $s_2 + 2s_3 + l_1 = 2$, which forces $s_3=0$.  
Hence
\[
s_1 + s_2 = 2,\qquad s_2 + l_1 = 2.
\]
Since $l_1 \ge 1$, we have 
\[
(s_1,s_2,l_1)=(2,0,2) \quad\text{or}\quad (1,1,1).
\]
Thus, 
$(s_1,s_2,s_3,s_4,s_5)=(2,0,0,1,0)$ giving  the configuration  $F_5$, or  $(s_1,s_2,s_3,s_4,s_5)=(1,1,0,1,0)$ giving the configuration $F_4$.

\medskip
\noindent\textbf{Case $k=4$}. Then $s_1 + s_2 + s_3 = 3$ and   $4 = s_1 + 2s_2 + 3s_3 + l_1$.

Thus   $s_2 + 2s_3 + l_1 = 1$ which implies  $s_3=0$,  $s_1 + s_2 = 3$ and $s_2 + l_1 = 1$.

If $(s_4,s_5)=(0,1)$, then $l_1=0$, so $s_2=1$ and $s_1=2$, giving $(s_1,s_2,s_3,s_4,s_5) = (2,1,0,0,1)$ and $G$ contains the configuration  $F_3$.

If $(s_4,s_5)=(1,0)$, then $l_1=1$, so $s_2=0$ and $s_1=3$, giving $(s_1,s_2,s_3,s_4,s_5) = (3,0,0,1,0)$ and $G$ contains the configuration $F_2$.

\medskip
\noindent\textbf{Case $k=5$}.  Then $s_1 + s_2 + s_3 = 4$ and  
$4 = s_1 + 2s_2 + 3s_3 + l_1$.

This forces $s_2 = s_3 = l_1 = 0$ and $s_1 = 4$.  
Since $l_1=0$, we have $(s_4,s_5) = (0,1)$, giving $(s_1,s_2,s_3,s_4,s_5) =(4,0,0,0,1)$, and $G$ contains $F_1$.

\medskip
In summary, we have shown that $G$ contains $F_i$ for some $i \in \{1,2,\dots,12\}$.
\end{proof}

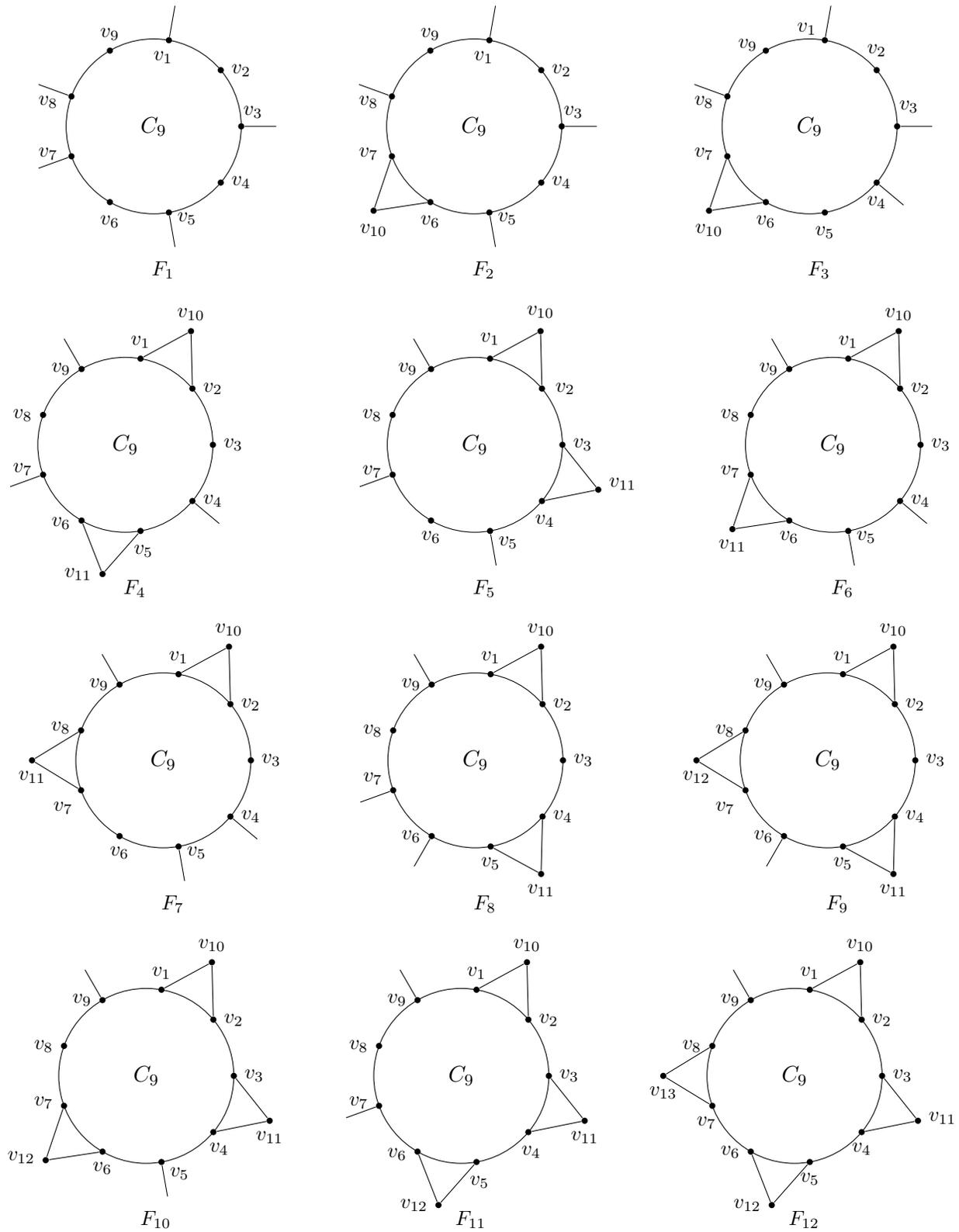
\begin{figure}
\begin{center}
\subfigure{\begin{tikzpicture}
[u/.style={fill=black, minimum size =3pt,ellipse,inner sep=1pt},node distance=1.5cm,scale=1.5]
\node[u] (v1) at (80:1){};
\node[u] (v2) at (40:1){};
\node[u] (v3) at (0:1){};
\node[u] (v4) at (320:1){};
\node[u] (v5) at (280:1){};
\node[u] (v6) at (240:1){};
\node[u] (v7) at (200:1){};
\node[u] (v8) at (160:1){};
\node[u] (v9) at (120:1){};
\node (C9) at (0, 0){$C_9$};

\draw   (0,0) circle[radius=1cm];
  
  \draw (v1) -- (80:1.4);
  \draw (v3) -- (0:1.4);
  \draw (v5) -- (280:1.4);
  \draw (v7) -- (200:1.4);
  \draw (v8) -- (160:1.4); 
    
   \node[left=0.1cm,  below=0.05cm,  font=\small] at (v1) {$v_1$};  
   \node[right=0.05cm,font=\small] at (v2) {$v_2$};  
   \node[right=0.2cm,above= 0.01cm, font=\small] at (v3) {$v_3$};       
   \node[right=0.05cm,font=\small] at (v4) {$v_4$};     
    \node[below=0.1cm,right=0.01cm, font=\small] at (v5) {$v_5$};    
   \node[below=0.1cm,font=\small] at (v6) {$v_6$};      
   \node[above=0.05cm, left=0.1cm, font=\small] at (v7) {$v_7$};  
   \node[below=0.1cm, left=0.1cm, font=\small] at (v8) {$v_8$};    
    \node[above=0.05cm,font=\small] at (v9) {$v_9$};      
    
  \node[left=0.1cm,below=0.7cm,font=\small] at (v5) {$F_1$};
 \end{tikzpicture}}
~~~~~~~~ \subfigure{\begin{tikzpicture}
[u/.style={fill=black, minimum size =3pt,ellipse,inner sep=1pt},node distance=1.5cm,scale=1.5]
\node[u] (v1) at (80:1){};
\node[u] (v2) at (40:1){};
\node[u] (v3) at (0:1){};
\node[u] (v4) at (320:1){};
\node[u] (v5) at (280:1){};
\node[u] (v6) at (240:1){};
\node[u] (v7) at (200:1){};
\node[u] (v8) at (160:1){};
\node[u] (v9) at (120:1){};
\node[u] (v10) at (220:1.5){};
\node (C9) at (0, 0){$C_9$};

\draw   (0,0) circle[radius=1cm];
  \draw (v6) -- (v10);
  \draw (v7) -- (v10);      
  
  \draw (v1) -- (80:1.4);
  \draw (v3) -- (0:1.4);
  \draw (v5) -- (280:1.4);
  \draw (v8) -- (160:1.4); 
    
   \node[left=0.1cm,  below=0.05cm,  font=\small] at (v1) {$v_1$};  
   \node[right=0.05cm,font=\small] at (v2) {$v_2$};  
   \node[right=0.2cm,above= 0.01cm, font=\small] at (v3) {$v_3$};      
   \node[right=0.05cm,font=\small] at (v4) {$v_4$};     
    \node[below=0.1cm,right=0.01cm, font=\small] at (v5) {$v_5$};    
   \node[below=0.1cm,font=\small] at (v6) {$v_6$};      
   \node[above=0.05cm, left=0.1cm, font=\small] at (v7) {$v_7$};  
   \node[below=0.1cm, left=0.1cm, font=\small] at (v8) {$v_8$};    
    \node[above=0.05cm,font=\small] at (v9) {$v_9$};      
     
     \node[below=0.1cm,  font=\small] at (v10) {$v_{10}$};

\node[left=0.1cm,below=0.7cm,font=\small] at (v5) {$F_2$};
 \end{tikzpicture}}
~~~~~~~~~~~\subfigure{\begin{tikzpicture}
[u/.style={fill=black, minimum size =3pt,ellipse,inner sep=1pt},node distance=1.5cm,scale=1.5]
\node[u] (v1) at (80:1){};
\node[u] (v2) at (40:1){};
\node[u] (v3) at (0:1){};
\node[u] (v4) at (320:1){};
\node[u] (v5) at (280:1){};
\node[u] (v6) at (240:1){};
\node[u] (v7) at (200:1){};
\node[u] (v8) at (160:1){};
\node[u] (v9) at (120:1){};
\node[u] (v10) at (220:1.5){};
\node (C9) at (0, 0){$C_9$};

\draw   (0,0) circle[radius=1cm];
  
  \draw (v1) -- (80:1.4);
  \draw (v3) -- (0:1.4);
  \draw (v4) -- (320:1.4);
  \draw (v8) -- (160:1.4); 
  \draw (v6) -- (v10);
  \draw (v7) -- (v10);   
    
   \node[above=0.2cm, left=0.01cm, font=\small] at (v1) {$v_1$};  
   \node[above=0.1cm,font=\small] at (v2) {$v_2$};  
   \node[right=0.2cm,above= 0.1cm, font=\small] at (v3) {$v_3$};     
   \node[below=0.1cm,font=\small] at (v4) {$v_4$};     
    \node[below=0.1cm, font=\small] at (v5) {$v_5$};    
   \node[below=0.1cm,font=\small] at (v6) {$v_6$};      
   \node[above=0.05cm, left=0.1cm, font=\small] at (v7) {$v_7$};  
   \node[below=0.1cm, left=0.1cm, font=\small] at (v8) {$v_8$};    
    \node[above=0.1cm, left=0.01cm, font=\small] at (v9) {$v_9$};  
     \node[below=0.1cm,font=\small] at (v10) {$v_{10}$};

\node[left=0.1cm,below=0.7cm,font=\small] at (v5) {$F_3$}; 
 \end{tikzpicture}}
~~~~~~~~\subfigure{\begin{tikzpicture}
[u/.style={fill=black, minimum size =3pt,ellipse,inner sep=1pt},node distance=1.5cm,scale=1.5]
\node[u] (v1) at (80:1){};
\node[u] (v2) at (40:1){};
\node[u] (v3) at (0:1){};
\node[u] (v4) at (320:1){};
\node[u] (v5) at (280:1){};
\node[u] (v6) at (240:1){};
\node[u] (v7) at (200:1){};
\node[u] (v8) at (160:1){};
\node[u] (v9) at (120:1){};
\node[u] (v10) at (60:1.5){};
\node[u] (v11) at (260:1.5){};
\node (C9) at (0, 0){$C_9$};

\draw   (0,0) circle[radius=1cm];
  \draw (v1) -- (v10);
  \draw (v2) -- (v10);  
  \draw (v5) -- (v11);
  \draw (v6) -- (v11);


  \draw (v4) -- (320:1.4);
  \draw (v7) -- (200:1.4); 
 \draw (v9) -- (120:1.4);   
    
   \node[above=0.05cm, font=\small] at (v1) {$v_1$};  
   \node[right=0.05cm,font=\small] at (v2) {$v_2$};  
   \node[right=0.05cm, font=\small] at (v3) {$v_3$};     
   \node[right=0.05cm,font=\small] at (v4) {$v_4$};     
    \node[right=0.05cm, below=0.1cm, font=\small] at (v5) {$v_5$};    
   \node[below=0.05cm,left=0.05cm, font=\small] at (v6) {$v_6$};      
   \node[above=0.1cm, left=0.05cm, font=\small] at (v7) {$v_7$};  
   \node[left=0.05cm, font=\small] at (v8) {$v_8$};    
    \node[left=0.05cm, font=\small] at (v9) {$v_9$};  
   \node[above=0.05cm,  font=\small] at (v10) {$v_{10}$};        
    \node[left=0.05cm,  font=\small] at (v11) {$v_{11}$};

\node[left=0.1cm,below=0.7cm,font=\small] at (v5) {$F_4$}; 
 \end{tikzpicture}}
~~~~~~~~~~~~~\subfigure{\begin{tikzpicture}
[u/.style={fill=black, minimum size =3pt,ellipse,inner sep=1pt},node distance=1.5cm,scale=1.5]
\node[u] (v1) at (80:1){};
\node[u] (v2) at (40:1){};
\node[u] (v3) at (0:1){};
\node[u] (v4) at (320:1){};
\node[u] (v5) at (280:1){};
\node[u] (v6) at (240:1){};
\node[u] (v7) at (200:1){};
\node[u] (v8) at (160:1){};
\node[u] (v9) at (120:1){};
\node[u] (v10) at (60:1.5){};
\node[u] (v11) at (340:1.5){};
\node (C9) at (0, 0){$C_9$};

\draw   (0,0) circle[radius=1cm];
  \draw (v1) -- (v10);
  \draw (v2) -- (v10);  
  \draw (v3) -- (v11);
  \draw (v4) -- (v11);  
        
  \draw (v5) -- (280:1.4);
  \draw (v7) -- (200:1.4); 
 \draw (v9) -- (120:1.4);   
    
   \node[above=0.1cm, font=\small] at (v1) {$v_1$};  
   \node[right=0.1cm,font=\small] at (v2) {$v_2$};  
   \node[right=0.05cm, font=\small] at (v3) {$v_3$};     
   \node[right=0.05cm, below=0.1cm,font=\small] at (v4) {$v_4$};     
    \node[below=0.2cm,right=0.01cm, font=\small] at (v5) {$v_5$};    
   \node[below=0.05cm,font=\small] at (v6) {$v_6$};      
   \node[above=0.05cm, left=0.05cm, font=\small] at (v7) {$v_7$};  
   \node[left=0.05cm, font=\small] at (v8) {$v_8$};    
    \node[left=0.1cm,font=\small] at (v9) {$v_9$};  
   \node[above=0.1cm,  font=\small] at (v10) {$v_{10}$};        
    \node[above=0.1cm, right=0.05cm, font=\small] at (v11) {$v_{11}$};

\node[left=0.1cm,below=0.7cm,font=\small] at (v5) {$F_5$}; 
 \end{tikzpicture}}
~~~~~~~~\subfigure{\begin{tikzpicture}
[u/.style={fill=black, minimum size =3pt,ellipse,inner sep=1pt},node distance=1.5cm,scale=1.5]
\node[u] (v1) at (80:1){};
\node[u] (v2) at (40:1){};
\node[u] (v3) at (0:1){};
\node[u] (v4) at (320:1){};
\node[u] (v5) at (280:1){};
\node[u] (v6) at (240:1){};
\node[u] (v7) at (200:1){};
\node[u] (v8) at (160:1){};
\node[u] (v9) at (120:1){};
\node[u] (v10) at (60:1.5){};
\node[u] (v11) at (220:1.5){};
\node (C9) at (0, 0){$C_9$};

\draw   (0,0) circle[radius=1cm];
  \draw (v1) -- (v10);
  \draw (v2) -- (v10);   
  \draw (v6) -- (v11);
  \draw (v7) -- (v11);

  \draw (v4) -- (320:1.4);
  \draw (v5) -- (280:1.4);
 \draw (v9) -- (120:1.4);   
    
   \node[above=0.05cm, font=\small] at (v1) {$v_1$};  
   \node[right=0.05cm,font=\small] at (v2) {$v_2$};  
   \node[right=0.05cm, font=\small] at (v3) {$v_3$};     
   \node[right=0.05cm,font=\small] at (v4) {$v_4$};     
    \node[below=0.1cm,right=0.01cm, font=\small] at (v5) {$v_5$};    
   \node[below=0.05cm,font=\small] at (v6) {$v_6$};      
   \node[above=0.05cm, left=0.05cm, font=\small] at (v7) {$v_7$};  
   \node[left=0.05cm, font=\small] at (v8) {$v_8$};    
    \node[left=0.05cm,font=\small] at (v9) {$v_9$};  
   \node[above=0.05cm,  font=\small] at (v10) {$v_{10}$};        
    \node[below=0.05cm, font=\small] at (v11) {$v_{11}$};

  \node[left=0.1cm,below=0.7cm, font=\small] at (v5) {$F_6$};
 \end{tikzpicture}}
~~~~~~~~\subfigure{\begin{tikzpicture}
[u/.style={fill=black, minimum size =3pt,ellipse,inner sep=1pt},node distance=1.5cm,scale=1.5]
\node[u] (v1) at (80:1){};
\node[u] (v2) at (40:1){};
\node[u] (v3) at (0:1){};
\node[u] (v4) at (320:1){};
\node[u] (v5) at (280:1){};
\node[u] (v6) at (240:1){};
\node[u] (v7) at (200:1){};
\node[u] (v8) at (160:1){};
\node[u] (v9) at (120:1){};
\node[u] (v10) at (60:1.5){};
\node[u] (v11) at (180:1.5){};
\node (C9) at (0, 0){$C_9$};

\draw   (0,0) circle[radius=1cm];
  \draw (v1) -- (v10);
  \draw (v2) -- (v10);   
  \draw (v7) -- (v11);
  \draw (v8) -- (v11);

  \draw (v4) -- (320:1.4);
  \draw (v5) -- (280:1.4);
 \draw (v9) -- (120:1.4);   
    
   \node[above=0.05cm, font=\small] at (v1) {$v_1$};  
   \node[right=0.05cm, font=\small] at (v2) {$v_2$};  
   \node[right=0.05cm, font=\small] at (v3) {$v_3$};     
   \node[right=0.05cm, font=\small] at (v4) {$v_4$};     
    \node[below=0.1cm, right=0.01cm, font=\small] at (v5) {$v_5$};    
   \node[below=0.05cm, font=\small] at (v6) {$v_6$};      
   \node[below=0.3cm, left=0.01cm, font=\small] at (v7) {$v_7$};  
   \node[above=0.05cm, left=0.05cm, font=\small] at (v8) {$v_8$};    
    \node[left=0.05cm, font=\small] at (v9) {$v_9$};  
   \node[above=0.05cm,  font=\small] at (v10) {$v_{10}$};        
    \node[below=0.05cm, font=\small] at (v11) {$v_{11}$};    

\node[left=0.1cm,below=0.7cm, font=\small]at (v5) {$F_7$};
 \end{tikzpicture}}
~~~~~~~~\subfigure{\begin{tikzpicture}
[u/.style={fill=black, minimum size =3pt,ellipse,inner sep=1pt},node distance=1.5cm,scale=1.5]
\node[u] (v1) at (80:1){};
\node[u] (v2) at (40:1){};
\node[u] (v3) at (0:1){};
\node[u] (v4) at (320:1){};
\node[u] (v5) at (280:1){};
\node[u] (v6) at (240:1){};
\node[u] (v7) at (200:1){};
\node[u] (v8) at (160:1){};
\node[u] (v9) at (120:1){};
\node[u] (v10) at (60:1.5){};
\node[u] (v11) at (300:1.5){};
\node (C9) at (0, 0){$C_9$};

\draw   (0,0) circle[radius=1cm];
  \draw (v1) -- (v10);
  \draw (v2) -- (v10);  
  \draw (v4) -- (v11);
  \draw (v5) -- (v11);

  \draw (v6) -- (240:1.4); 
  \draw (v7) -- (200:1.4); 
 \draw (v9) -- (120:1.4);   
    
   \node[above=0.05cm, font=\small] at (v1) {$v_1$};  
   \node[right=0.05cm,font=\small] at (v2) {$v_2$};  
   \node[right=0.05cm, font=\small] at (v3) {$v_3$};     
   \node[right=0.05cm, font=\small] at (v4) {$v_4$};     
    \node[below=0.05cm,  font=\small] at (v5) {$v_5$};    
   \node[left=0.05cm,font=\small] at (v6) {$v_6$};      
   \node[above=0.2cm, left=0.05cm, font=\small] at (v7) {$v_7$};  
   \node[left=0.05cm, font=\small] at (v8) {$v_8$};    
    \node[left=0.05cm,font=\small] at (v9) {$v_9$};  
   \node[above=0.05cm,  font=\small] at (v10) {$v_{10}$};        
    \node[below=0.05cm, font=\small] at (v11) {$v_{11}$};

\node[left=0.1cm,below=0.7cm, font=\small] at (v5) {$F_8$}; 
 \end{tikzpicture}}
~~~~~~~~ \subfigure{\begin{tikzpicture}
[u/.style={fill=black, minimum size =3pt,ellipse,inner sep=1pt},node distance=1.5cm,scale=1.5]
\node[u] (v1) at (80:1){};
\node[u] (v2) at (40:1){};
\node[u] (v3) at (0:1){};
\node[u] (v4) at (320:1){};
\node[u] (v5) at (280:1){};
\node[u] (v6) at (240:1){};
\node[u] (v7) at (200:1){};
\node[u] (v8) at (160:1){};
\node[u] (v9) at (120:1){};
\node[u] (v10) at (60:1.5){};
\node[u] (v11) at (300:1.5){};
\node[u] (v12) at (180:1.5){};
\node (C9) at (0, 0){$C_9$};

\draw   (0,0) circle[radius=1cm];
  \draw (v1) -- (v10);
  \draw (v2) -- (v10);  
  \draw (v4) -- (v11);
  \draw (v5) -- (v11); 

  \draw (v7) -- (v12);
  \draw (v8) -- (v12);

  \draw (v6) -- (240:1.4); 
 \draw (v9) -- (120:1.4);   
    
   \node[above=0.05cm, font=\small] at (v1) {$v_1$};  
   \node[right=0.05cm,font=\small] at (v2) {$v_2$};  
   \node[right=0.05cm, font=\small] at (v3) {$v_3$};     
   \node[right=0.05cm,font=\small] at (v4) {$v_4$};     
    \node[below=0.05cm,  font=\small] at (v5) {$v_5$};    
   \node[left=0.05cm,font=\small] at (v6) {$v_6$};      
   \node[below=0.3cm, left=0.05cm, font=\small] at (v7) {$v_7$};  
   \node[left=0.05cm, font=\small] at (v8) {$v_8$};    
    \node[left=0.05cm,font=\small] at (v9) {$v_9$};  
   \node[above=0.05cm,  font=\small] at (v10) {$v_{10}$};        
    \node[below=0.05cm, font=\small] at (v11) {$v_{11}$};   
    \node[below=0.05cm,  font=\small] at (v12) {$v_{12}$};    

\node[left=0.1cm,below=0.7cm, font=\small] at (v5) { $F_9$};
 \end{tikzpicture}}
~~~~\subfigure{\begin{tikzpicture}
[u/.style={fill=black, minimum size =3pt,ellipse,inner sep=1pt},node distance=1.5cm,scale=1.5]
\node[u] (v1) at (80:1){};
\node[u] (v2) at (40:1){};
\node[u] (v3) at (0:1){};
\node[u] (v4) at (320:1){};
\node[u] (v5) at (280:1){};
\node[u] (v6) at (240:1){};
\node[u] (v7) at (200:1){};
\node[u] (v8) at (160:1){};
\node[u] (v9) at (120:1){};
\node[u] (v10) at (60:1.5){};
\node[u] (v11) at (340:1.5){};
\node[u] (v12) at (220:1.5){};
\node (C9) at (0, 0){$C_9$};

\draw   (0,0) circle[radius=1cm];
  \draw (v1) -- (v10);
  \draw (v2) -- (v10);  
  \draw (v3) -- (v11);
  \draw (v4) -- (v11);  
  \draw (v6) -- (v12);
  \draw (v7) -- (v12);

  \draw (v5) -- (280:1.4);
 \draw (v9) -- (120:1.4);   
 
   \node[above=0.05cm, font=\small] at (v1) {$v_1$};  
   \node[right=0.05cm,font=\small] at (v2) {$v_2$};  
   \node[right=0.05cm, font=\small] at (v3) {$v_3$};     
   \node[right=0.1cm, below=0.05cm, font=\small] at (v4) {$v_4$};     
    \node[below=0.2cm,right=0.01cm, font=\small] at (v5) {$v_5$};    
   \node[below=0.05cm,font=\small] at (v6) {$v_6$};      
   \node[above=0.05cm, left=0.05cm, font=\small] at (v7) {$v_7$};  
   \node[left=0.05cm, font=\small] at (v8) {$v_8$};    
    \node[left=0.05cm,font=\small] at (v9) {$v_9$};  
   \node[above=0.05cm,  font=\small] at (v10) {$v_{10}$};        
    \node[below=0.05cm,  font=\small] at (v11) {$v_{11}$};    
     \node[above=0.05cm, left=0.05cm, font=\small] at (v12) {$v_{12}$};                  
  \node[left=0.1cm,below=0.7cm,font=\small] at (v5) { $F_{10}$};
 \end{tikzpicture}}
~~~~~~\subfigure{\begin{tikzpicture}
[u/.style={fill=black, minimum size =3pt,ellipse,inner sep=1pt},node distance=1.5cm,scale=1.5]
\node[u] (v1) at (80:1){};
\node[u] (v2) at (40:1){};
\node[u] (v3) at (0:1){};
\node[u] (v4) at (320:1){};
\node[u] (v5) at (280:1){};
\node[u] (v6) at (240:1){};
\node[u] (v7) at (200:1){};
\node[u] (v8) at (160:1){};
\node[u] (v9) at (120:1){};
\node[u] (v10) at (60:1.5){};
\node[u] (v11) at (340:1.5){};
\node[u] (v12) at (260:1.5){};
\node (C9) at (0, 0){$C_9$};

\draw   (0,0) circle[radius=1cm];
  \draw (v1) -- (v10);
  \draw (v2) -- (v10);  
  \draw (v3) -- (v11);
  \draw (v4) -- (v11); 

  \draw (v5) -- (v12);
  \draw (v6) -- (v12);  

  \draw (v7) -- (200:1.4); 
 \draw (v9) -- (120:1.4);   
    
   \node[above=0.05cm, font=\small] at (v1) {$v_1$};  
   \node[right=0.05cm,font=\small] at (v2) {$v_2$};  
   \node[right=0.05cm, font=\small] at (v3) {$v_3$};  
   \node[right=0.05cm, below=0.1cm, font=\small] at (v4) {$v_4$};     
    \node[right=0.05cm, below=0.1cm, font=\small] at (v5) {$v_5$};    
   \node[below=0.05cm, left=0.05cm, font=\small] at (v6) {$v_6$};      
   \node[above=0.05cm, left=0.05cm, font=\small] at (v7) {$v_7$};  
   \node[left=0.05cm, font=\small] at (v8) {$v_8$};    
    \node[left=0.05cm,font=\small] at (v9) {$v_9$};  
   \node[above=0.05cm,  font=\small] at (v10) {$v_{10}$};        
    \node[below=0.05cm, font=\small] at (v11) {$v_{11}$};   
    \node[left=0.05cm,  font=\small] at (v12) {$v_{12}$};                              
    
\node[left=0.1cm,below=0.7cm, font=\small] at (v5) {$F_{11}$}; 
 \end{tikzpicture}}
~~~~\subfigure{\begin{tikzpicture}
[u/.style={fill=black, minimum size =3pt,ellipse,inner sep=1pt},node distance=1.5cm,scale=1.5]
\node[u] (v1) at (80:1){};
\node[u] (v2) at (40:1){};
\node[u] (v3) at (0:1){};
\node[u] (v4) at (320:1){};
\node[u] (v5) at (280:1){};
\node[u] (v6) at (240:1){};
\node[u] (v7) at (200:1){};
\node[u] (v8) at (160:1){};
\node[u] (v9) at (120:1){};
\node[u] (v10) at (60:1.5){};
\node[u] (v11) at (340:1.5){};
\node[u] (v12) at (260:1.5){};
\node[u] (v13) at (180:1.5){};
\node (C9) at (0, 0){$C_9$};

\draw   (0,0) circle[radius=1cm];
  \draw (v1) -- (v10);
  \draw (v2) -- (v10);  
  \draw (v3) -- (v11);
  \draw (v4) -- (v11); 

  \draw (v5) -- (v12);
  \draw (v6) -- (v12);  
  \draw (v7) -- (v13);
  \draw (v8) -- (v13);     

 \draw (v9) -- (120:1.4);   
    
   \node[above=0.05cm, font=\small] at (v1) {$v_1$};  
   \node[right=0.05cm, font=\small] at (v2) {$v_2$};  
   \node[right=0.05cm, font=\small] at (v3) {$v_3$};     
   \node[below=0.05cm,font=\small] at (v4) {$v_4$};     
    \node[below=0.05cm,  font=\small] at (v5) {$v_5$};    
   \node[left=0.05cm, font=\small] at (v6) {$v_6$};      
   \node[left=0.1cm, below=0.05cm, font=\small] at (v7) {$v_7$};  
   \node[left=0.05cm, font=\small] at (v8) {$v_8$};    
    \node[left=0.05cm, font=\small] at (v9) {$v_9$};  
   \node[above=0.05cm,  font=\small] at (v10) {$v_{10}$};        
    \node[above=0.05cm, right=0.05cm, font=\small] at (v11) {$v_{11}$};   
    \node[left=0.05cm,  font=\small] at (v12) {$v_{12}$};    
     \node[below=0.05cm, font=\small] at (v13) {$v_{13}$};                             
      
\node[left=0.1cm,below=0.7cm, font=\small] at (v5) {$F_{12}$}; 
 \end{tikzpicture}}
\end{center}
 \caption{Reducible configurations on $9$-cycles.} 
\label{W12-subgraph}
  \end{figure}
  
  In the next two subsections, we show that $G$ does not contain $F_i$ as a subgraph for any $i \in \{1,\ldots,12\}$, thereby completing the proof of Claim~\ref{lemma-C10}.

For $i \in \{1,4,6,7,8,11\}$, in Subsection~\ref{subsection-F-one}, we use coloring extension arguments to show that $F_i$ cannot occur in $G$.
For $i \in \{2,3,5,9,10,12\}$, in Subsection~\ref{subsection-F-two}, we apply the Combinatorial Nullstellensatz to rule out $F_i$ as a subgraph of $G$, as the corresponding coloring extension arguments would be considerably longer and more involved.
\subsubsection{$G$ does not contain $F_i$ for $i \in \{1,4,6,7,8,11\}$}
\label{subsection-F-one}

In this subsection, we apply the coloring extension technique to show that $F_i$ does not appear in $G$ for any $i \in \{1,4,6,7,8,11\}$.

\begin{claim} \label{claim-F-one}
$G$ does not contain $F_i$ as a subgraph for each $i \in \{1,4,6,7,8,11\}$.
\end{claim}

\noindent {\bf \it Proof.}
Suppose to the contrary that $G$ contains $F_i$ as a subgraph for some $i\in\{1,4,6,7,8,11\}$.  
Let $G_1 = G - V(F_i)$.

Since no vertex of $F_i$ has two neighbors outside $F_i$, whenever $u,v \in V(G_1)$ satisfy $d_G(u,v)\le 2$, we also have $d_{G_1}(u,v)\le 2$.  
Moreover, since $G$ contains no $4$- to $8$-cycles, any two $3$-vertices of $F_i$ that do not lie on the $9$-cycle have distance at least $3$ in $G$, so their colors never conflict.

By the minimality of $G$, the graph $G_1^2$ admits an $L$-coloring $\phi$.  
We show that $\phi$ extends to $F_i$, yielding an $L$-coloring of $G^2$, contradicting that $G$ is a counterexample.

For each $v \in V(F_i)$, define
\[
L_{F_i}(v)
   = L(v)\setminus\{\phi(x) : vx \in E(G^2)\ \text{and}\ x\notin V(F_i)\}.
\]

\medskip
\begin{enumerate}[(a)]

\item \textbf{Case $F_i = F_1$.}

As indicated in  Figure~\ref{F1-color} (a), we have
\[
|L_{F_1}(v_i)| \ge
\begin{cases}
2, & i \in \{7,8\},\\
3, & i \in \{1,3,5\},\\
4, & i \in \{2,4,6,9\}.
\end{cases}
\]

Since $|L_{F_1}(v_6)| \ge 4$ and $|L_{F_1}(v_5)| \ge 3$, we may choose a color 
$\alpha \in L_{F_1}(v_6)$ such that
\[
|L_{F_1}(v_5)\setminus\{\alpha\}| \ge 3.
\]

Color $v_8$ with a color $c_8 \in L_{F_1}(v_8)\setminus\{\alpha\}$, and color 
$v_7$ with a color $c_7 \in L_{F_1}(v_7)\setminus\{c_8\}$.

Let $L_{F_1}'(v_i)$ be the updated lists after coloring $v_7$ and $v_8$.  
See Figure~\ref{F1-color}(b) for the sizes of the lists $L_{F_1}'(v_i)$.
We claim that
\[
|L_{F_1}'(v_6)\cup L_{F_1}'(v_5)| \ge 3.
\]

If not, then $L_{F_1}'(v_6)=L_{F_1}'(v_5)$ and both lists have size $2$, since each has size at least $2$.    
Since  $L_{F_1}'(v_5)=L_{F_1}(v_5)\setminus\{c_7\}$ and $|L_{F_1}(v_5)\setminus\{\alpha\}| \ge 3$, we have that
$$ c_7 \in L_{F_1}(v_5), ~\text{and} ~ |L_{F_1}(v_5)| = 3.$$
By the fact $|L_{F_1}(v_5)\setminus\{\alpha\}| \ge 3$ again, we have $\alpha \not \in L_{F_1}(v_5)$ and thus $\alpha \not \in L_{F_1}'(v_5)$.   Since   $\alpha\ne c_8$,  we have
$\alpha\in L_{F_1}'(v_6)$, a contradiction to the assumption that $L_{F_1}'(v_6)=L_{F_1}'(v_5)$. This proves that $|L_{F_1}'(v_6)\cup L_{F_1}'(v_5)| \ge 3$.

Next choose $c_4\in L_{F_1}'(v_4)$ such that $|L_{F_1}'(v_3)\setminus\{c_4\}|\ge 3$.  
Since $|L_{F_1}'(v_6)\cup L_{F_1}'(v_5)| \ge 3$, we may select distinct colors
\[
c_6 \in L_{F_1}'(v_6) \qquad \text{and} \qquad  
c_5 \in L_{F_1}'(v_5)
\]
such that $c_4 \notin \{c_5, c_6\}$.  
Color $v_4, v_5, v_6$ with $c_4, c_5, c_6$, respectively.

Let $L_{F_1}''(v_i)$ be the lists after coloring $v_4, v_5, v_6$. As indicated in  Figure~\ref{F1-color}(c), we have 
\[
|L_{F_1}''(v_1)| \ge 2,\qquad
|L_{F_1}''(v_2)| \ge 3,\qquad
|L_{F_1}''(v_3)| \ge 2,\qquad
|L_{F_1}''(v_9)| \ge 2.
\]

By Lemma~\ref{lem-key-K4-edge}, we can color $v_9, v_1, v_2, v_3$ to obtain an 
$L$-coloring of $G^2$, a contradiction.  
Thus $F_1$ is not a subgraph of $G$.

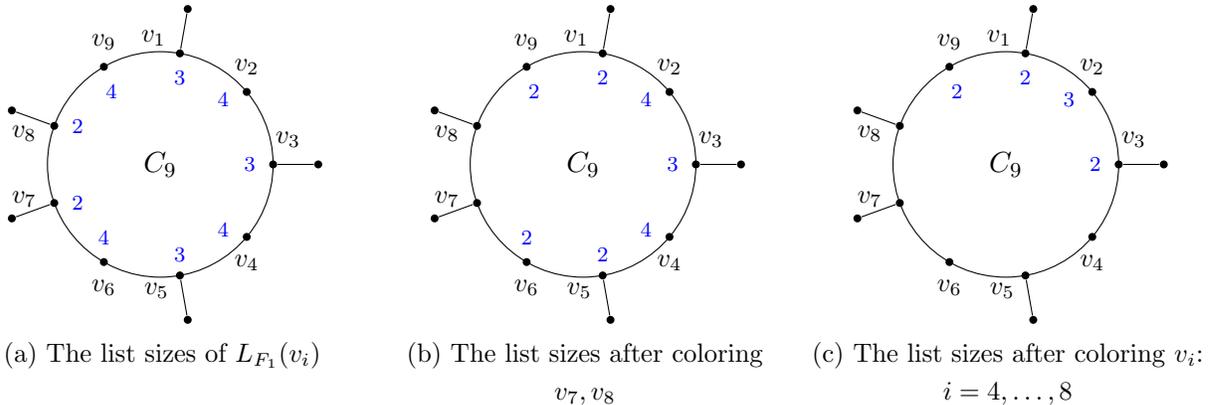
\begin{figure*}[htbp]
\begin{multicols}{3}
\begin{center}
\newdimen\R
\R=1.0cm
\begin{tikzpicture}
[u/.style={fill=black, minimum size =3pt,ellipse,inner sep=1pt},node distance=1.5cm,scale=1.5]
\node[u] (v1) at (80:1){};
\node[u] (v2) at (40:1){};
\node[u] (v3) at (0:1){};
\node[u] (v4) at (320:1){};
\node[u] (v5) at (280:1){};
\node[u] (v6) at (240:1){};
\node[u] (v7) at (200:1){};
\node[u] (v8) at (160:1){};
\node[u] (v9) at (120:1){};
\node (C9) at (0, 0){$C_9$};

\draw   (0,0) circle[radius=1cm];
  
\node[u] (w1) at (80:1.4){};
\node[u] (w3) at (0:1.4){};
\node[u] (w5) at (280:1.4){};
\node[u] (w7) at (200:1.4){};
\node[u] (w8) at (160:1.4){};
  \draw (v1) -- (w1);
  \draw (v3) -- (w3);
  \draw (v5) -- (w5);
  \draw (v7) -- (w7);
  \draw (v8) -- (w8); 
    
   \node[above=0.2cm, left=0.05cm, font=\small] at (v1) {$v_1$};  
   \node[above=0.1cm,font=\small] at (v2) {$v_2$};  
   \node[right=0.2cm,above= 0.1cm, font=\small] at (v3) {$v_3$};     
   \node[below=0.1cm,font=\small] at (v4) {$v_4$};     
    \node[below=0.2cm,left=0.01cm, font=\small] at (v5) {$v_5$};    
   \node[below=0.1cm,font=\small] at (v6) {$v_6$};      
   \node[above=0.05cm, left=0.1cm, font=\small] at (v7) {$v_7$};  
   \node[below=0.1cm, left=0.1cm, font=\small] at (v8) {$v_8$};    
    \node[above=0.1cm,font=\small] at (v9) {$v_9$};        
    
 \node[below=0.1cm,  font=\scriptsize] at (v1) {\textcolor{blue}{3}};
\node[below=0.1cm, left=0.1cm, font=\scriptsize] at (v2) {\textcolor{blue}{4}};  
\node[left=0.1cm,font=\scriptsize] at (v3) {\textcolor{blue}{3}};  
\node[above=0.1cm, left=0.1cm,font=\scriptsize] at (v4) {\textcolor{blue}{4}};  
\node[above=0.05cm,font=\scriptsize] at (v5) {\textcolor{blue}{3}};  
\node[above=0.1cm, font=\scriptsize] at (v6) {\textcolor{blue}{4}};  
\node[right=0.1cm, font=\scriptsize] at (v7) {\textcolor{blue}{2}};  
\node[right=0.1cm, font=\scriptsize] at (v8) {\textcolor{blue}{2}};  
\node[right=0.1cm, below=0.1cm, font=\scriptsize] at (v9) {\textcolor{blue}{4}};  

 \end{tikzpicture}
         \vfill {\small (a) The list sizes  of $L_{F_1}(v_i)$\ \ \ \ \ \ \ \ \ \ \ \ \ }
\end{center}
\par
\begin{center}
\begin{tikzpicture}
[u/.style={fill=black, minimum size =3pt,ellipse,inner sep=1pt},node distance=1.5cm,scale=1.5]
\node[u] (v1) at (80:1){};
\node[u] (v2) at (40:1){};
\node[u] (v3) at (0:1){};
\node[u] (v4) at (320:1){};
\node[u] (v5) at (280:1){};
\node[u] (v6) at (240:1){};
\node[u] (v7) at (200:1){};
\node[u] (v8) at (160:1){};
\node[u] (v9) at (120:1){};
\node (C9) at (0, 0){$C_9$};

\draw   (0,0) circle[radius=1cm];
  
\node[u] (w1) at (80:1.4){};
\node[u] (w3) at (0:1.4){};
\node[u] (w5) at (280:1.4){};
\node[u] (w7) at (200:1.4){};
\node[u] (w8) at (160:1.4){};
  \draw (v1) -- (w1);
  \draw (v3) -- (w3);
  \draw (v5) -- (w5);
  \draw (v7) -- (w7);
  \draw (v8) -- (w8); 
    
   \node[above=0.2cm, left=0.05cm, font=\small] at (v1) {$v_1$};  
   \node[above=0.1cm,font=\small] at (v2) {$v_2$};  
   \node[right=0.2cm,above= 0.1cm, font=\small] at (v3) {$v_3$};     
   \node[below=0.1cm,font=\small] at (v4) {$v_4$};     
    \node[below=0.2cm,left=0.01cm, font=\small] at (v5) {$v_5$};    
   \node[below=0.1cm,font=\small] at (v6) {$v_6$};      
   \node[above=0.05cm, left=0.1cm, font=\small] at (v7) {$v_7$};  
   \node[below=0.1cm, left=0.1cm, font=\small] at (v8) {$v_8$};    
    \node[above=0.1cm,font=\small] at (v9) {$v_9$};        
    
 \node[below=0.1cm,  font=\scriptsize] at (v1) {\textcolor{blue}{2}};
\node[below=0.1cm, left=0.1cm, font=\scriptsize] at (v2) {\textcolor{blue}{4}};  
\node[left=0.1cm,font=\scriptsize] at (v3) {\textcolor{blue}{3}};  
\node[above=0.1cm, left=0.1cm,font=\scriptsize] at (v4) {\textcolor{blue}{4}};  
\node[above=0.05cm,font=\scriptsize] at (v5) {\textcolor{blue}{2}};  
\node[above=0.1cm, font=\scriptsize] at (v6) {\textcolor{blue}{2}};  
\node[right=0.1cm, font=\scriptsize] at (v7) {};  
\node[right=0.1cm, font=\scriptsize] at (v8) {};  
\node[right=0.1cm, below=0.1cm, font=\scriptsize] at (v9) {\textcolor{blue}{2}};  
 \end{tikzpicture}
  \vfill {\small (b) The list sizes after coloring $v_7,v_8$}
\end{center}
\par
\begin{center}
\begin{tikzpicture}
[u/.style={fill=black, minimum size =3pt,ellipse,inner sep=1pt},node distance=1.5cm,scale=1.5]
\node[u] (v1) at (80:1){};
\node[u] (v2) at (40:1){};
\node[u] (v3) at (0:1){};
\node[u] (v4) at (320:1){};
\node[u] (v5) at (280:1){};
\node[u] (v6) at (240:1){};
\node[u] (v7) at (200:1){};
\node[u] (v8) at (160:1){};
\node[u] (v9) at (120:1){};
\node (C9) at (0, 0){$C_9$};

\draw   (0,0) circle[radius=1cm];
  
\node[u] (w1) at (80:1.4){};
\node[u] (w3) at (0:1.4){};
\node[u] (w5) at (280:1.4){};
\node[u] (w7) at (200:1.4){};
\node[u] (w8) at (160:1.4){};
  \draw (v1) -- (w1);
  \draw (v3) -- (w3);
  \draw (v5) -- (w5);
  \draw (v7) -- (w7);
  \draw (v8) -- (w8); 
    
   \node[above=0.2cm, left=0.05cm, font=\small] at (v1) {$v_1$};  
   \node[above=0.1cm,font=\small] at (v2) {$v_2$};  
   \node[right=0.2cm,above= 0.1cm, font=\small] at (v3) {$v_3$};     
   \node[below=0.1cm,font=\small] at (v4) {$v_4$};     
    \node[below=0.2cm,left=0.01cm, font=\small] at (v5) {$v_5$};    
   \node[below=0.1cm,font=\small] at (v6) {$v_6$};      
   \node[above=0.05cm, left=0.1cm, font=\small] at (v7) {$v_7$};  
   \node[below=0.1cm, left=0.1cm, font=\small] at (v8) {$v_8$};    
    \node[above=0.1cm,font=\small] at (v9) {$v_9$};        
    
 \node[below=0.1cm,  font=\scriptsize] at (v1) {\textcolor{blue}{2}};
\node[below=0.1cm, left=0.1cm, font=\scriptsize] at (v2) {\textcolor{blue}{3}};  
\node[left=0.1cm,font=\scriptsize] at (v3) {\textcolor{blue}{2}};  
\node[above=0.1cm, left=0.1cm,font=\scriptsize] at (v4) {};  
\node[above=0.05cm,font=\scriptsize] at (v5) {};  
\node[above=0.1cm, font=\scriptsize] at (v6) {};  
\node[right=0.1cm, font=\scriptsize] at (v7) {};  
\node[right=0.1cm, font=\scriptsize] at (v8) {};  
\node[right=0.1cm, below=0.1cm, font=\scriptsize] at (v9) {\textcolor{blue}{2}};  
 \end{tikzpicture}
       \vfill {\small (c) The list sizes after coloring $v_i$: $i = 4,\dots, 8$}
\end{center}
\end{multicols} 
\caption{The graph  $F_1$: The list sizes after coloring some vertices of $F_1$.} 
\label{F1-color}
\end{figure*}

\item \textbf{Case $F_i = F_4$}.

As indicated in Figure~\ref{F4-color}(a), we have
\[
|L_{F_4}(v_i)| \ge 
\begin{cases}
3, & i \in \{4, 7, 9, 10, 11\},\\
4, & i \in \{1, 5, 6, 8\},\\
5, & i \in \{2, 3\}.
\end{cases}
\]

Using an argument similar to the proof that $T_1$ is not a subgraph of $G$ (see Figure~\ref{T123-color}(i)), we first color the vertices 
$v_{11}, v_4, v_5, v_6$, and $v_7$.

Let $L_{F_4}'(v_i)$ be the updated lists after coloring 
$v_4, v_5, v_6, v_7$, and $v_{11}$ (see Figure~\ref{F4-color}(b)).

Since $|L_{F_4}'(v_2)| \ge 4$ and $|L_{F_4}'(v_3)| \ge 3$, there exists a color 
$\alpha \in L_{F_4}'(v_2)$ such that 
$|L_{F_4}'(v_3) \setminus \{\alpha\}| \ge 3$.
Color $v_9$ with a color 
$c_9 \in L_{F_4}'(v_9) \setminus \{\alpha\}$, and then color $v_8$ with a color 
$c_8 \in L_{F_4}'(v_8) \setminus \{c_9\}$.

Let $L_{F_4}''(v_i)$ be the lists after coloring $v_8$ and $v_9$ (see Figure~\ref{F4-color}(c)).  
Then
\[
|L_{F_4}''(v_1)| \ge 2,\qquad
|L_{F_4}''(v_2)| \ge 3,\qquad
|L_{F_4}''(v_3)| \ge 3,\qquad
|L_{F_4}''(v_{10})| \ge 2.
\]

Moreover, since $\alpha \in L_{F_4}''(v_2)\setminus L_{F_4}''(v_3)$, we have 
$|L_{F_4}''(v_2) \cup L_{F_4}''(v_3)| \ge 4$.
Thus, one can further color the vertices 
$v_1, v_2, v_3$, and $v_{10}$ to obtain an $L$-coloring of $G^2$, a contradiction.  
This proves that $F_4$ is not a subgraph of $G$.

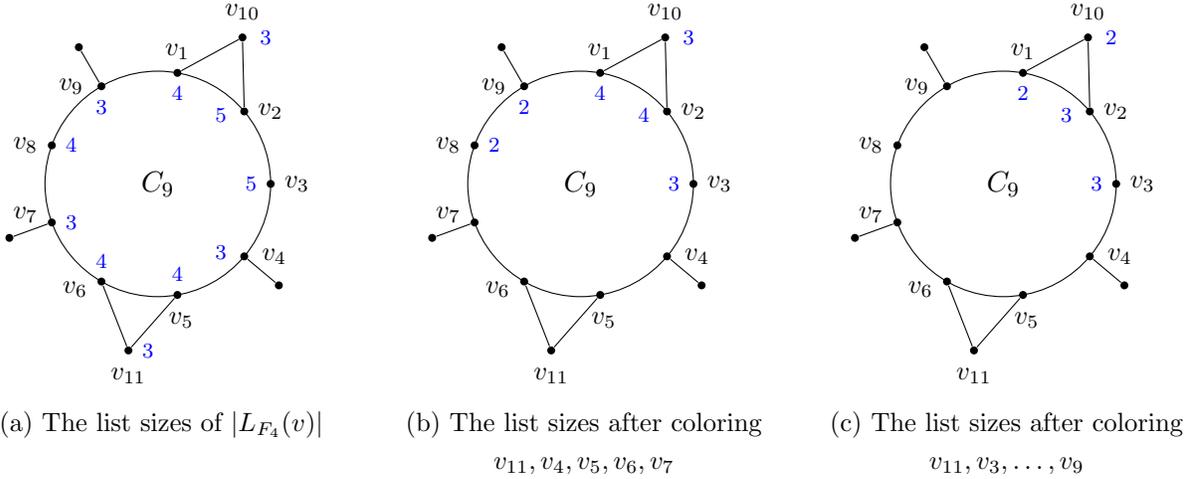
\begin{figure*}[htbp]
\begin{multicols}{3}
\begin{center}
\begin{tikzpicture}
[u/.style={fill=black, minimum size =3pt,ellipse,inner sep=1pt},node distance=1.5cm,scale=1.5]
\node[u] (v1) at (80:1){};
\node[u] (v2) at (40:1){};
\node[u] (v3) at (0:1){};
\node[u] (v4) at (320:1){};
\node[u] (v5) at (280:1){};
\node[u] (v6) at (240:1){};
\node[u] (v7) at (200:1){};
\node[u] (v8) at (160:1){};
\node[u] (v9) at (120:1){};
\node[u] (v10) at (60:1.5){};
\node[u] (v11) at (260:1.5){};
\node (C9) at (0, 0){$C_9$};

\draw   (0,0) circle[radius=1cm];
  \draw (v1) -- (v10);
  \draw (v2) -- (v10);  
  \draw (v5) -- (v11);
  \draw (v6) -- (v11);

\node[u] (w4) at (320:1.4){};
\node[u] (w7) at (200:1.4){};
\node[u] (w9) at (120:1.4){};
  \draw (v4) -- (w4);
  \draw (v7) -- (w7); 
 \draw (v9) -- (w9);   
    
   \node[above=0.05cm, font=\small] at (v1) {$v_1$};  
   \node[right=0.05cm,font=\small] at (v2) {$v_2$};  
   \node[right=0.05cm, font=\small] at (v3) {$v_3$};     
   \node[right=0.1cm,font=\small] at (v4) {$v_4$};     
    \node[right=0.05cm, below=0.1cm, font=\small] at (v5) {$v_5$};    
   \node[below=0.1cm,left=0.05cm, font=\small] at (v6) {$v_6$};      
   \node[above=0.1cm, left=0.05cm, font=\small] at (v7) {$v_7$};  
   \node[left=0.05cm, font=\small] at (v8) {$v_8$};    
    \node[left=0.1cm,font=\small] at (v9) {$v_9$};  
   \node[above=0.1cm,  font=\small] at (v10) {$v_{10}$};        
    \node[below=0.1cm,  font=\small] at (v11) {$v_{11}$};               
    
 \node[below=0.05cm,  font=\scriptsize] at (v1) {\textcolor{blue}{4}};
\node[below=0.05cm, left=0.1cm, font=\scriptsize] at (v2) {\textcolor{blue}{5}};  
\node[left=0.05cm,font=\scriptsize] at (v3) {\textcolor{blue}{5}};  
\node[above=0.05cm, left=0.1cm,font=\scriptsize] at (v4) {\textcolor{blue}{3}};  
\node[above=0.05cm,font=\scriptsize] at (v5) {\textcolor{blue}{4}};  
\node[above=0.05cm, font=\scriptsize] at (v6) {\textcolor{blue}{4}};  
\node[right=0.05cm, font=\scriptsize] at (v7) {\textcolor{blue}{3}};  
\node[right=0.05cm, font=\scriptsize] at (v8) {\textcolor{blue}{4}};  
\node[below=0.05cm, font=\scriptsize] at (v9) {\textcolor{blue}{3}};  
\node[right=0.1cm, font=\scriptsize] at (v10) {\textcolor{blue}{3}}; 
\node[right=0.05cm, font=\scriptsize] at (v11) {\textcolor{blue}{3}}; 
 \end{tikzpicture}
        \vfill {\small (a) The list sizes of $|L_{F_4}(v)|$ \ \ \ \ \ \ \ \ \ \ \ \ \ \ \ \ \ \  } 
\end{center}
\par
\begin{center}
\begin{tikzpicture}
[u/.style={fill=black, minimum size =3pt,ellipse,inner sep=1pt},node distance=1.5cm,scale=1.5]
\node[u] (v1) at (80:1){};
\node[u] (v2) at (40:1){};
\node[u] (v3) at (0:1){};
\node[u] (v4) at (320:1){};
\node[u] (v5) at (280:1){};
\node[u] (v6) at (240:1){};
\node[u] (v7) at (200:1){};
\node[u] (v8) at (160:1){};
\node[u] (v9) at (120:1){};
\node[u] (v10) at (60:1.5){};
\node[u] (v11) at (260:1.5){};
\node (C9) at (0, 0){$C_9$};

\draw   (0,0) circle[radius=1cm];
  \draw (v1) -- (v10);
  \draw (v2) -- (v10);  
  \draw (v5) -- (v11);
  \draw (v6) -- (v11);

\node[u] (w4) at (320:1.4){};
\node[u] (w7) at (200:1.4){};
\node[u] (w9) at (120:1.4){};
  \draw (v4) -- (w4);
  \draw (v7) -- (w7); 
 \draw (v9) -- (w9);   
    
   \node[above=0.05cm, font=\small] at (v1) {$v_1$};  
   \node[right=0.05cm,font=\small] at (v2) {$v_2$};  
   \node[right=0.05cm, font=\small] at (v3) {$v_3$};     
   \node[right=0.1cm,font=\small] at (v4) {$v_4$};     
    \node[right=0.05cm, below=0.1cm, font=\small] at (v5) {$v_5$};    
   \node[below=0.1cm,left=0.05cm, font=\small] at (v6) {$v_6$};      
   \node[above=0.1cm, left=0.05cm, font=\small] at (v7) {$v_7$};  
   \node[left=0.05cm, font=\small] at (v8) {$v_8$};    
    \node[left=0.1cm,font=\small] at (v9) {$v_9$};  
   \node[above=0.1cm,  font=\small] at (v10) {$v_{10}$};        
    \node[below=0.1cm,  font=\small] at (v11) {$v_{11}$};               
    
 \node[below=0.05cm,  font=\scriptsize] at (v1) {\textcolor{blue}{4}};
\node[below=0.05cm, left=0.1cm, font=\scriptsize] at (v2) {\textcolor{blue}{4}};  
\node[left=0.05cm,font=\scriptsize] at (v3) {\textcolor{blue}{3}};  
\node[above=0.05cm, left=0.1cm,font=\scriptsize] at (v4) {};  
\node[above=0.05cm,font=\scriptsize] at (v5) {};  
\node[above=0.05cm, font=\scriptsize] at (v6) {};  
\node[right=0.05cm, font=\scriptsize] at (v7) {};  
\node[right=0.05cm, font=\scriptsize] at (v8) {\textcolor{blue}{2}};  
\node[below=0.05cm, font=\scriptsize] at (v9) {\textcolor{blue}{2}};  
\node[right=0.1cm, font=\scriptsize] at (v10) {\textcolor{blue}{3}}; 
\node[right=0.05cm, font=\scriptsize] at (v11) {}; 
 \end{tikzpicture}
       \vfill {\small (b) The list sizes after coloring $v_{11}, v_4, v_5, v_6, v_7$} 
\end{center}
\par
\begin{center}
\begin{tikzpicture}
[u/.style={fill=black, minimum size =3pt,ellipse,inner sep=1pt},node distance=1.5cm,scale=1.5]
\node[u] (v1) at (80:1){};
\node[u] (v2) at (40:1){};
\node[u] (v3) at (0:1){};
\node[u] (v4) at (320:1){};
\node[u] (v5) at (280:1){};
\node[u] (v6) at (240:1){};
\node[u] (v7) at (200:1){};
\node[u] (v8) at (160:1){};
\node[u] (v9) at (120:1){};
\node[u] (v10) at (60:1.5){};
\node[u] (v11) at (260:1.5){};
\node (C9) at (0, 0){$C_9$};

\draw   (0,0) circle[radius=1cm];
  \draw (v1) -- (v10);
  \draw (v2) -- (v10);  
  \draw (v5) -- (v11);
  \draw (v6) -- (v11);

\node[u] (w4) at (320:1.4){};
\node[u] (w7) at (200:1.4){};
\node[u] (w9) at (120:1.4){};
  \draw (v4) -- (w4);
  \draw (v7) -- (w7); 
 \draw (v9) -- (w9);   
    
   \node[above=0.05cm, font=\small] at (v1) {$v_1$};  
   \node[right=0.05cm,font=\small] at (v2) {$v_2$};  
   \node[right=0.05cm, font=\small] at (v3) {$v_3$};     
   \node[right=0.1cm,font=\small] at (v4) {$v_4$};     
    \node[right=0.05cm, below=0.1cm, font=\small] at (v5) {$v_5$};    
   \node[below=0.1cm,left=0.05cm, font=\small] at (v6) {$v_6$};      
   \node[above=0.1cm, left=0.05cm, font=\small] at (v7) {$v_7$};  
   \node[left=0.05cm, font=\small] at (v8) {$v_8$};    
    \node[left=0.1cm,font=\small] at (v9) {$v_9$};  
   \node[above=0.1cm,  font=\small] at (v10) {$v_{10}$};        
    \node[below=0.1cm,  font=\small] at (v11) {$v_{11}$};               
    
 \node[below=0.05cm,  font=\scriptsize] at (v1) {\textcolor{blue}{2}};
\node[below=0.05cm, left=0.1cm, font=\scriptsize] at (v2) {\textcolor{blue}{3}};  
\node[left=0.05cm,font=\scriptsize] at (v3) {\textcolor{blue}{3}};  
\node[above=0.05cm, left=0.1cm,font=\scriptsize] at (v4) {};  
\node[above=0.05cm,font=\scriptsize] at (v5) {};  
\node[above=0.05cm, font=\scriptsize] at (v6) {};  
\node[right=0.05cm, font=\scriptsize] at (v7) {};  
\node[right=0.05cm, font=\scriptsize] at (v8) {}; 
\node[below=0.05cm, font=\scriptsize] at (v9) {}; 
\node[right=0.1cm, font=\scriptsize] at (v10) {\textcolor{blue}{2}}; 
\node[right=0.05cm, font=\scriptsize] at (v11) {}; 
 \end{tikzpicture}
       \vfill {\small (c) The list sizes after coloring $v_{11}, v_3, \dots, v_9$}
\end{center}
\end{multicols} 
\caption{The graph  $F_4$: The list sizes after coloring some vertices of $F_4$.} 
\label{F4-color}
\end{figure*}

\item \textbf{Case $F_i = F_6$}.

As indicated in Figure~\ref{F6-color}(a), we have
\[
|L_{F_6}(v_i)| \ge 
\begin{cases}
2, & i \in \{4,5\},\\
3, & i \in \{9,10,11\},\\
4, & i \in \{1,6\},\\
5, & i \in \{2,3,7,8\}.\\
\end{cases}
\]

Since $|L_{F_6}(v_6)| \ge 4$ and $|L_{F_6}(v_{11})| \ge 3$, there exists a color 
$\alpha \in L_{F_6}(v_6)$ such that 
\[
|L_{F_6}(v_{11}) \setminus \{\alpha\}| \ge 3.
\]
Color $v_4$ by a color $c_4 \in L_{F_6}(v_4) \setminus \{\alpha\}$, and then color $v_5$ by a color 
$c_5 \in L_{F_6}(v_5) \setminus \{c_4\}$.

Let $L_{F_6}'(v_i)$ be the updated lists after coloring $v_4$ and $v_5$ (see Figure~\ref{F6-color}(b)).  
By a similar argument to the one used in Case~$F_1$, we have
\[
|L_{F_6}'(v_6) \cup L_{F_6}'(v_{11})| \ge 3.
\]

Now, since $|L_{F_6}'(v_8)| \ge 5$ and $|L_{F_6}'(v_7)| \ge 4$, there exists a color  
$c_8 \in L_{F_6}'(v_8)$ such that  
\[
|L_{F_6}'(v_7) \setminus \{c_8\}| \ge 4.
\]
Color $v_8$ with $c_8$.  
Since $|L_{F_6}'(v_6) \cup L_{F_6}'(v_{11})| \ge 3$, we have
\[
|L_{F_6}'(v_6) \setminus \{c_8\} \cup L_{F_6}'(v_{11}) \setminus \{c_8\}| \ge 2,
\]
and hence we may color $v_6$ and $v_{11}$.

Let $L_{F_6}''(v_i)$ be the updated lists after coloring $v_6$, $v_{11}$, and $v_8$ 
(see Figure~\ref{F6-color}(c)).  
Then
\[
|L_{F_6}''(v_1)| \ge 3,\quad
|L_{F_6}''(v_2)| \ge 4,\quad
|L_{F_6}''(v_3)| \ge 3,\quad
|L_{F_6}''(v_9)| \ge 2,\quad
|L_{F_6}''(v_{10})| \ge 3,\quad
|L_{F_6}''(v_7)| \ge 2.
\]

We first color $v_1$, $v_2$, $v_3$, $v_9$, and $v_{10}$ by Lemma~\ref{lem-key-J1}, and then color $v_7$ to obtain an $L$-coloring of $G^2$, a contradiction.  
This proves that $F_6$ is not a subgraph of $G$.
\begin{figure*}[htbp]
\begin{multicols}{3}
\begin{center}
\begin{tikzpicture}
[u/.style={fill=black, minimum size =3pt,ellipse,inner sep=1pt},node distance=1.5cm,scale=1.5]
\node[u] (v1) at (80:1){};
\node[u] (v2) at (40:1){};
\node[u] (v3) at (0:1){};
\node[u] (v4) at (320:1){};
\node[u] (v5) at (280:1){};
\node[u] (v6) at (240:1){};
\node[u] (v7) at (200:1){};
\node[u] (v8) at (160:1){};
\node[u] (v9) at (120:1){};
\node[u] (v10) at (60:1.5){};
\node[u] (v11) at (220:1.5){};
\node (C9) at (0, 0){$C_9$};

\draw   (0,0) circle[radius=1cm];
  \draw (v1) -- (v10);
  \draw (v2) -- (v10);   
  \draw (v6) -- (v11);
  \draw (v7) -- (v11);

\node[u] (w4) at (320:1.4){};
\node[u] (w5) at (280:1.4){};
\node[u] (w9) at (120:1.4){};
  \draw (v4) -- (w4);
  \draw (v5) -- (w5);
 \draw (v9) -- (w9);   
    
   \node[above=0.1cm, font=\small] at (v1) {$v_1$};  
   \node[right=0.1cm,font=\small] at (v2) {$v_2$};  
   \node[right=0.1cm, font=\small] at (v3) {$v_3$};     
   \node[right=0.1cm,font=\small] at (v4) {$v_4$};     
    \node[below=0.2cm,right=0.01cm, font=\small] at (v5) {$v_5$};    
   \node[below=0.1cm,font=\small] at (v6) {$v_6$};      
   \node[above=0.05cm, left=0.05cm, font=\small] at (v7) {$v_7$};  
   \node[left=0.1cm, font=\small] at (v8) {$v_8$};    
    \node[left=0.1cm,font=\small] at (v9) {$v_9$};  
   \node[above=0.1cm,  font=\small] at (v10) {$v_{10}$};        
    \node[below=0.1cm, font=\small] at (v11) {$v_{11}$};

 \node[below=0.05cm,  font=\scriptsize] at (v1) {\textcolor{blue}{4}};
\node[below=0.05cm, left=0.1cm, font=\scriptsize] at (v2) {\textcolor{blue}{5}};  
\node[left=0.05cm,font=\scriptsize] at (v3) {\textcolor{blue}{5}};  
\node[above=0.05cm, left=0.1cm,font=\scriptsize] at (v4) {\textcolor{blue}{2}};  
\node[above=0.05cm,font=\scriptsize] at (v5) {\textcolor{blue}{2}};  
\node[above=0.05cm, font=\scriptsize] at (v6) {\textcolor{blue}{4}};  
\node[right=0.05cm, font=\scriptsize] at (v7) {\textcolor{blue}{5}};  
\node[right=0.05cm, font=\scriptsize] at (v8) {\textcolor{blue}{5}};  
\node[below=0.05cm, font=\scriptsize] at (v9) {\textcolor{blue}{3}};  
\node[ below=0.1cm, right=0.1cm, font=\scriptsize] at (v10) {\textcolor{blue}{3}}; 
\node[left=0.1cm, font=\scriptsize] at (v11) {\textcolor{blue}{3}}; 
 \end{tikzpicture}
       \vfill {\small (a) The list sizes of $|L_{F_6}(v)|$  \ \ \ \ \ \ \ \ \ \ \ \ \ \ \ \ \ \ } 
 
\end{center}
\par
\begin{center}
\begin{tikzpicture}
[u/.style={fill=black, minimum size =3pt,ellipse,inner sep=1pt},node distance=1.5cm,scale=1.5]
\node[u] (v1) at (80:1){};
\node[u] (v2) at (40:1){};
\node[u] (v3) at (0:1){};
\node[u] (v4) at (320:1){};
\node[u] (v5) at (280:1){};
\node[u] (v6) at (240:1){};
\node[u] (v7) at (200:1){};
\node[u] (v8) at (160:1){};
\node[u] (v9) at (120:1){};
\node[u] (v10) at (60:1.5){};
\node[u] (v11) at (220:1.5){};
\node (C9) at (0, 0){$C_9$};

\draw   (0,0) circle[radius=1cm];
  \draw (v1) -- (v10);
  \draw (v2) -- (v10);   
  \draw (v6) -- (v11);
  \draw (v7) -- (v11);

\node[u] (w4) at (320:1.4){};
\node[u] (w5) at (280:1.4){};
\node[u] (w9) at (120:1.4){};
  \draw (v4) -- (w4);
  \draw (v5) -- (w5);
 \draw (v9) -- (w9);   
    
   \node[above=0.1cm, font=\small] at (v1) {$v_1$};  
   \node[right=0.1cm,font=\small] at (v2) {$v_2$};  
   \node[right=0.1cm, font=\small] at (v3) {$v_3$};     
   \node[right=0.1cm,font=\small] at (v4) {$v_4$};     
    \node[below=0.2cm,right=0.01cm, font=\small] at (v5) {$v_5$};    
   \node[below=0.1cm,font=\small] at (v6) {$v_6$};      
   \node[above=0.05cm, left=0.05cm, font=\small] at (v7) {$v_7$};  
   \node[left=0.1cm, font=\small] at (v8) {$v_8$};    
    \node[left=0.1cm,font=\small] at (v9) {$v_9$};  
   \node[above=0.1cm,  font=\small] at (v10) {$v_{10}$};        
    \node[below=0.1cm, font=\small] at (v11) {$v_{11}$};

 \node[below=0.05cm,  font=\scriptsize] at (v1) {\textcolor{blue}{4}};
\node[below=0.05cm, left=0.1cm, font=\scriptsize] at (v2) {\textcolor{blue}{4}};  
\node[left=0.05cm,font=\scriptsize] at (v3) {\textcolor{blue}{3}};  
\node[above=0.05cm, left=0.1cm,font=\scriptsize] at (v4) {};  
\node[above=0.05cm,font=\scriptsize] at (v5) {};  
\node[above=0.05cm, font=\scriptsize] at (v6) {\textcolor{blue}{2}};  
\node[right=0.05cm, font=\scriptsize] at (v7) {\textcolor{blue}{4}};  
\node[right=0.05cm, font=\scriptsize] at (v8) {\textcolor{blue}{5}};  
\node[below=0.05cm, font=\scriptsize] at (v9) {\textcolor{blue}{3}};  
\node[ below=0.1cm, right=0.1cm, font=\scriptsize] at (v10) {\textcolor{blue}{3}}; 
\node[left=0.1cm, font=\scriptsize] at (v11) {\textcolor{blue}{2}}; 
 \end{tikzpicture}
        \vfill {\small (b) The list sizes after coloring $v_4, v_5$}  
\end{center}
\par
\begin{center}
\begin{tikzpicture}
[u/.style={fill=black, minimum size =3pt,ellipse,inner sep=1pt},node distance=1.5cm,scale=1.5]
\node[u] (v1) at (80:1){};
\node[u] (v2) at (40:1){};
\node[u] (v3) at (0:1){};
\node[u] (v4) at (320:1){};
\node[u] (v5) at (280:1){};
\node[u] (v6) at (240:1){};
\node[u] (v7) at (200:1){};
\node[u] (v8) at (160:1){};
\node[u] (v9) at (120:1){};
\node[u] (v10) at (60:1.5){};
\node[u] (v11) at (220:1.5){};
\node (C9) at (0, 0){$C_9$};

\draw   (0,0) circle[radius=1cm];
  \draw (v1) -- (v10);
  \draw (v2) -- (v10);   
  \draw (v6) -- (v11);
  \draw (v7) -- (v11);

\node[u] (w4) at (320:1.4){};
\node[u] (w5) at (280:1.4){};
\node[u] (w9) at (120:1.4){};
  \draw (v4) -- (w4);
  \draw (v5) -- (w5);
 \draw (v9) -- (w9);   
    
   \node[above=0.1cm, font=\small] at (v1) {$v_1$};  
   \node[right=0.1cm,font=\small] at (v2) {$v_2$};  
   \node[right=0.1cm, font=\small] at (v3) {$v_3$};     
   \node[right=0.1cm,font=\small] at (v4) {$v_4$};     
    \node[below=0.2cm,right=0.01cm, font=\small] at (v5) {$v_5$};    
   \node[below=0.1cm,font=\small] at (v6) {$v_6$};      
   \node[above=0.05cm, left=0.05cm, font=\small] at (v7) {$v_7$};  
   \node[left=0.1cm, font=\small] at (v8) {$v_8$};    
    \node[left=0.1cm,font=\small] at (v9) {$v_9$};  
   \node[above=0.1cm,  font=\small] at (v10) {$v_{10}$};        
    \node[below=0.1cm, font=\small] at (v11) {$v_{11}$};

 \node[below=0.05cm,  font=\scriptsize] at (v1) {\textcolor{blue}{3}};
\node[below=0.05cm, left=0.1cm, font=\scriptsize] at (v2) {\textcolor{blue}{4}};  
\node[left=0.05cm,font=\scriptsize] at (v3) {\textcolor{blue}{3}};  
\node[above=0.05cm, left=0.1cm,font=\scriptsize] at (v4) {};  
\node[above=0.05cm,font=\scriptsize] at (v5) {};  
\node[right=0.05cm, font=\scriptsize] at (v7) {\textcolor{blue}{2}};  
\node[right=0.05cm, font=\scriptsize] at (v8) {};  
\node[below=0.05cm, font=\scriptsize] at (v9) {\textcolor{blue}{2}};  
\node[ below=0.1cm, right=0.1cm, font=\scriptsize] at (v10) {\textcolor{blue}{3}}; 
 \end{tikzpicture}
         \vfill {\small (c) The list sizes after coloring $v_4, v_5,v_8$}  
\end{center}
\end{multicols} 
\caption{The graph $F_6$: The list sizes after coloring some vertices of $F_6$.}
\label{F6-color}
\end{figure*}
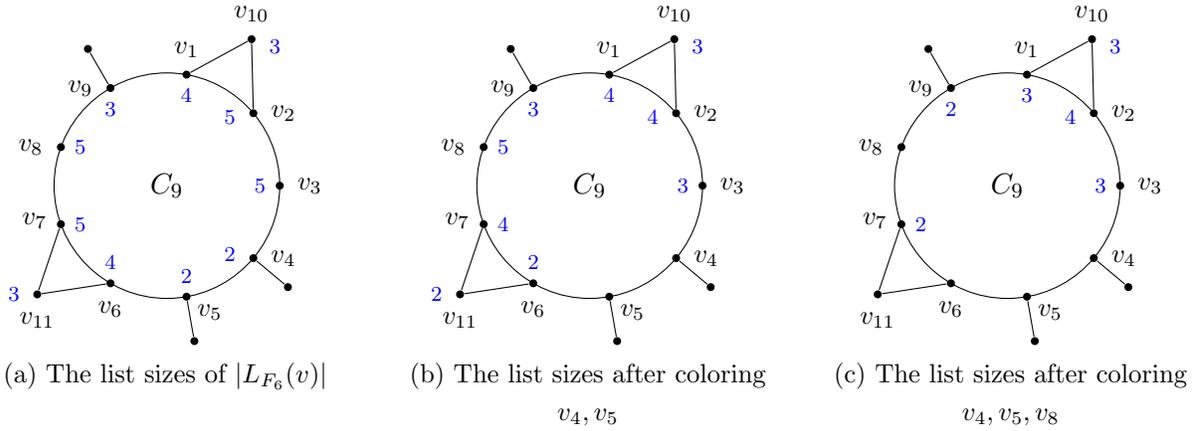

\item \textbf{Case $F_i = F_7$}.

As indicated in Figure~\ref{F7-color}(a), we have:
\[
|L_{F_7}(v_i)| \ge 
\begin{cases}
2, & i\in \{4,5\},\\
3, & i \in \{9,10,11\},\\
4, & i \in \{1,8\},\\
5, & i \in \{2,3,6,7\}.\\
\end{cases}
\]

First color $v_4$ and $v_5$ greedily.  
Let $L_{F_7}'(v_i)$ be the lists after coloring $v_4$ and $v_5$ (see Figure~\ref{F7-color}(b)).

Now color $v_8$ by a color $c_8 \in L_{F_7}'(v_8)$ such that 
\[
|L_{F_7}'(v_6) \setminus \{c_8\}| \ge 3.
\]
Let $L_{F_7}''(v_i)$ be the lists after coloring $v_4$, $v_5$, and $v_8$ (see Figure~\ref{F7-color}(c)).

By Lemma~\ref{lem-key-J1}, color the vertices $v_1, v_2, v_3, v_9$, and $v_{10}$.  
Then color $v_{11}$, $v_7$, and $v_6$ in this order to yield an $L$-coloring of $G^2$, a contradiction.  
Thus, $F_7$ is not a subgraph of $G$.
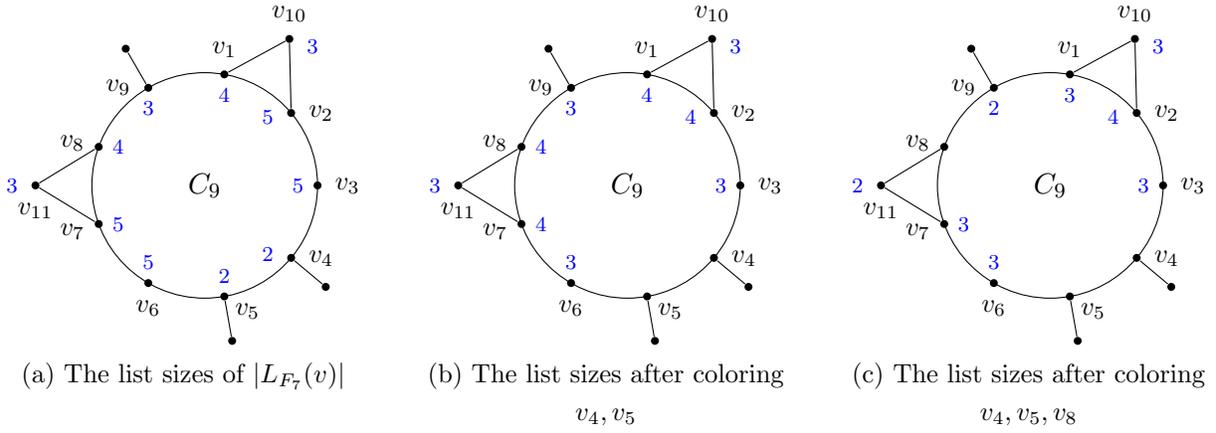
\begin{figure*}[htbp]
\begin{multicols}{3}
\begin{center}
\begin{tikzpicture}
[u/.style={fill=black, minimum size =3pt,ellipse,inner sep=1pt},node distance=1.5cm,scale=1.5]
\node[u] (v1) at (80:1){};
\node[u] (v2) at (40:1){};
\node[u] (v3) at (0:1){};
\node[u] (v4) at (320:1){};
\node[u] (v5) at (280:1){};
\node[u] (v6) at (240:1){};
\node[u] (v7) at (200:1){};
\node[u] (v8) at (160:1){};
\node[u] (v9) at (120:1){};
\node[u] (v10) at (60:1.5){};
\node[u] (v11) at (180:1.5){};
\node (C9) at (0, 0){$C_9$};

\draw   (0,0) circle[radius=1cm];
  \draw (v1) -- (v10);
  \draw (v2) -- (v10);   
  \draw (v7) -- (v11);
  \draw (v8) -- (v11);

\node[u] (w4) at (320:1.4){};
\node[u] (w5) at (280:1.4){};
\node[u] (w9) at (120:1.4){};
  \draw (v4) -- (w4);
  \draw (v5) -- (w5);
 \draw (v9) -- (w9);   
    
   \node[above=0.1cm, font=\small] at (v1) {$v_1$};  
   \node[right=0.1cm,font=\small] at (v2) {$v_2$};  
   \node[right=0.1cm, font=\small] at (v3) {$v_3$};     
   \node[right=0.1cm,font=\small] at (v4) {$v_4$};     
    \node[below=0.2cm,right=0.01cm, font=\small] at (v5) {$v_5$};    
   \node[below=0.1cm,font=\small] at (v6) {$v_6$};      
   \node[below=0.1cm, left=0.05cm, font=\small] at (v7) {$v_7$};  
   \node[above=0.05cm, left=0.05cm, font=\small] at (v8) {$v_8$};    
    \node[left=0.1cm,font=\small] at (v9) {$v_9$};  
   \node[above=0.1cm,  font=\small] at (v10) {$v_{10}$};        
    \node[below=0.1cm, font=\small] at (v11) {$v_{11}$};

 \node[below=0.05cm,  font=\scriptsize] at (v1) {\textcolor{blue}{4}};
\node[below=0.05cm, left=0.1cm, font=\scriptsize] at (v2) {\textcolor{blue}{5}};  
\node[left=0.05cm,font=\scriptsize] at (v3) {\textcolor{blue}{5}};  
\node[above=0.05cm, left=0.1cm,font=\scriptsize] at (v4) {\textcolor{blue}{2}};  
\node[above=0.05cm,font=\scriptsize] at (v5) {\textcolor{blue}{2}};  
\node[above=0.05cm, font=\scriptsize] at (v6) {\textcolor{blue}{5}};  
\node[right=0.05cm, font=\scriptsize] at (v7) {\textcolor{blue}{5}};  
\node[right=0.05cm, font=\scriptsize] at (v8) {\textcolor{blue}{4}};  
\node[below=0.05cm, font=\scriptsize] at (v9) {\textcolor{blue}{3}};  
\node[ below=0.1cm, right=0.1cm, font=\scriptsize] at (v10) {\textcolor{blue}{3}}; 
\node[left=0.1cm, font=\scriptsize] at (v11) {\textcolor{blue}{3}}; 
 \end{tikzpicture}
        \vfill {\small (a) The list sizes of $|L_{F_7}(v)|$  \ \ \ \ \ \ \ \ \ \ \ \ \ \ \ \ \ \ }  
\end{center}
\par
\begin{center}
\begin{tikzpicture}
[u/.style={fill=black, minimum size =3pt,ellipse,inner sep=1pt},node distance=1.5cm,scale=1.5]
\node[u] (v1) at (80:1){};
\node[u] (v2) at (40:1){};
\node[u] (v3) at (0:1){};
\node[u] (v4) at (320:1){};
\node[u] (v5) at (280:1){};
\node[u] (v6) at (240:1){};
\node[u] (v7) at (200:1){};
\node[u] (v8) at (160:1){};
\node[u] (v9) at (120:1){};
\node[u] (v10) at (60:1.5){};
\node[u] (v11) at (180:1.5){};
\node (C9) at (0, 0){$C_9$};

\draw   (0,0) circle[radius=1cm];
  \draw (v1) -- (v10);
  \draw (v2) -- (v10);   
  \draw (v7) -- (v11);
  \draw (v8) -- (v11);

\node[u] (w4) at (320:1.4){};
\node[u] (w5) at (280:1.4){};
\node[u] (w9) at (120:1.4){};
  \draw (v4) -- (w4);
  \draw (v5) -- (w5);
 \draw (v9) -- (w9);   
    
   \node[above=0.1cm, font=\small] at (v1) {$v_1$};  
   \node[right=0.1cm,font=\small] at (v2) {$v_2$};  
   \node[right=0.1cm, font=\small] at (v3) {$v_3$};     
   \node[right=0.1cm,font=\small] at (v4) {$v_4$};     
    \node[below=0.2cm,right=0.01cm, font=\small] at (v5) {$v_5$};    
   \node[below=0.1cm,font=\small] at (v6) {$v_6$};      
   \node[below=0.1cm, left=0.05cm, font=\small] at (v7) {$v_7$};  
   \node[above=0.05cm, left=0.05cm, font=\small] at (v8) {$v_8$};    
    \node[left=0.1cm,font=\small] at (v9) {$v_9$};  
   \node[above=0.1cm,  font=\small] at (v10) {$v_{10}$};        
    \node[below=0.1cm, font=\small] at (v11) {$v_{11}$};

 \node[below=0.05cm,  font=\scriptsize] at (v1) {\textcolor{blue}{4}};
\node[below=0.05cm, left=0.1cm, font=\scriptsize] at (v2) {\textcolor{blue}{4}};  
\node[left=0.05cm,font=\scriptsize] at (v3) {\textcolor{blue}{3}};  
\node[above=0.05cm, left=0.1cm,font=\scriptsize] at (v4) {};  
\node[above=0.05cm,font=\scriptsize] at (v5) {};  
\node[above=0.05cm, font=\scriptsize] at (v6) {\textcolor{blue}{3}};  
\node[right=0.05cm, font=\scriptsize] at (v7) {\textcolor{blue}{4}};  
\node[right=0.05cm, font=\scriptsize] at (v8) {\textcolor{blue}{4}};  
\node[below=0.05cm, font=\scriptsize] at (v9) {\textcolor{blue}{3}};  
\node[ below=0.1cm, right=0.1cm, font=\scriptsize] at (v10) {\textcolor{blue}{3}}; 
\node[left=0.1cm, font=\scriptsize] at (v11) {\textcolor{blue}{3}}; 
 \end{tikzpicture}
         \vfill {\small (b) The list sizes after coloring $v_4, v_5$} 

\end{center}
\par
\begin{center}
\begin{tikzpicture}
[u/.style={fill=black, minimum size =3pt,ellipse,inner sep=1pt},node distance=1.5cm,scale=1.5]
\node[u] (v1) at (80:1){};
\node[u] (v2) at (40:1){};
\node[u] (v3) at (0:1){};
\node[u] (v4) at (320:1){};
\node[u] (v5) at (280:1){};
\node[u] (v6) at (240:1){};
\node[u] (v7) at (200:1){};
\node[u] (v8) at (160:1){};
\node[u] (v9) at (120:1){};
\node[u] (v10) at (60:1.5){};
\node[u] (v11) at (180:1.5){};
\node (C9) at (0, 0){$C_9$};

\draw   (0,0) circle[radius=1cm];
  \draw (v1) -- (v10);
  \draw (v2) -- (v10);   
  \draw (v7) -- (v11);
  \draw (v8) -- (v11);

\node[u] (w4) at (320:1.4){};
\node[u] (w5) at (280:1.4){};
\node[u] (w9) at (120:1.4){};
  \draw (v4) -- (w4);
  \draw (v5) -- (w5);
 \draw (v9) -- (w9);   
    
   \node[above=0.1cm, font=\small] at (v1) {$v_1$};  
   \node[right=0.1cm,font=\small] at (v2) {$v_2$};  
   \node[right=0.1cm, font=\small] at (v3) {$v_3$};     
   \node[right=0.1cm,font=\small] at (v4) {$v_4$};     
    \node[below=0.2cm,right=0.01cm, font=\small] at (v5) {$v_5$};    
   \node[below=0.1cm,font=\small] at (v6) {$v_6$};      
   \node[below=0.1cm, left=0.05cm, font=\small] at (v7) {$v_7$};  
   \node[above=0.05cm, left=0.05cm, font=\small] at (v8) {$v_8$};    
    \node[left=0.1cm,font=\small] at (v9) {$v_9$};  
   \node[above=0.1cm,  font=\small] at (v10) {$v_{10}$};        
    \node[below=0.1cm, font=\small] at (v11) {$v_{11}$};

 \node[below=0.05cm,  font=\scriptsize] at (v1) {\textcolor{blue}{3}};
\node[below=0.05cm, left=0.1cm, font=\scriptsize] at (v2) {\textcolor{blue}{4}};  
\node[left=0.05cm,font=\scriptsize] at (v3) {\textcolor{blue}{3}};  
\node[above=0.05cm, left=0.1cm,font=\scriptsize] at (v4) {};  
\node[above=0.05cm,font=\scriptsize] at (v5) {};  
\node[above=0.05cm, font=\scriptsize] at (v6) {\textcolor{blue}{3}};  
\node[right=0.05cm, font=\scriptsize] at (v7) {\textcolor{blue}{3}};  
\node[right=0.05cm, font=\scriptsize] at (v8) {};  
\node[below=0.05cm, font=\scriptsize] at (v9) {\textcolor{blue}{2}};  
\node[ below=0.1cm, right=0.1cm, font=\scriptsize] at (v10) {\textcolor{blue}{3}}; 
\node[left=0.1cm, font=\scriptsize] at (v11) {\textcolor{blue}{2}}; 
 \end{tikzpicture}
         \vfill {\small (c) The list sizes after coloring $v_4, v_5, v_8$}
\end{center}
\end{multicols} 
\caption{The graph $F_7$. The list sizes after coloring some vertices of $F_7$.}
\label{F7-color}
\end{figure*}

\item \textbf{Case $F_i = F_8$}.

As indicated in Figure~\ref{F8-color}(a), we have
\[
|L_{F_8}(v_i)| \ge 
\begin{cases}
2, & i = 6,7,\\
3, & i = 9,10,11,\\
4, & i = 1,5,8,\\
5, & i = 2,4,\\
6, & i = 3.\\
\end{cases}
\]

By an argument similar to that used in Case~$F_1$ (in the choice of $c_7$ and $c_8$),  
there exist two distinct colors 
\[
c_6 \in L_{F_8}(v_6), \qquad 
c_7 \in L_{F_8}(v_7),
\]
such that
\[
\left| L_{F_8}(v_5)\setminus\{c_6,c_7\} \;\cup\; L_{F_8}(v_{11})\setminus\{c_6\} \right| \ge 3.
\]

We first color $v_6$ and $v_7$ with $c_6$ and $c_7$, respectively, and then color  
$v_8$, $v_9$, $v_1$, and $v_{10}$ in this order.

Let $L_{F_8}'(v_i)$ be the updated lists after coloring  
$v_1$, $v_6$, $v_7$, $v_8$, $v_9$, and $v_{10}$ (see Figure~\ref{F8-color}(b)).  
Then
\[
|L_{F_8}'(v_5) \cup L_{F_8}'(v_{11})| \ge 3.
\]

Thus there exists a color $c_{11} \in L_{F_8}'(v_{11})$ such that  
\[
|L_{F_8}'(v_5) \setminus \{c_{11}\}| \ge 2.
\]
Color $v_{11}$ with $c_{11}$.

Let $L_{F_8}''(v_i)$ be the updated lists after coloring $v_{11}$  
(see Figure~\ref{F8-color}(c)).  
Then
\[
|L_{F_8}''(v_2)| \ge 2,\qquad
|L_{F_8}''(v_3)| \ge 3,\qquad
|L_{F_8}''(v_4)| \ge 3,\qquad
|L_{F_8}''(v_5)| \ge 2.
\]

We may now color $v_2$, $v_3$, $v_4$, and $v_5$ by Lemma~\ref{lem-key-K4-edge} to obtain an $L$-coloring of $G^2$, a contradiction.  
Thus, $F_8$ is not a subgraph of $G$.

\begin{figure*}[htbp]
\begin{multicols}{3}
\begin{center}
\begin{tikzpicture}
[u/.style={fill=black, minimum size =3pt,ellipse,inner sep=1pt},node distance=1.5cm,scale=1.5]
\node[u] (v1) at (80:1){};
\node[u] (v2) at (40:1){};
\node[u] (v3) at (0:1){};
\node[u] (v4) at (320:1){};
\node[u] (v5) at (280:1){};
\node[u] (v6) at (240:1){};
\node[u] (v7) at (200:1){};
\node[u] (v8) at (160:1){};
\node[u] (v9) at (120:1){};
\node[u] (v10) at (60:1.5){};
\node[u] (v11) at (300:1.5){};
\node (C9) at (0, 0){$C_9$};

\draw   (0,0) circle[radius=1cm];
  \draw (v1) -- (v10);
  \draw (v2) -- (v10);  
  \draw (v4) -- (v11);
  \draw (v5) -- (v11);

\node[u] (w6) at (240:1.4){};
\node[u] (w7) at (200:1.4){};
\node[u] (w9) at (120:1.4){};
  \draw (v6) -- (w6); 
  \draw (v7) -- (w7); 
 \draw (v9) -- (w9);   
    
   \node[above=0.05cm, font=\small] at (v1) {$v_1$};  
   \node[right=0.1cm,font=\small] at (v2) {$v_2$};  
   \node[right=0.05cm, font=\small] at (v3) {$v_3$};     
   \node[right=0.1cm,font=\small] at (v4) {$v_4$};     
    \node[below=0.1cm,  font=\small] at (v5) {$v_5$};    
   \node[left=0.1cm,font=\small] at (v6) {$v_6$};      
   \node[above=0.2cm, left=0.05cm, font=\small] at (v7) {$v_7$};  
   \node[left=0.05cm, font=\small] at (v8) {$v_8$};    
    \node[left=0.05cm,font=\small] at (v9) {$v_9$};  
   \node[above=0.1cm,  font=\small] at (v10) {$v_{10}$};        
    \node[below=0.1cm, font=\small] at (v11) {$v_{11}$};

 \node[below=0.05cm,  font=\scriptsize] at (v1) {\textcolor{blue}{4}};
\node[below=0.05cm, left=0.1cm, font=\scriptsize] at (v2) {\textcolor{blue}{5}};  
\node[left=0.05cm,font=\scriptsize] at (v3) {\textcolor{blue}{6}};  
\node[above=0.05cm, left=0.1cm,font=\scriptsize] at (v4) {\textcolor{blue}{5}};  
\node[above=0.05cm,font=\scriptsize] at (v5) {\textcolor{blue}{4}};  
\node[above=0.05cm, font=\scriptsize] at (v6) {\textcolor{blue}{2}};  
\node[right=0.05cm, font=\scriptsize] at (v7) {\textcolor{blue}{2}};  
\node[right=0.05cm, font=\scriptsize] at (v8) {\textcolor{blue}{4}};  
\node[below=0.05cm, font=\scriptsize] at (v9) {\textcolor{blue}{3}};  
\node[ below=0.05cm, right=0.1cm, font=\scriptsize] at (v10) {\textcolor{blue}{3}}; 
\node[right=0.05cm, font=\scriptsize] at (v11) {\textcolor{blue}{3}}; 
 \end{tikzpicture}
      \vfill {\small (a) The list sizes of $|L_{F_8}(v)|$  \ \ \ \ \ \ \ \ \ \ \ \ \ \ \ \ \ \ }  
\end{center}
\par
\begin{center}
\begin{tikzpicture}
[u/.style={fill=black, minimum size =3pt,ellipse,inner sep=1pt},node distance=1.5cm,scale=1.5]
\node[u] (v1) at (80:1){};
\node[u] (v2) at (40:1){};
\node[u] (v3) at (0:1){};
\node[u] (v4) at (320:1){};
\node[u] (v5) at (280:1){};
\node[u] (v6) at (240:1){};
\node[u] (v7) at (200:1){};
\node[u] (v8) at (160:1){};
\node[u] (v9) at (120:1){};
\node[u] (v10) at (60:1.5){};
\node[u] (v11) at (300:1.5){};
\node (C9) at (0, 0){$C_9$};

\draw   (0,0) circle[radius=1cm];
  \draw (v1) -- (v10);
  \draw (v2) -- (v10);  
  \draw (v4) -- (v11);
  \draw (v5) -- (v11);

\node[u] (w6) at (240:1.4){};
\node[u] (w7) at (200:1.4){};
\node[u] (w9) at (120:1.4){};
  \draw (v6) -- (w6); 
  \draw (v7) -- (w7); 
 \draw (v9) -- (w9);   
    
   \node[above=0.05cm, font=\small] at (v1) {$v_1$};  
   \node[right=0.1cm,font=\small] at (v2) {$v_2$};  
   \node[right=0.05cm, font=\small] at (v3) {$v_3$};     
   \node[right=0.1cm,font=\small] at (v4) {$v_4$};     
    \node[below=0.1cm,  font=\small] at (v5) {$v_5$};    
   \node[left=0.1cm,font=\small] at (v6) {$v_6$};      
   \node[above=0.2cm, left=0.05cm, font=\small] at (v7) {$v_7$};  
   \node[left=0.05cm, font=\small] at (v8) {$v_8$};    
    \node[left=0.05cm,font=\small] at (v9) {$v_9$};  
   \node[above=0.1cm,  font=\small] at (v10) {$v_{10}$};        
    \node[below=0.1cm, font=\small] at (v11) {$v_{11}$};

 \node[below=0.05cm,  font=\scriptsize] at (v1) {};
\node[below=0.05cm, left=0.1cm, font=\scriptsize] at (v2) {\textcolor{blue}{2}};  
\node[left=0.05cm,font=\scriptsize] at (v3) {\textcolor{blue}{4}};  
\node[above=0.05cm, left=0.1cm,font=\scriptsize] at (v4) {\textcolor{blue}{4}};  
\node[above=0.05cm,font=\scriptsize] at (v5) {\textcolor{blue}{2}};  
\node[above=0.05cm, font=\scriptsize] at (v6) {};  
\node[right=0.05cm, font=\scriptsize] at (v7) {};  
\node[right=0.05cm, font=\scriptsize] at (v8) {};  
\node[below=0.05cm, font=\scriptsize] at (v9) {};  
\node[ below=0.05cm, right=0.1cm, font=\scriptsize] at (v10) {};  
\node[right=0.05cm, font=\scriptsize] at (v11) {\textcolor{blue}{2}};  
 \end{tikzpicture}
      \vfill {\small (b) The list sizes after coloring $v_1$ and $v_i$: $i= 6,\dots, 10$} 

\end{center}
\par
\begin{center}
\begin{tikzpicture}
[u/.style={fill=black, minimum size =3pt,ellipse,inner sep=1pt},node distance=1.5cm,scale=1.5]
\node[u] (v1) at (80:1){};
\node[u] (v2) at (40:1){};
\node[u] (v3) at (0:1){};
\node[u] (v4) at (320:1){};
\node[u] (v5) at (280:1){};
\node[u] (v6) at (240:1){};
\node[u] (v7) at (200:1){};
\node[u] (v8) at (160:1){};
\node[u] (v9) at (120:1){};
\node[u] (v10) at (60:1.5){};
\node[u] (v11) at (300:1.5){};
\node (C9) at (0, 0){$C_9$};

\draw   (0,0) circle[radius=1cm];
  \draw (v1) -- (v10);
  \draw (v2) -- (v10);  
  \draw (v4) -- (v11);
  \draw (v5) -- (v11);

\node[u] (w6) at (240:1.4){};
\node[u] (w7) at (200:1.4){};
\node[u] (w9) at (120:1.4){};
  \draw (v6) -- (w6); 
  \draw (v7) -- (w7); 
 \draw (v9) -- (w9);   
    
   \node[above=0.05cm, font=\small] at (v1) {$v_1$};  
   \node[right=0.1cm,font=\small] at (v2) {$v_2$};  
   \node[right=0.05cm, font=\small] at (v3) {$v_3$};     
   \node[right=0.1cm,font=\small] at (v4) {$v_4$};     
    \node[below=0.1cm,  font=\small] at (v5) {$v_5$};    
   \node[left=0.1cm,font=\small] at (v6) {$v_6$};      
   \node[above=0.2cm, left=0.05cm, font=\small] at (v7) {$v_7$};  
   \node[left=0.05cm, font=\small] at (v8) {$v_8$};    
    \node[left=0.05cm,font=\small] at (v9) {$v_9$};  
   \node[above=0.1cm,  font=\small] at (v10) {$v_{10}$};        
    \node[below=0.1cm, font=\small] at (v11) {$v_{11}$};

 \node[below=0.05cm,  font=\scriptsize] at (v1) {};
\node[below=0.05cm, left=0.1cm, font=\scriptsize] at (v2) {\textcolor{blue}{2}};  
\node[left=0.05cm,font=\scriptsize] at (v3) {\textcolor{blue}{3}};  
\node[above=0.05cm, left=0.1cm,font=\scriptsize] at (v4) {\textcolor{blue}{3}};  
\node[above=0.05cm,font=\scriptsize] at (v5) {\textcolor{blue}{2}};  
\node[above=0.05cm, font=\scriptsize] at (v6) {};  
\node[right=0.05cm, font=\scriptsize] at (v7) {};  
\node[right=0.05cm, font=\scriptsize] at (v8) {};  
\node[below=0.05cm, font=\scriptsize] at (v9) {};  
\node[ below=0.05cm, right=0.1cm, font=\scriptsize] at (v10) {};  
\node[right=0.05cm, font=\scriptsize] at (v11) {};  
 \end{tikzpicture}
        \vfill {\small (c) The list sizes after coloring $v_1$ and $v_i$: $i =6,\dots, 11$} 
\end{center}
\end{multicols} 
\caption{The graph $F_8$: The list sizes after coloring some vertices of $F_8$.}
\label{F8-color}
\end{figure*}
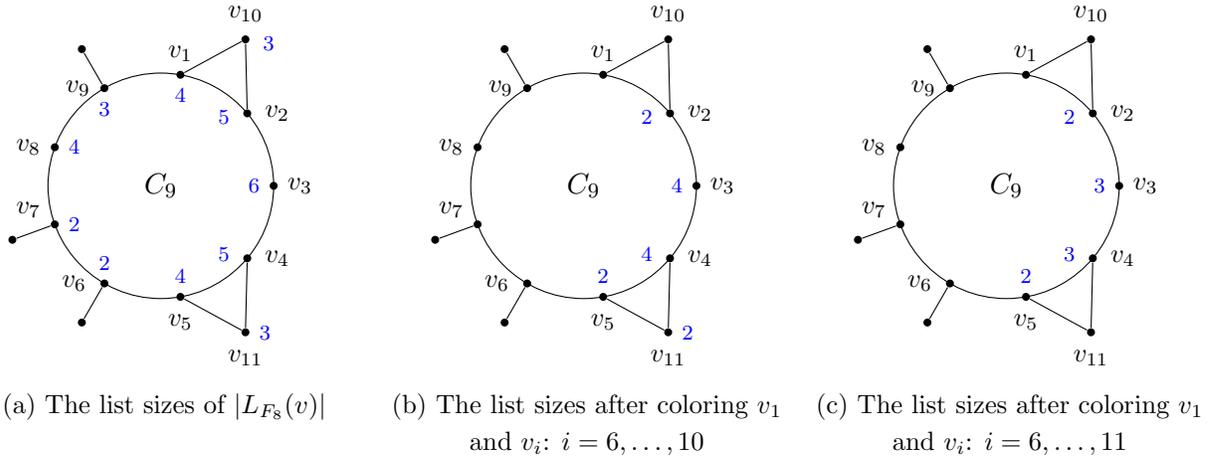

\item \textbf{Case $F_i = F_{11}$}.

As indicated in Figure~\ref{W11-color}(a), we have 
\[
|L_{F_{11}}(v_i)| \ge 
\begin{cases}
3, & i = 7, 9, 10, 11, 12,\\
4, & i = 1, 6, 8,\\
5, & i = 2, 3, 4, 5.
\end{cases}
\]

Since $|L_{F_{11}}(v_5)| \ge 5$, $|L_{F_{11}}(v_6)| \ge 4$, and  
$|L_{F_{11}}(v_{12})| \ge 3$, there exist two distinct colors  
\[
c_5 \in L_{F_{11}}(v_5), \qquad c_6 \in L_{F_{11}}(v_6),
\]
such that 
\[
|L_{F_{11}}(v_{12}) \setminus \{c_5,c_6\}| \ge 3.
\]
Color $v_5$ and $v_6$ with $c_5$ and $c_6$, respectively, and then color  
$v_7$ with some $c_7 \in L_{F_{11}}(v_7) \setminus \{c_5,c_6\}$.

Let $L_{F_{11}}'(v_i)$ be the updated lists after coloring  
$v_5$, $v_6$, and $v_7$ (see Figure~\ref{W11-color}(b)).

By an argument similar to that used in Case~$F_1$ (in the choice of $c_7$ and $c_8$),  
there exist two distinct colors  
\[
c_8 \in L_{F_{11}}(v_8), \qquad 
c_9 \in L_{F_{11}}(v_9),
\]
such that
\[
\bigl| L_{F_{11}}(v_1)\setminus\{c_8,c_9\} 
\;\cup\; L_{F_{11}}(v_{10})\setminus\{c_9\} \bigr| \ge 3.
\]
Color $v_8$ and $v_9$ with $c_8$ and $c_9$, respectively.

Let $L_{F_{11}}''(v_i)$ be the updated lists after coloring  
$v_8$ and $v_9$ (see Figure~\ref{W11-color}(c)).  
Then
\[
|L_{F_{11}}''(v_1)| \ge 2,\quad
|L_{F_{11}}''(v_2)| \ge 4,\quad
|L_{F_{11}}''(v_3)| \ge 4,\quad
|L_{F_{11}}''(v_4)| \ge 3,\quad
|L_{F_{11}}''(v_{10})| \ge 2,\quad
|L_{F_{11}}''(v_{11})| \ge 2.
\]
Moreover,
\[
|L_{F_{11}}''(v_1) \cup L_{F_{11}}''(v_{10})| \ge 3.
\]

We may now color $v_1, v_2, v_3, v_4, v_{10}, v_{11}$ by  
Lemma~\ref{lem-J2}(1), and then color $v_{12}$ to obtain an $L$-coloring of $G^2$,  
a contradiction.  
Thus, $F_{11}$ is not a subgraph of $G$.  
This completes the proof of Claim~\ref{claim-F-one}. \qed\end{enumerate}

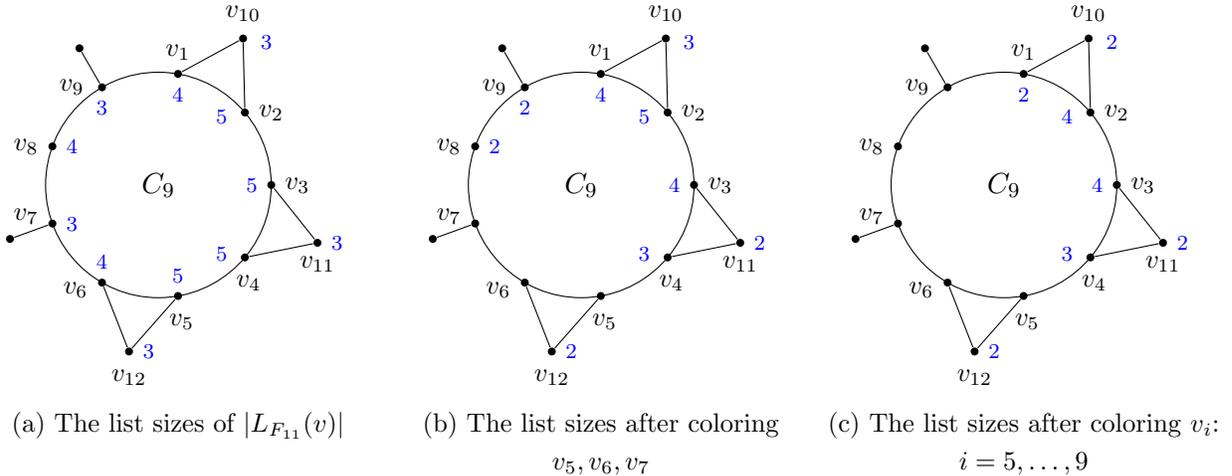
\begin{figure*}[htbp]
\begin{multicols}{3}
\begin{center}
\begin{tikzpicture}
[u/.style={fill=black, minimum size =3pt,ellipse,inner sep=1pt},node distance=1.5cm,scale=1.5]
\node[u] (v1) at (80:1){};
\node[u] (v2) at (40:1){};
\node[u] (v3) at (0:1){};
\node[u] (v4) at (320:1){};
\node[u] (v5) at (280:1){};
\node[u] (v6) at (240:1){};
\node[u] (v7) at (200:1){};
\node[u] (v8) at (160:1){};
\node[u] (v9) at (120:1){};
\node[u] (v10) at (60:1.5){};
\node[u] (v11) at (340:1.5){};
\node[u] (v12) at (260:1.5){};
\node (C9) at (0, 0){$C_9$};

\draw   (0,0) circle[radius=1cm];
  \draw (v1) -- (v10);
  \draw (v2) -- (v10);  
  \draw (v3) -- (v11);
  \draw (v4) -- (v11); 

  \draw (v5) -- (v12);
  \draw (v6) -- (v12);

\node[u] (w7) at (200:1.4){};
\node[u] (w9) at (120:1.4){};
  \draw (v7) -- (w7); 
 \draw (v9) -- (w9);   
    
   \node[above=0.05cm, font=\small] at (v1) {$v_1$};  
   \node[right=0.05cm,font=\small] at (v2) {$v_2$};  
   \node[right=0.05cm, font=\small] at (v3) {$v_3$};  
   \node[right=0.05cm, below=0.1cm, font=\small] at (v4) {$v_4$};     
    \node[right=0.05cm, below=0.1cm, font=\small] at (v5) {$v_5$};    
   \node[below=0.1cm,left=0.05cm, font=\small] at (v6) {$v_6$};      
   \node[above=0.1cm, left=0.05cm, font=\small] at (v7) {$v_7$};  
   \node[left=0.05cm, font=\small] at (v8) {$v_8$};    
    \node[left=0.1cm,font=\small] at (v9) {$v_9$};  
   \node[above=0.1cm,  font=\small] at (v10) {$v_{10}$};        
    \node[below=0.05cm, font=\small] at (v11) {$v_{11}$};   
    \node[below=0.1cm,  font=\small] at (v12) {$v_{12}$};                              
    
 \node[below=0.05cm,  font=\scriptsize] at (v1) {\textcolor{blue}{4}};
\node[below=0.05cm, left=0.1cm, font=\scriptsize] at (v2) {\textcolor{blue}{5}};  
\node[left=0.05cm,font=\scriptsize] at (v3) {\textcolor{blue}{5}};  
\node[above=0.05cm, left=0.1cm,font=\scriptsize] at (v4) {\textcolor{blue}{5}};  
\node[above=0.05cm,font=\scriptsize] at (v5) {\textcolor{blue}{5}};  
\node[above=0.05cm, font=\scriptsize] at (v6) {\textcolor{blue}{4}};  
\node[right=0.05cm, font=\scriptsize] at (v7) {\textcolor{blue}{3}};  
\node[right=0.05cm, font=\scriptsize] at (v8) {\textcolor{blue}{4}};  
\node[below=0.05cm, font=\scriptsize] at (v9) {\textcolor{blue}{3}};  
\node[ below=0.05cm, right=0.1cm, font=\scriptsize] at (v10) {\textcolor{blue}{3}}; 
\node[right=0.05cm, font=\scriptsize] at (v11) {\textcolor{blue}{3}}; 
\node[right=0.05cm, font=\scriptsize] at (v12) {\textcolor{blue}{3}}; 
 \end{tikzpicture}
       \vfill {\small (a) The list sizes of $|L_{F_{11}}(v)|$ \ \ \ \ \ \ \ \ \ \ \ \ \ \ \ \ \ \ }  
\end{center}
\par
\begin{center}
\begin{tikzpicture}
[u/.style={fill=black, minimum size =3pt,ellipse,inner sep=1pt},node distance=1.5cm,scale=1.5]
\node[u] (v1) at (80:1){};
\node[u] (v2) at (40:1){};
\node[u] (v3) at (0:1){};
\node[u] (v4) at (320:1){};
\node[u] (v5) at (280:1){};
\node[u] (v6) at (240:1){};
\node[u] (v7) at (200:1){};
\node[u] (v8) at (160:1){};
\node[u] (v9) at (120:1){};
\node[u] (v10) at (60:1.5){};
\node[u] (v11) at (340:1.5){};
\node[u] (v12) at (260:1.5){};
\node (C9) at (0, 0){$C_9$};

\draw   (0,0) circle[radius=1cm];
  \draw (v1) -- (v10);
  \draw (v2) -- (v10);  
  \draw (v3) -- (v11);
  \draw (v4) -- (v11); 

  \draw (v5) -- (v12);
  \draw (v6) -- (v12);

\node[u] (w7) at (200:1.4){};
\node[u] (w9) at (120:1.4){};
  \draw (v7) -- (w7); 
 \draw (v9) -- (w9);   
    
   \node[above=0.05cm, font=\small] at (v1) {$v_1$};  
   \node[right=0.05cm,font=\small] at (v2) {$v_2$};  
   \node[right=0.05cm, font=\small] at (v3) {$v_3$};  
   \node[right=0.05cm, below=0.1cm, font=\small] at (v4) {$v_4$};     
    \node[right=0.05cm, below=0.1cm, font=\small] at (v5) {$v_5$};    
   \node[below=0.1cm,left=0.05cm, font=\small] at (v6) {$v_6$};      
   \node[above=0.1cm, left=0.05cm, font=\small] at (v7) {$v_7$};  
   \node[left=0.05cm, font=\small] at (v8) {$v_8$};    
    \node[left=0.1cm,font=\small] at (v9) {$v_9$};  
   \node[above=0.1cm,  font=\small] at (v10) {$v_{10}$};        
    \node[below=0.05cm, font=\small] at (v11) {$v_{11}$};   
    \node[below=0.1cm,  font=\small] at (v12) {$v_{12}$};                              
    
 \node[below=0.05cm,  font=\scriptsize] at (v1) {\textcolor{blue}{4}};
\node[below=0.05cm, left=0.1cm, font=\scriptsize] at (v2) {\textcolor{blue}{5}};  
\node[left=0.05cm,font=\scriptsize] at (v3) {\textcolor{blue}{4}};  
\node[above=0.05cm, left=0.1cm,font=\scriptsize] at (v4) {\textcolor{blue}{3}};  
\node[above=0.05cm,font=\scriptsize] at (v5) {};  
\node[above=0.05cm, font=\scriptsize] at (v6) {};  
\node[right=0.05cm, font=\scriptsize] at (v7) {};  
\node[right=0.05cm, font=\scriptsize] at (v8) {\textcolor{blue}{2}};  
\node[below=0.05cm, font=\scriptsize] at (v9) {\textcolor{blue}{2}};  
\node[ below=0.05cm, right=0.1cm, font=\scriptsize] at (v10) {\textcolor{blue}{3}}; 
\node[right=0.05cm, font=\scriptsize] at (v11) {\textcolor{blue}{2}}; 
\node[right=0.05cm, font=\scriptsize] at (v12) {\textcolor{blue}{2}};   
 \end{tikzpicture}
       \vfill {\small (b) The list sizes after coloring $v_5, v_6, v_7$} 
\end{center}
\par
\begin{center}
\begin{tikzpicture}
[u/.style={fill=black, minimum size =3pt,ellipse,inner sep=1pt},node distance=1.5cm,scale=1.5]
\node[u] (v1) at (80:1){};
\node[u] (v2) at (40:1){};
\node[u] (v3) at (0:1){};
\node[u] (v4) at (320:1){};
\node[u] (v5) at (280:1){};
\node[u] (v6) at (240:1){};
\node[u] (v7) at (200:1){};
\node[u] (v8) at (160:1){};
\node[u] (v9) at (120:1){};
\node[u] (v10) at (60:1.5){};
\node[u] (v11) at (340:1.5){};
\node[u] (v12) at (260:1.5){};
\node (C9) at (0, 0){$C_9$};

\draw   (0,0) circle[radius=1cm];
  \draw (v1) -- (v10);
  \draw (v2) -- (v10);  
  \draw (v3) -- (v11);
  \draw (v4) -- (v11); 

  \draw (v5) -- (v12);
  \draw (v6) -- (v12);

\node[u] (w7) at (200:1.4){};
\node[u] (w9) at (120:1.4){};
  \draw (v7) -- (w7); 
 \draw (v9) -- (w9);   
    
   \node[above=0.05cm, font=\small] at (v1) {$v_1$};  
   \node[right=0.05cm,font=\small] at (v2) {$v_2$};  
   \node[right=0.05cm, font=\small] at (v3) {$v_3$};  
   \node[right=0.05cm, below=0.1cm, font=\small] at (v4) {$v_4$};     
    \node[right=0.05cm, below=0.1cm, font=\small] at (v5) {$v_5$};    
   \node[below=0.1cm,left=0.05cm, font=\small] at (v6) {$v_6$};      
   \node[above=0.1cm, left=0.05cm, font=\small] at (v7) {$v_7$};  
   \node[left=0.05cm, font=\small] at (v8) {$v_8$};    
    \node[left=0.1cm,font=\small] at (v9) {$v_9$};  
   \node[above=0.1cm,  font=\small] at (v10) {$v_{10}$};        
    \node[below=0.05cm, font=\small] at (v11) {$v_{11}$};   
    \node[below=0.1cm,  font=\small] at (v12) {$v_{12}$};                              
    
 \node[below=0.05cm,  font=\scriptsize] at (v1) {\textcolor{blue}{2}};
\node[below=0.05cm, left=0.1cm, font=\scriptsize] at (v2) {\textcolor{blue}{4}};  
\node[left=0.05cm,font=\scriptsize] at (v3) {\textcolor{blue}{4}};  
\node[above=0.05cm, left=0.1cm,font=\scriptsize] at (v4) {\textcolor{blue}{3}};  
\node[above=0.05cm,font=\scriptsize] at (v5) {};  
\node[above=0.05cm, font=\scriptsize] at (v6) {};  
\node[right=0.05cm, font=\scriptsize] at (v7) {};  
\node[right=0.05cm, font=\scriptsize] at (v8) {}; 
\node[below=0.05cm, font=\scriptsize] at (v9) {};  
\node[ below=0.05cm, right=0.1cm, font=\scriptsize] at (v10) {\textcolor{blue}{2}}; 
\node[right=0.05cm, font=\scriptsize] at (v11) {\textcolor{blue}{2}}; 
\node[right=0.05cm, font=\scriptsize] at (v12) {\textcolor{blue}{2}};   
 \end{tikzpicture}
         \vfill {\small (c) The list sizes after coloring $v_i$: $i = 5,\dots, 9$} 
\end{center}
\end{multicols} 
\caption{The graph $F_{11}$: The list sizes after coloring some vertices of $F_{11}$.}
\label{W11-color}
\end{figure*}

\subsubsection{$G$ does not contain $F_i$ for $i \in \{2, 3, 5, 9, 10, 12\}$} \label{subsection-F-two}

In this subsection, we apply the Combinatorial Nullstellensatz to show that $F_i$ is not a subgraph of $G$ for each $i \in \{2, 3,5, 9, 10, 12\}$.

The \emph{graph polynomial} of a graph $G$ is defined by
\[
P_{G}(\bm{x})=\prod_{u\sim v,\; u<v}(x_u-x_v),
\]
where $u\sim v$ means $u$ and $v$ are adjacent, and  
$\bm{x}=(x_v)_{v\in V(G)}$ is a vector of variables indexed by vertices of $G$.

\begin{theorem}[\cite{Alon1999}]\label{cnull}
\textnormal{(Combinatorial Nullstellensatz)}  
Let $\mathbb{F}$ be a field, and let  
$f(x_1,\ldots,x_n)\in\mathbb{F}[x_1,\ldots,x_n]$.  
Suppose $\deg(f)=\sum_{i=1}^n t_i$ with each $t_i\ge 0$, and the coefficient of  
$\prod_{i=1}^n x_i^{t_i}$ is nonzero.  
If $S_1,\ldots,S_n\subseteq \mathbb{F}$ satisfy $|S_i|\ge t_i+1$ for all $i$,  
then there exist $s_i\in S_i$ such that  
$f(s_1,\ldots,s_n)\neq 0$.
\end{theorem}

The following consequence is frequently used in list coloring.

\begin{theorem} \label{Null-list}
Let $G$ be a graph with $V(G)=\{v_1,\ldots,v_n\}$ and list assignment $L$.  
If $P_G(\bm{x})$ has a monomial $\prod_{i=1}^n x_i^{t_i}$ with nonzero coefficient and  
$|L(v_i)|\ge t_i+1$ for each $i$, then $G$ is $L$-colorable.
\end{theorem}

\begin{claim} \label{claim-F-two}
For each $i \in \{2,3,5,9,10, 12\}$, the graph $G$ does not contain $F_i$ as a subgraph.
\end{claim}
\noindent {\bf \it Proof.}
Suppose to the contrary that $G$ contains $F_i$ as a subgraph for some $i \in \{2,3,5,9,10, 12\}$.  Let $G_1 = G - V(F_i)$.  

Since no vertex of $F_i$ has two neighbors outside $F_i$, for any two vertices $u,v \in V(G_1)$,  
if $d_G(u,v)\le 2$ then $d_{G_1}(u,v)\le 2$.  
Moreover, because $G$ contains no $4$- to $8$-cycles, any two $3$-vertices of $F_i$ that do not lie on the $9$-cycle have distance at least $3$ in $G$, so their colors never conflict.

By the minimality of $G$, the graph $G_1^2$ admits an $L$-coloring $\phi$.  
We will show that $\phi$ extends to $F_i$, producing an $L$-coloring of $G^2$ and contradicting that $G$ is a counterexample.

For each $v \in V(F_i)$, define
\[
L_{F_i}(v)
   = L(v)\setminus\{\phi(x) : vx \in E(G^2)\ \text{and}\ x \notin V(F_i)\}.
\]

\begin{enumerate}
\item \textbf{Case $F_i = F_2$}.

As indicated in Figure~\ref{F12-subgraph}, 
\[
|L_{F_2}(v_i)| \ge 
\begin{cases}
3, & i = 1,3,5,8,10, \\
4, & i = 2,4,6,7,9.
\end{cases}
\]

The graph polynomial of $F_2^2$ is
\begin{eqnarray*}
P_{F_2^2}(\bm{x})
&=& 
 (x_1 -x_2)(x_1 - x_3)(x_1 - x_8)(x_1 - x_9)(x_2 - x_3)(x_2 - x_4) \\
&& (x_2 - x_9)(x_3- x_4)(x_3 - x_5)(x_4 - x_5)(x_4 - x_6)(x_5-x_6)   \\
&&(x_5 -x_7)(x_5 -x_{10})(x_6 - x_7)(x_6 -x_8)(x_6 -x_{10})(x_7-x_8)\\
&&(x_7-x_9)(x_7-x_{10})(x_8-x_9)(x_8-x_{10}).
\end{eqnarray*}

A Mathematica computation shows that the coefficient of  
\[
x_1^2 x_2^3 x_3^2 x_4^3 x_5^2 x_6^3 x_7^2 x_8^2 x_9^2 x_{10}
\]
is $2\neq 0$.  
Thus, by Theorem~\ref{Null-list}, $\phi$ extends to $F_2^2$, yielding an $L$-coloring of $G^2$.

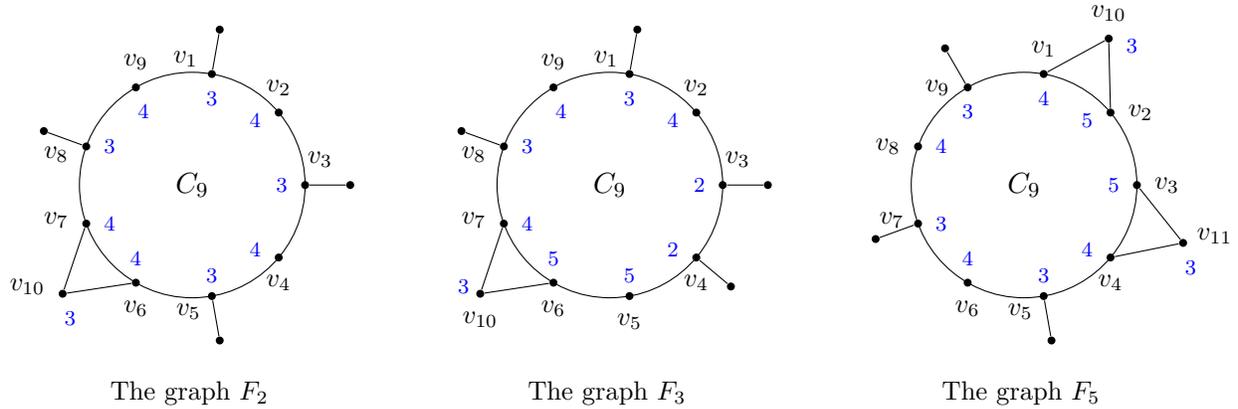
\begin{figure}[htbp]
\begin{center}
\begin{tikzpicture}
[u/.style={fill=black, minimum size =3pt,ellipse,inner sep=1pt},node distance=1.5cm,scale=1.5]
\node[u] (v1) at (80:1){};
\node[u] (v2) at (40:1){};
\node[u] (v3) at (0:1){};
\node[u] (v4) at (320:1){};
\node[u] (v5) at (280:1){};
\node[u] (v6) at (240:1){};
\node[u] (v7) at (200:1){};
\node[u] (v8) at (160:1){};
\node[u] (v9) at (120:1){};
\node[u] (v10) at (220:1.5){};
\node (C9) at (0, 0){$C_9$};

\draw   (0,0) circle[radius=1cm];
  \draw (v6) -- (v10);
  \draw (v7) -- (v10);      
  
\node[u] (w1) at (80:1.4){};
\node[u] (w3) at (0:1.4){};
\node[u] (w5) at (280:1.4){};
\node[u] (w8) at (160:1.4){};
  \draw (v1) -- (w1);
  \draw (v3) -- (w3);
  \draw (v5) -- (w5);
  \draw (v8) -- (w8); 
    
   \node[above=0.2cm, left=0.05cm, font=\small] at (v1) {$v_1$};  
   \node[above=0.1cm,font=\small] at (v2) {$v_2$};  
   \node[right=0.2cm, above= 0.1cm, font=\small] at (v3) {$v_3$};     
   \node[below=0.1cm,font=\small] at (v4) {$v_4$};     
    \node[below=0.2cm,left=0.01cm, font=\small] at (v5) {$v_5$};    
   \node[below=0.1cm,font=\small] at (v6) {$v_6$};      
   \node[above=0.05cm, left=0.1cm, font=\small] at (v7) {$v_7$};  
   \node[below=0.1cm, left=0.1cm, font=\small] at (v8) {$v_8$};    
    \node[above=0.1cm,font=\small] at (v9) {$v_9$};  
     \node[above=0.1cm, left=0.1cm, font=\small] at (v10) {$v_{10}$};          
    
 \node[below=0.1cm,  font=\scriptsize] at (v1) {\textcolor{blue}{3}};
\node[below=0.1cm, left=0.1cm, font=\scriptsize] at (v2) {\textcolor{blue}{4}};  
\node[left=0.1cm,font=\scriptsize] at (v3) {\textcolor{blue}{3}};  
\node[above=0.1cm, left=0.1cm,font=\scriptsize] at (v4) {\textcolor{blue}{4}};  
\node[above=0.05cm,font=\scriptsize] at (v5) {\textcolor{blue}{3}};  
\node[above=0.1cm, font=\scriptsize] at (v6) {\textcolor{blue}{4}};  
\node[right=0.1cm, font=\scriptsize] at (v7) {\textcolor{blue}{4}};  
\node[right=0.1cm, font=\scriptsize] at (v8) {\textcolor{blue}{3}};  
\node[right=0.1cm, below=0.1cm, font=\scriptsize] at (v9) {\textcolor{blue}{4}};  
\node[right=0.1cm, below=0.1cm, font=\scriptsize] at (v10) {\textcolor{blue}{3}}; 
\node[left=0.3cm,below=1.0cm] at (v5) {\small The graph $F_2$};
 \end{tikzpicture}
\hfill \hspace{0.1cm}
\begin{tikzpicture}
[u/.style={fill=black, minimum size =3pt,ellipse,inner sep=1pt},node distance=1.5cm,scale=1.5]
\node[u] (v1) at (80:1){};
\node[u] (v2) at (40:1){};
\node[u] (v3) at (0:1){};
\node[u] (v4) at (320:1){};
\node[u] (v5) at (280:1){};
\node[u] (v6) at (240:1){};
\node[u] (v7) at (200:1){};
\node[u] (v8) at (160:1){};
\node[u] (v9) at (120:1){};
\node[u] (v10) at (220:1.5){};
\node (C9) at (0, 0){$C_9$};

\draw   (0,0) circle[radius=1cm];
  
\node[u] (w1) at (80:1.4){};
\node[u] (w3) at (0:1.4){};
\node[u] (w4) at (320:1.4){};
\node[u] (w8) at (160:1.4){};
  \draw (v1) -- (w1);
  \draw (v3) -- (w3);
  \draw (v4) -- (w4);
  \draw (v8) -- (w8); 
  \draw (v6) -- (v10);
  \draw (v7) -- (v10);   
    
   \node[above=0.2cm, left=0.01cm, font=\small] at (v1) {$v_1$};  
   \node[above=0.1cm,font=\small] at (v2) {$v_2$};  
   \node[right=0.2cm,above= 0.1cm, font=\small] at (v3) {$v_3$};     
   \node[below=0.1cm,font=\small] at (v4) {$v_4$};     
    \node[below=0.1cm, font=\small] at (v5) {$v_5$};    
   \node[below=0.1cm,font=\small] at (v6) {$v_6$};      
   \node[above=0.05cm, left=0.1cm, font=\small] at (v7) {$v_7$};  
   \node[below=0.1cm, left=0.1cm, font=\small] at (v8) {$v_8$};    
    \node[above=0.1cm,font=\small] at (v9) {$v_9$};  
     \node[below=0.1cm,font=\small] at (v10) {$v_{10}$};            
    
 \node[below=0.1cm,  font=\scriptsize] at (v1) {\textcolor{blue}{3}};
\node[below=0.1cm, left=0.1cm, font=\scriptsize] at (v2) {\textcolor{blue}{4}};  
\node[left=0.1cm,font=\scriptsize] at (v3) {\textcolor{blue}{2}};  
\node[above=0.1cm, left=0.1cm,font=\scriptsize] at (v4) {\textcolor{blue}{2}};  
\node[above=0.05cm,font=\scriptsize] at (v5) {\textcolor{blue}{5}};  
\node[above=0.1cm, font=\scriptsize] at (v6) {\textcolor{blue}{5}};  
\node[right=0.1cm, font=\scriptsize] at (v7) {\textcolor{blue}{4}};  
\node[right=0.1cm, font=\scriptsize] at (v8) {\textcolor{blue}{3}};  
\node[right=0.1cm, below=0.1cm, font=\scriptsize] at (v9) {\textcolor{blue}{4}};  
\node[above=0.1cm, left=0.01cm, font=\scriptsize] at (v10) {\textcolor{blue}{3}}; 
\node[left=0.3cm,below=1.0cm] at (v5) {\small The graph $F_3$};
 \end{tikzpicture} 
\hfill \hspace{0.1cm}
\begin{tikzpicture}
[u/.style={fill=black, minimum size =3pt,ellipse,inner sep=1pt},node distance=1.5cm,scale=1.5]
\node[u] (v1) at (80:1){};
\node[u] (v2) at (40:1){};
\node[u] (v3) at (0:1){};
\node[u] (v4) at (320:1){};
\node[u] (v5) at (280:1){};
\node[u] (v6) at (240:1){};
\node[u] (v7) at (200:1){};
\node[u] (v8) at (160:1){};
\node[u] (v9) at (120:1){};
\node[u] (v10) at (60:1.5){};
\node[u] (v11) at (340:1.5){};
\node (C9) at (0, 0){$C_9$};

\draw   (0,0) circle[radius=1cm];
  \draw (v1) -- (v10);
  \draw (v2) -- (v10);  
  \draw (v3) -- (v11);
  \draw (v4) -- (v11);

\node[u] (w5) at (280:1.4){};
\node[u] (w7) at (200:1.4){};
\node[u] (w9) at (120:1.4){};
  \draw (v5) -- (w5);
  \draw (v7) -- (w7); 
 \draw (v9) -- (w9);   
    
   \node[above=0.1cm, font=\small] at (v1) {$v_1$};  
   \node[right=0.1cm,font=\small] at (v2) {$v_2$};  
   \node[right=0.1cm, font=\small] at (v3) {$v_3$};     
   \node[below=0.1cm,font=\small] at (v4) {$v_4$};     
    \node[below=0.2cm,left=0.01cm, font=\small] at (v5) {$v_5$};    
   \node[below=0.1cm,font=\small] at (v6) {$v_6$};      
   \node[above=0.05cm, left=0.05cm, font=\small] at (v7) {$v_7$};  
   \node[left=0.1cm, font=\small] at (v8) {$v_8$};    
    \node[left=0.1cm,font=\small] at (v9) {$v_9$};  
   \node[above=0.1cm,  font=\small] at (v10) {$v_{10}$};        
    \node[above=0.1cm, right=0.05cm, font=\small] at (v11) {$v_{11}$};               
    
 \node[below=0.1cm,  font=\scriptsize] at (v1) {\textcolor{blue}{4}};
\node[below=0.1cm, left=0.1cm, font=\scriptsize] at (v2) {\textcolor{blue}{5}};  
\node[left=0.1cm,font=\scriptsize] at (v3) {\textcolor{blue}{5}};  
\node[above=0.1cm, left=0.1cm,font=\scriptsize] at (v4) {\textcolor{blue}{4}};  
\node[above=0.05cm,font=\scriptsize] at (v5) {\textcolor{blue}{3}};  
\node[above=0.1cm, font=\scriptsize] at (v6) {\textcolor{blue}{4}};  
\node[right=0.1cm, font=\scriptsize] at (v7) {\textcolor{blue}{3}};  
\node[right=0.1cm, font=\scriptsize] at (v8) {\textcolor{blue}{4}};  
\node[below=0.1cm, font=\scriptsize] at (v9) {\textcolor{blue}{3}};  
\node[ below=0.1cm, right=0.1cm, font=\scriptsize] at (v10) {\textcolor{blue}{3}}; 
\node[right=0.1cm, below=0.1cm, font=\scriptsize] at (v11) {\textcolor{blue}{3}}; 
\node[left=0.3cm,below=1.0cm] at (v5) {\small The graph $F_5$};
 \end{tikzpicture}
\end{center} 
\caption{The graphs $F_2, F_3$ and $F_5$ and their corresponding list sizes.} 
\label{F12-subgraph}
\end{figure}


\item \textbf{Case $F_i = F_3$}.

Note that $P_{F_2^2}(\bm{x}) = P_{F_3^2}(\bm{x})$.  
As indicated in Figure~\ref{F12-subgraph},
\[
|L_{F_3}(v_i)| \ge 
\begin{cases}
2, & i = 3,4, \\
3, & i = 1,8,10, \\
4, & i = 2,7,9, \\
5, & i = 5,6.
\end{cases}
\]

A Mathematica computation shows that the coefficient of  
\[
x_1^{2}x_2^{2}x_3 x_4 x_5^{4}x_6^{4}x_7^{2}x_8^{2}x_9^{2}x_{10}^{2}
\]
is $1\neq 0$.  
Thus, by Theorem~\ref{Null-list},  
$\phi$ extends to $F_3^2$, producing an $L$-coloring of $G^2$.

\item \textbf{Case $F_i = F_5$.}
As indicated in Figure~\ref{F12-subgraph}, 
$$
|L_{F_5}(v_i)| \geq 
\begin{cases}
3, & i = 5, 7, 9,   10, 11, \\
4, & i = 1,  4, 6, 8, \\
5, & i = 2, 3.
\end{cases}
$$

The graph polynomial of $F_5^2$ is

\begin{eqnarray*} \label{poly-F1}
P_{F_5^2}(\bm{x})
&=& (x_1 -x_2)(x_1 - x_3)(x_1 - x_8) (x_1 - x_9)(x_1 - x_{10})(x_2 - x_3)(x_2 - x_4) \\
&& (x_2 - x_9)(x_2-x_{10})(x_2 - x_{11}) (x_3- x_4)(x_3 - x_5)(x_3-x_{10})(x_3 - x_{11}) \\
&&(x_4 - x_5)(x_4 - x_6) (x_4 -x_{11}) (x_5-x_6) (x_5 -x_7) (x_5 -x_{11})(x_6 - x_7)  \\
&&(x_6 -x_8)(x_7-x_8)(x_7-x_9) (x_8-x_9)(x_9-x_{10})
\end{eqnarray*}

A Mathematica computation shows that 
the coefficient of 
$ x_1^3x_2^4x_3^4x_4^3x_5^1x_6^3x_7^2x_8^3x_9^1x_{10}^1x_{11}^1$
is $-2 \not = 0$. 
Thus, by Theorem~\ref{Null-list}, $\phi$ extends to $F_5^2$, yielding an $L$-coloring of $G^2$.

\item \textbf{Case $F_i = F_9$}.

As indicated in  Figure \ref{F45-subgraph}, 
$$
|L_{F_9}(v_i)| \geq 
\begin{cases}
3, & i = 6,  9,   10, 11, 12,  \\
4, & i = 1,  5, 7, 8,  \\
5, & i = 2, 4, \\
6, & i = 3. \\
\end{cases}
$$
The graph polynomial for $F_9^2$ is 
\begin{eqnarray*} \label{poly-F1}
P_{F_9^2}(\bm{x})
&=& 
 (x_1 -x_2)(x_1 - x_3)(x_1 - x_8) (x_1 - x_9)(x_1 - x_{10})(x_2 - x_3)(x_2 - x_4) \\
&& (x_2 - x_9)(x_2-x_{10})(x_3- x_4)(x_3 - x_5)(x_3-x_{10})(x_3 - x_{11}) \\
&& (x_4 - x_5)(x_4 - x_6) (x_4 -x_{11}) (x_5-x_6) (x_5 -x_7) (x_5 -x_{11})   \\
&& (x_6 - x_7)(x_6 -x_8)(x_6 -x_{11})(x_6 -x_{12})(x_7-x_8)(x_7-x_9) \\
&& (x_7-x_{12})(x_8-x_9)(x_8-x_{12})(x_9-x_{10})(x_9-x_{12})
\end{eqnarray*}
A Mathematica computation shows that 
the coefficient of 
$ x_1^3x_2^4x_3^3x_4^4x_5^3x_6^2x_7^3x_8^3x_9^2x_{10}^1x_{11}^1x_{12}^1$
is $4 \not = 0$. Thus, by Theorem~\ref{Null-list}, $\phi$ extends to $F_9^2$, yielding an $L$-coloring of $G^2$.

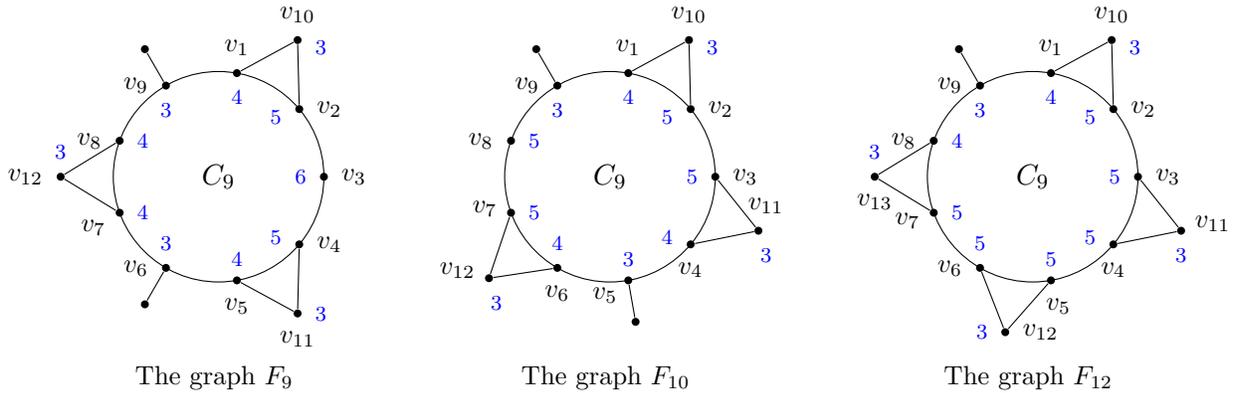
\begin{figure}[htbp]
\begin{center}
\begin{tikzpicture}
[u/.style={fill=black, minimum size =3pt,ellipse,inner sep=1pt},node distance=1.5cm,scale=1.4]
\node[u] (v1) at (80:1){};
\node[u] (v2) at (40:1){};
\node[u] (v3) at (0:1){};
\node[u] (v4) at (320:1){};
\node[u] (v5) at (280:1){};
\node[u] (v6) at (240:1){};
\node[u] (v7) at (200:1){};
\node[u] (v8) at (160:1){};
\node[u] (v9) at (120:1){};
\node[u] (v10) at (60:1.5){};
\node[u] (v11) at (300:1.5){};
\node[u] (v12) at (180:1.5){};
\node (C9) at (0, 0){$C_9$};

\draw   (0,0) circle[radius=1cm];
  \draw (v1) -- (v10);
  \draw (v2) -- (v10);  
  \draw (v4) -- (v11);
  \draw (v5) -- (v11); 

  \draw (v7) -- (v12);
  \draw (v8) -- (v12);  
  
\node[u] (w6) at (240:1.4){};
\node[u] (w9) at (120:1.4){};
  \draw (v6) -- (w6); 
 \draw (v9) -- (w9);   
    
   \node[above=0.1cm, font=\small] at (v1) {$v_1$};  
   \node[right=0.1cm,font=\small] at (v2) {$v_2$};  
   \node[right=0.1cm, font=\small] at (v3) {$v_3$};     
   \node[right=0.1cm,font=\small] at (v4) {$v_4$};     
    \node[below=0.1cm,  font=\small] at (v5) {$v_5$};    
   \node[left=0.1cm,font=\small] at (v6) {$v_6$};      
   \node[below=0.2cm, left=0.05cm, font=\small] at (v7) {$v_7$};  
   \node[left=0.1cm, font=\small] at (v8) {$v_8$};    
    \node[left=0.1cm,font=\small] at (v9) {$v_9$};  
   \node[above=0.1cm,  font=\small] at (v10) {$v_{10}$};        
    \node[below=0.1cm, font=\small] at (v11) {$v_{11}$};   
    \node[left=0.1cm,  font=\small] at (v12) {$v_{12}$};

 \node[below=0.1cm,  font=\scriptsize] at (v1) {\textcolor{blue}{4}};
\node[below=0.1cm, left=0.1cm, font=\scriptsize] at (v2) {\textcolor{blue}{5}};  
\node[left=0.1cm,font=\scriptsize] at (v3) {\textcolor{blue}{6}};  
\node[above=0.1cm, left=0.1cm,font=\scriptsize] at (v4) {\textcolor{blue}{5}};  
\node[above=0.05cm,font=\scriptsize] at (v5) {\textcolor{blue}{4}};  
\node[above=0.1cm, font=\scriptsize] at (v6) {\textcolor{blue}{3}};  
\node[right=0.1cm, font=\scriptsize] at (v7) {\textcolor{blue}{4}};  
\node[right=0.1cm, font=\scriptsize] at (v8) {\textcolor{blue}{4}};  
\node[below=0.1cm, font=\scriptsize] at (v9) {\textcolor{blue}{3}};  
\node[ below=0.1cm, right=0.1cm, font=\scriptsize] at (v10) {\textcolor{blue}{3}}; 
\node[right=0.1cm, font=\scriptsize] at (v11) {\textcolor{blue}{3}}; 
\node[above=0.1cm, font=\scriptsize] at (v12) {\textcolor{blue}{3}}; 
\node[left=0.3cm, below=1.0cm] at (v5) {\small The graph $F_9$};
 \end{tikzpicture}
\hfill 
\hspace{0.1cm} 
\begin{tikzpicture}
[u/.style={fill=black, minimum size =3pt,ellipse,inner sep=1pt},node distance=1.5cm,scale=1.4]
\node[u] (v1) at (80:1){};
\node[u] (v2) at (40:1){};
\node[u] (v3) at (0:1){};
\node[u] (v4) at (320:1){};
\node[u] (v5) at (280:1){};
\node[u] (v6) at (240:1){};
\node[u] (v7) at (200:1){};
\node[u] (v8) at (160:1){};
\node[u] (v9) at (120:1){};
\node[u] (v10) at (60:1.5){};
\node[u] (v11) at (340:1.5){};
\node[u] (v12) at (220:1.5){};
\node (C9) at (0, 0){$C_9$};

\draw   (0,0) circle[radius=1cm];
  \draw (v1) -- (v10);
  \draw (v2) -- (v10);  
  \draw (v3) -- (v11);
  \draw (v4) -- (v11);  
  \draw (v6) -- (v12);
  \draw (v7) -- (v12);

\node[u] (w5) at (280:1.4){};
\node[u] (w9) at (120:1.4){};
  \draw (v5) -- (w5);
 \draw (v9) -- (w9);   
    
   \node[above=0.1cm, font=\small] at (v1) {$v_1$};  
   \node[right=0.1cm,font=\small] at (v2) {$v_2$};  
   \node[right=0.1cm, font=\small] at (v3) {$v_3$};     
   \node[below=0.1cm,font=\small] at (v4) {$v_4$};     
    \node[below=0.2cm,left=0.01cm, font=\small] at (v5) {$v_5$};    
   \node[below=0.1cm,font=\small] at (v6) {$v_6$};      
   \node[above=0.05cm, left=0.05cm, font=\small] at (v7) {$v_7$};  
   \node[left=0.1cm, font=\small] at (v8) {$v_8$};    
    \node[left=0.1cm,font=\small] at (v9) {$v_9$};  
   \node[above=0.1cm,  font=\small] at (v10) {$v_{10}$};        
    \node[right=0.1cm, above=0.1cm, font=\small] at (v11) {$v_{11}$};    
     \node[above=0.1cm, left=0.05cm, font=\small] at (v12) {$v_{12}$};                  
    
 \node[below=0.1cm,  font=\scriptsize] at (v1) {\textcolor{blue}{4}};
\node[below=0.1cm, left=0.1cm, font=\scriptsize] at (v2) {\textcolor{blue}{5}};  
\node[left=0.1cm,font=\scriptsize] at (v3) {\textcolor{blue}{5}};  
\node[above=0.1cm, left=0.1cm,font=\scriptsize] at (v4) {\textcolor{blue}{4}};  
\node[above=0.05cm,font=\scriptsize] at (v5) {\textcolor{blue}{3}};  
\node[above=0.1cm, font=\scriptsize] at (v6) {\textcolor{blue}{4}};  
\node[right=0.1cm, font=\scriptsize] at (v7) {\textcolor{blue}{5}};  
\node[right=0.1cm, font=\scriptsize] at (v8) {\textcolor{blue}{5}};  
\node[below=0.1cm, font=\scriptsize] at (v9) {\textcolor{blue}{3}};  
\node[ below=0.1cm, right=0.1cm, font=\scriptsize] at (v10) {\textcolor{blue}{3}}; 
\node[right=0.1cm, below=0.1cm, font=\scriptsize] at (v11) {\textcolor{blue}{3}}; 
\node[right=0.1cm, below=0.1cm, font=\scriptsize] at (v12) {\textcolor{blue}{3}}; 
\node[left=0.3cm,below=1.0cm] at (v5) {\small The graph $F_{10}$};
 \end{tikzpicture}
\hfill 
\hspace{0.1cm}
\begin{tikzpicture}
[u/.style={fill=black, minimum size =3pt,ellipse,inner sep=1pt},node distance=1.5cm,scale=1.4]
\node[u] (v1) at (80:1){};
\node[u] (v2) at (40:1){};
\node[u] (v3) at (0:1){};
\node[u] (v4) at (320:1){};
\node[u] (v5) at (280:1){};
\node[u] (v6) at (240:1){};
\node[u] (v7) at (200:1){};
\node[u] (v8) at (160:1){};
\node[u] (v9) at (120:1){};
\node[u] (v10) at (60:1.5){};
\node[u] (v11) at (340:1.5){};
\node[u] (v12) at (260:1.5){};
\node[u] (v13) at (180:1.5){};
\node (C9) at (0, 0){$C_9$};

\draw   (0,0) circle[radius=1cm];
  \draw (v1) -- (v10);
  \draw (v2) -- (v10);  
  \draw (v3) -- (v11);
  \draw (v4) -- (v11); 

  \draw (v5) -- (v12);
  \draw (v6) -- (v12);  
  \draw (v7) -- (v13);
  \draw (v8) -- (v13);     
        
\node[u] (w9) at (120:1.4){};
 \draw (v9) -- (w9);   
    
   \node[above=0.1cm, font=\small] at (v1) {$v_1$};  
   \node[right=0.1cm,font=\small] at (v2) {$v_2$};  
   \node[right=0.1cm, font=\small] at (v3) {$v_3$};     
   \node[below=0.1cm,font=\small] at (v4) {$v_4$};     
    \node[right=0.1cm, below=0.1cm,  font=\small] at (v5) {$v_5$};    
   \node[left=0.1cm,font=\small] at (v6) {$v_6$};      
   \node[below=0.1cm, left=0.05cm, font=\small] at (v7) {$v_7$};  
   \node[left=0.1cm, font=\small] at (v8) {$v_8$};    
    \node[left=0.1cm,font=\small] at (v9) {$v_9$};  
   \node[above=0.1cm,  font=\small] at (v10) {$v_{10}$};        
    \node[above=0.1cm, right=0.05cm, font=\small] at (v11) {$v_{11}$};   
    \node[right=0.1cm,  font=\small] at (v12) {$v_{12}$};    
     \node[below=0.1cm, font=\small] at (v13) {$v_{13}$};                             
    
 \node[below=0.1cm,  font=\scriptsize] at (v1) {\textcolor{blue}{4}};
\node[below=0.1cm, left=0.1cm, font=\scriptsize] at (v2) {\textcolor{blue}{5}};  
\node[left=0.1cm,font=\scriptsize] at (v3) {\textcolor{blue}{5}};  
\node[above=0.1cm, left=0.1cm,font=\scriptsize] at (v4) {\textcolor{blue}{5}};  
\node[above=0.05cm,font=\scriptsize] at (v5) {\textcolor{blue}{5}};  
\node[above=0.1cm, font=\scriptsize] at (v6) {\textcolor{blue}{5}};  
\node[right=0.1cm, font=\scriptsize] at (v7) {\textcolor{blue}{5}};  
\node[right=0.1cm, font=\scriptsize] at (v8) {\textcolor{blue}{4}};  
\node[below=0.1cm, font=\scriptsize] at (v9) {\textcolor{blue}{3}};  
\node[ below=0.1cm, right=0.1cm, font=\scriptsize] at (v10) {\textcolor{blue}{3}}; 
\node[below=0.1cm, font=\scriptsize] at (v11) {\textcolor{blue}{3}}; 
\node[left=0.1cm, font=\scriptsize] at (v12) {\textcolor{blue}{3}}; 
\node[above=0.1cm, font=\scriptsize] at (v13) {\textcolor{blue}{3}}; 
\node[left=0.3cm, below=1.0cm] at (v5) {\small The graph $F_{12}$};
 \end{tikzpicture}
\end{center} 
\caption{The graphs $F_9, F_{10}$ and $F_{12}$ and their corresponding list sizes.} 
\label{F45-subgraph}
\end{figure}

\item \textbf{Case $F_i = F_{10}$}.

As indicated in  Figure \ref{F45-subgraph},
$$
|L_{F_{10}}(v_i)| \geq 
\begin{cases}
3, & i = 5, 9,  10, 11, 12,  \\
4, & i = 1,  4, 6, \\
5, & i = 2, 3, 7, 8.
\end{cases}
$$

The graph polynomial for $F_{10}^2$ is 
\begin{eqnarray*} \label{poly-F1}
P_{F_{10}^2}(\bm{x})
&=& 
 (x_1 -x_2)(x_1 - x_3)(x_1 - x_8) (x_1 - x_9)(x_1 - x_{10})(x_2 - x_3)(x_2 - x_4) \\
&& (x_2 - x_9)(x_2-x_{10})(x_2 - x_{11}) (x_3- x_4)(x_3 - x_5)(x_3-x_{10})(x_3 - x_{11}) \\
&&(x_4 - x_5)(x_4 - x_6) (x_4 -x_{11}) (x_5-x_6) (x_5 -x_7) (x_5 -x_{11})  (x_5 - x_{12}) \\
&&  (x_6 - x_7)(x_6 -x_8)(x_6-x_{12})(x_7-x_8)(x_7-x_9)  (x_7-x_{12})(x_8-x_9)\\
&&(x_8-x_{12})(x_9-x_{10})
\end{eqnarray*}

A Mathematica computation shows that the coefficient of 
$x_1^3x_2^4x_3^3x_4^3x_5^2x_6^3x_7^3x_8^4x_9^1x_{10}^1x_{11}^2x_{12}^1$
is $2 \not = 0$. Thus, by Theorem~\ref{Null-list}, $\phi$ extends to $F_{10}^2$, yielding an $L$-coloring of $G^2$.

\item \textbf{Case $F_i = F_{12}$.}

As indicated in in Figure \ref{F45-subgraph}, 
$$
|L_{F_{12}}(v_i)| \geq 
\begin{cases}
3, & i = 9,  10, 11, 12, 13,   \\
4, & i = 1,  8, \\
5, & i = 2, 3, 4, 5, 6, 7. 
\end{cases}
$$

The graph polynomial for $F_{12}^2$ is
\begin{eqnarray*} \label{poly-F1}
P_{F_{12}^2}(\bm{x})
&=& 
 (x_1 -x_2)(x_1 - x_3)(x_1 - x_8) (x_1 - x_9)(x_1 - x_{10})(x_2 - x_3)(x_2 - x_4) \\
&& (x_2 - x_9)(x_2-x_{10})(x_2 - x_{11}) (x_3- x_4)(x_3 - x_5) 
(x_3-x_{10})(x_3 - x_{11}) \\
&&(x_4 - x_5)(x_4 - x_6) (x_4 -x_{11}) (x_4-x_{12}) (x_5-x_6) (x_5 -x_7) (x_5 -x_{11}) \\
&& (x_5 - x_{12}) (x_6 - x_7)(x_6 -x_8)(x_6-x_{12})(x_6-x_{13})(x_7-x_8)(x_7-x_9) \\
&& (x_7-x_{12})(x_7-x_{13})(x_8-x_9)(x_8-x_{13})(x_9-x_{10})(x_9-x_{13}) 
\end{eqnarray*}

A Mathematica computation shows that the coefficient of   
$x_1^3x_2^3x_3^4x_4^4x_5^3x_6^3x_7^4x_8^3x_9^1x_{10}^1x_{11}^1x_{12}^2x_{13}^2$
is $-1 \not = 0$.
Thus, by Theorem~\ref{Null-list}, $\phi$ extends to $F_{12}^2$, yielding an $L$-coloring of $G^2$.
This completes the proof of Claim \ref{claim-F-two}.  \qed
\end{enumerate}

\section*{Acknowledgments}
Seog-Jin Kim was supported by the National Research Foundation of Korea(NRF) grant funded by the Korea government(MSIT)(No.NRF-2021R1A2C1005785).  Rong Luo was supported by a grant from  Simons Foundation (No. 839830).

\end{document}